\documentclass[11pt]{amsart}

\usepackage {amsmath, amssymb, amscd, mathrsfs,  amsthm, stmaryrd,  bbm, diagbox, enumerate, slashed, graphicx, color, subfig, comment, enumitem}
\usepackage[all, cmtip]{xy}
\usepackage[text={6.1in,8.6in},centering,letterpaper,dvips]{geometry}
\usepackage{url, hyperref}
\usepackage{amscd}
\usepackage{mathtools} 
\usepackage{amstext} 
\usepackage{array}  
\usepackage{xspace}

\newtheorem{thm}{Theorem}[section]
\newtheorem{theorem}[thm]{Theorem}

\newtheorem{corollary}[thm]{Corollary}
\newtheorem{lem}[thm]{Lemma}
\newtheorem{lemma}[thm]{Lemma}

\newtheorem{proposition}[thm]{Proposition}

\theoremstyle{definition}

\newtheorem{definition}[thm]{Definition}

\theoremstyle{remark}

\newtheorem{convention}[thm]{Convention}
\newtheorem{remark}[thm]{Remark}

\newtheorem{example}[thm]{Example}

\newcommand{\R}{\ensuremath{\mathbb{R}}}
\newcommand{\C}{\ensuremath{\mathbb{C}}}

\newcommand{\Z}{\ensuremath{\mathbb{Z}}}
\newcommand{\s}{\mathfrak{s}}
\newcommand{\del}{\partial}
\def\gr{\operatorname{gr}}
\def\grHF{\operatorname{gr}_{\mathrm{HF}}}

\def\x{\mathbf x}
\def\y{\mathbf y}
\def\z{\mathbf z}
\newcommand{\Spinc}{\text{Spin}^{\text{c}}}
\newcommand{\CF}{\mathit{CF}}
\newcommand{\HF}{\mathit{HF}}
\def\prel{\operatorname{prel}}
\newcommand{\HFhat}{\widehat{\HF}}
\newcommand{\HFp}{{\HF}^+}
\newcommand{\HFm}{{\HF}^-}
\newcommand{\HFi}{{\HF}^\infty}
\newcommand{\CFhat}{\widehat{\CF}}

\newcommand{\CFm}{{\CF}^-}

\newcommand{\CFcirc}{{\CF}^\circ}
\newcommand{\HFhattw}{\widehat{\HFtw}}
\newcommand{\HFptw}{{\HFtw}^+}
\newcommand{\HFmtw}{{\HFtw}^-}
\newcommand{\HFitw}{{\HFtw}^\infty}
\newcommand{\HFhatprel}{\widehat{\HF}_{\!\!\prel}}
\newcommand{\HFpprel}{{\HF}^+_{\!\!\prel}}
\newcommand{\HFmprel}{{\HF}^-_{\!\!\prel}}
\newcommand{\HFiprel}{{\HF}^\infty_{\!\!\prel}}
\newcommand{\HFcircprel}{{\HF}^\circ_{\!\!\prel}}
\newcommand{\CFhatprel}{\widehat{\CF}_{\!\!\prel}}
\newcommand{\CFpprel}{{\CF}^+_{\!\!\prel}}
\newcommand{\CFmprel}{{\CF}^-_{\!\!\prel}}
\newcommand{\CFiprel}{{\CF}^\infty_{\!\!\prel}}
\newcommand{\CFcircprel}{{\CF}^\circ_{\!\!\prel}}
\newcommand{\bunderline}[1]{\underline{#1\mkern-2mu}\mkern2mu }

\def\HFtw{\bunderline{\HF}}
\def\CFtw{\bunderline{\CF}}
\def\CFtwcan{\bunderline{\CF}^{\operatorname{can}}}
\def\ftw{\bunderline{f}}
\def\Hs{\mathbb{H}}
\def\ccdot {\! \cdot \!}
\newcommand{\HFcirc}{{\HF}^\circ}

\newcommand{\HFLhat}{\widehat{\mathit{HFL}}}
\newcommand{\HFLm}{\mathit{HFL}^-}
\newcommand{\CFLcurved}{\mathcal{CF\!L}^-}
\def\M {\mathcal {M}}
\def\Mhat {\widehat{\M}}
\def\cN{\mathcal{N}}
\def\cNhat{\widehat{\cN}}

\newcommand{\Gr}{\mathit{Gr}}
\def\He{\mathcal{H}}
\def\T{\mathbb{T}}
\newcommand{\Ta}{\mathbb{T}_{\alpha}}
\newcommand{\Tb}{\mathbb{T}_{\beta}}
\newcommand{\Tg}{\mathbb{T}_{\gamma}}
\newcommand{\Td}{\mathbb{T}_{\delta}}
\newcommand{\Te}{\mathbb{T}_{\epsilon}}
\newcommand{\Tz}{\mathbb{T}_{\zeta}}
\newcommand{\Tap}{\mathbb{T}_{\alpha'}}
\newcommand{\Tbp}{\mathbb{T}_{\beta'}}
\newcommand{\Tgp}{\mathbb{T}_{\gamma'}}
\newcommand{\Tdp}{\mathbb{T}_{\delta'}}

\def\vbeta{\vec{\beta}}
\def\delbar{\bar{\partial}}
\def\alphas{\boldsymbol\alpha}
\def\betas{\boldsymbol\beta}
\def\gammas{\boldsymbol\gamma}
\def\deltas{\boldsymbol\delta}
\def\epsilons{\boldsymbol\epsilon}
\def\zetas{\boldsymbol\zeta}

\def\aa{\mathbf{a}}
\def\bb{\mathbf{b}}
\def\Path{\mathbf{P}}
\def\t{\mathfrak{t}}
\def\Pin{\operatorname{Pin}}
\def\Spin{\operatorname{Spin}}
\def\tSpin{\widetilde{\Spin}}
\def\On{\operatorname{O}(n)}

\def\G{\mathcal{G}}
\def\Gr{\operatorname{Gr}}
\def\Gstab{\mathcal{G}_{\operatorname{stab}}}
\def\Gdiff{\mathcal{G}_{\operatorname{diff}}}

\def\AA{\mathbb{A}}
\def\BB{\mathbb{B}}

\def\can{\operatorname{can}}
\def\Ps{P^\#}
\def\Pa{\Ps_{\alpha}}
\def\Pb{\Ps_{\beta}}
\def\Pg{\Ps_{\gamma}}
\def\Pd{\Ps_{\delta}}
\def\Rs{R^\#}
\def\Ra{\Rs_{\alpha}}
\def\Rb{\Rs_{\beta}}
\def\Cat{\mathcal{C}}
\def\Scan{S^\#_{\operatorname{can}}}
\def\H{\mathcal{H}}
\def\bH{\overline{\mathcal{H}}}
\def\Set{\mathbf{S}}
\def\bSet{\underline{\Set}}
\def\bs{\underline{\s}}
\def\new{\operatorname{new}}
\def\O{\operatorname{O}}
\def\GL{\operatorname{GL}}
\def\SO{\operatorname{SO}}

\def\B{\operatorname{B}\!}
\def\Grd{\operatorname{Gr}^{\dagger}}
\def\Grp{\operatorname{Gr}^{\#}}
\def\Lambdad{\Lambda^{\dagger}}
\def\Lambdap{\Lambda^{\#}}
\def\RP{\mathbb{RP}}
\def\SFH{\mathit{SFH}}
\def\HFIcirc {\mathit{HFI}^\circ}
\def\F{\mathbb{F}}
\def\Link{\mathbb{L}}
\def\Pathhat{\widehat{\Path}}
\def\index{\operatorname{index}}
\def\ltop{\lambda^{\operatorname{top}}}
\def\Cl{\mathcal{C}\ell}
\def\ws{\mathbf{w}}
\def\zs{\mathbf{z}}
\def\bs{\underline{\s}}
\def\Span{\operatorname{Span}}
\def\Int{\operatorname{Int}}
\DeclareMathOperator{\Sym}{Sym}
\DeclareMathOperator{\coker}{coker}

\DeclareMathOperator{\id}{id}
\DeclareMathOperator{\ind}{ind}

\DeclareMathOperator{\Aut}{Aut}
\DeclareMathOperator{\Hom}{Hom}

\DeclareMathOperator{\Fr}{Fr}

\DeclareMathOperator{\Map}{Map}

\DeclareFontFamily{U}{matha}{\hyphenchar\font45}
\DeclareFontShape{U}{matha}{m}{n}{
      <5> <6> <7> <8> <9> <10> gen * matha
      <10.95> matha10 <12> <14.4> <17.28> <20.74> <24.88> matha12
      }{}
\DeclareSymbolFont{matha}{U}{matha}{m}{n}
\DeclareFontSubstitution{U}{matha}{m}{n}
\DeclareMathSymbol{\abxcup}{\mathbin}{matha}{'131}

\title[Canonical orientations in Heegaard Floer theory]{Canonical orientations in Heegaard Floer theory}

\author[Mohammed Abouzaid]{Mohammed Abouzaid}
\address{Department of Mathematics, Stanford University\\
450 Jane Stanford Way, Building 380, Stanford, CA 94305-2125, USA.}
\email{abouzaid@stanford.edu}

\author[Ciprian Manolescu]{Ciprian Manolescu}
\address{Department of Mathematics, Stanford University\\
450 Jane Stanford Way, Building 380, Stanford, CA 94305-2125, USA.}
\email{cm5@stanford.edu}

\begin{document}

\begin{abstract}
We set up Heegaard Floer theory over the integers, using canonical orientations coming from coupled Spin structures on the Lagrangian tori. We prove naturality of Heegaard Floer homology, sutured Floer homology, and link Floer homology over $\Z$. We give a new proof of the surgery exact triangle in this context, as well as a definition of involutive Heegaard Floer homology over $\Z$.
\end{abstract}
\maketitle

\section{Introduction}
In a series of papers, such as \cite{HolDisk}, \cite{HolDiskTwo}, \cite{HolDiskFour}, \cite{OS-knots}, Ozsv\'ath and Szab\'o developed Heegaard Floer theory: a collection of invariants of $3$-manifolds, $4$-manifolds, and knots. Since then, the theory has become an important tool in low-dimensional topology, leading to numerous applications. 

In the original papers, the $3$-manifold and knot invariants were defined as homology groups with coefficients in $\Z$, and the $4$-manifold invariants were integers defined up to sign. However, at some point the community started ignoring signs, and it became customary to work over the coefficient field $\mathbb{F}_2$. Notably, naturality of the Heegaard Floer invariants \cite{JTZ} and their full functoriality properties under cobordisms \cite{ZemkeHF, ZemkeHFK} were only established over $\mathbb{F}_2$. This is sufficient for many applications, but it does constitute a limitation for others. By contrast, monopole Floer homology \cite{KMBook} was defined over $\Z$, and the Seiberg-Witten invariants \cite{Witten} are integers (with the sign determined by a homology orientation on the $4$-manifold). 

In Lagrangian Floer homology, there are several ways to choose signs for the counts of moduli spaces. One way is by {\em coherent orientations} \cite{FloerHofer}, where one trivializes the determinant line bundles over sufficiently many homotopy classes of disks arbitrarily, and then trivializes them over the rest of the homotopy classes in the unique compatible way. Another way, which is more popular in the recent literature, is through {\em canonical orientations} induced from Spin or Pin structures on the Lagrangians; see \cite{DeSilva, FOOO2, SeidelBook}. Yet another way is by using  twisted derived local systems on the Lagrangians \cite{Rezchikov}.

In their papers on Heegaard Floer homology, Ozsv\'ath and Szab\'o used coherent orientations to pin down the signs. The purpose of our paper is to offer a different perspective on Heegaard Floer theory over $\Z$, by using canonical orientations instead. These are better suited for questions about naturality and functoriality. We establish the naturality properties here and thus place the theory over $\Z$ on a solid footing. 

The canonical orientations on moduli spaces are constructed as follows. Ideally, we would like to fix Spin structures on the Lagrangian tori $\Ta$ and $\Tb$ that appear in Heegaard Floer homology. However, a Spin structure requires in particular an orientation, and in our case the Lagrangians do not have natural orientations. Rather, they have a {\em coupled orientation} (that is, an orientation on their product), which  has the same origin as the absolute $\Z/2$-grading in Heegaard Floer homology defined in \cite[Section 10.4]{HolDiskTwo}. 

Instead of Spin structures, one can settle for Pin structures on the Lagrangians; these do not need a background orientation. We choose Lie group Pin structures, that is, structures invariant under the torus action. Up to homotopy, there is only one Lie group Pin structure on the torus, and this suffices to determine the isomorphism class of the Heegaard Floer groups over $\Z$. Nevertheless, specifying the choice only up to homotopy is not enough for discussing naturality. The space of Lie group Pin structures has the homotopy type of $\RP^\infty$, with $\pi_1 =\Z/2$, and  the monodromy on Floer homology around a non-trivial loop in this space is multiplication by $-1$.

To rectify this problem, note that a Lie group Pin structure on the Lagrangian torus $\Ta$ is specified by the Pin structure in the tangent space at each point. In turn, this tangent space can be identified with $\AA =H^1(U_{\alpha}; \R)$, where $U_{\alpha}$ is the alpha handlebody in the Heegaard splitting. Similarly, the tangent space to $\Tb$ can be identified with $\BB=H^1(U_{\beta}; \R)$, where $U_{\beta}$ is the beta handlebody. While there are no canonical Pin structures on $\AA$ and $\BB$, it turns out that there is a canonical {\em coupled Spin  structure} on the pair $(\AA, \BB)$. Coupled Spin structures are a new concept that we introduce in this paper. They make sense for a pair of parallelizable Lagrangians (such as our tori), and make use of an underlying coupled orientation (which exists in our setting). We will see that coupled Spin structures produce canonical orientations on the moduli spaces of $J$-holomorphic strips, and thus give Lagrangian Floer groups that are well-defined up to canonical isomorphism. In our context, we use them to define Heegaard Floer complexes over $\Z$.

With these choices, we establish the invariance of Heegaard Floer homology under the usual Heegaard moves from \cite{HolDisk}, as well as its naturality, following the proof scheme in \cite{JTZ}:

\begin{theorem}
\label{thm:main}
Let $Y$ be a closed, oriented $3$-manifold equipped with a basepoint $z \in Y$ and a $\Spinc$ structure $\s$. The Heegaard Floer homologies $\HFhat$,  $\HFp$, $\HFm$ and $\HFi$ (defined using the canonical coupled Spin structure) are invariants of the triple $(Y, \s, z)$, well-defined up to natural isomorphism in the category of $\Z[U]$-modules. 
\end{theorem}

The canonical orientations on moduli spaces that we construct produce in particular a coherent orientation system as in \cite[Definition 3.11]{HolDisk}. Thus, our Heegaard Floer groups are isomorphic to those constructed by Ozsv\'ath-Szab\'o in \cite{HolDisk}, for some (unspecified) orientation system. Up to equivalence, there are $2^{b_1(Y)}$ orientation systems. In \cite[Theorem 10.12]{HolDiskTwo}, Ozsv\'ath and Szab\'o further select a special orientation system, using their calculation of a twisted version $\HFitw$. We expect that their system is the same as the one arising from our work, but do not prove this here.

We mention one important way in which the proof of naturality in Theorem~\ref{thm:main} departs from 
the arguments in \cite{JTZ}. Juh\'asz, Thurston, and Zemke used maps associated to equivalences between the curve systems on the Heegaard diagram. (An equivalence is replacing such a system with any other that represents the same handlebody. We can think of it as a composition of isotopies and handleslides.) Over $\mathbb{F}_2$, it is relatively easy to see that the maps induced by equivalences commute. Over $\Z$, it is harder; one can prove this up to a sign, as in \cite{Gartner}. What we do instead is to use maps associated to isotopies and handleslides. We then have to check certain relations associated to loops of handleslides, using the description of those loops in \cite[Appendix A]{JTZ}. 

With our framework in place, we give a new proof of the surgery exact triangle from \cite[Section 9]{HolDiskTwo}, with $\Z$ coefficients. (See Theorem~\ref{thm:exact}.) Interestingly, the natural maps induced by counting holomorphic triangles (using the coupled Spin structures) are not the ones that appear in the triangle. One needs to use twisted coefficients for the maps, and we show that one can do so without changing the three Heegaard Floer groups themselves.

Another application of Theorem~\ref{thm:main} is a definition of involutive Heegaard Floer homology over $\Z$. Over $\F_2$, this theory was defined in \cite{HFI}, by taking the homology of the mapping cone of an involution $\iota$ on Heegaard Floer complexes. The map $\iota$ involves moves on Heegaard diagrams and thus relies on naturality. Since we now have this property over $\Z$, the same construction can be done with integer coefficients. (See Section~\ref{sec:HFI} below.)

In a different direction, we prove analogues of Theorem~\ref{thm:main} for sutured Floer homology and for link Floer homology. Sutured Floer homology is an invariant of balanced sutured manifolds introduced by Juh\'asz \cite{Juhasz}. Roughly, a balanced sutured manifold is a compact $3$-manifold $M$ with boundary, whose boundary $\del M$ is split along a collection of sutures $\gamma$ into two parts $R_+$ and $R_-$. The balanced condition says that $R_+$ and $R_-$ have the same Euler characteristic. The construction of sutured Floer homology proceeds along similar lines to that of Heegaard Floer homology, but we no longer have a canonical coupled Spin structure. Instead, we are forced to choose one, and this is equivalent to choosing a coupled Spin structure $S$ on the pair of vector spaces $(H_1(M,R_-; \R), H_1(M,R_+; \R))$. We call this data a {\em homological coupled Spin structure}. From this we can define a sutured Floer homology group $\SFH(M,\gamma, \s, S)$ over $\Z$.

\begin{theorem}
\label{thm:main_sutured}
Let $(M, \gamma)$ be a sutured manifold equipped with a $\Spinc$ structure $\s$ and a homological coupled Spin structure $S$. Then, the sutured Floer homology $\SFH(M,\gamma, \s, S)$ is a natural invariant of the quadruple $(M, \gamma, \s, S)$.
\end{theorem}

Finally, we consider link Floer homology, an invariant of links in three-manifolds introduced by Ozsv\'ath and Szab\'o in \cite{OS-knots}, \cite{Links}. The original construction was over $\Z$ for knots (see \cite{OS-knots}, \cite{RasmussenThesis}, but only over $\F_2$ for links, due to the presence of disk bubbles. For links in $S^3$, a refinement over $\Z$ was constructed using grid homology in \cite{MOST}. For links in arbitrary $3$-manifolds, Sarkar \cite{SarkarSigns} explained the existence of several possible theories over $\Z$ (depending on some choices). In our set-up, we find a canonical coupled Spin structure on the Lagrangian tori, and we use it to identify a single preferred theory over $\Z$. This theory comes in several flavors, such as $\HFLhat$ or $\HFLm$.

\begin{theorem}
\label{thm:main_links}
Let $\Link=(L, \ws, \zs)$ be a multi-based oriented link in a three-manifold $Y$, and $\bs$ a $\Spinc$ structure on $Y$ relative to $L$. Then, the link Floer homologies $\HFLhat(Y,\Link,\bs)$ and $\HFLm(Y, \Link,\bs)$ are natural invariants of the triple $(Y, \Link, \bs)$.
\end{theorem}

In the proofs of the naturality statements in Theorems~\ref{thm:main_sutured} and \ref{thm:main_links}, we again need to consider loops of handleslides. We will make use of the work of Qin \cite{Qin}, who recently generalized the results of \cite[Appendix A]{JTZ} to these contexts. 

In future work, we will use canonical orientations to prove that Heegaard Floer homology and link Floer homology are functorial over $\Z$ under cobordisms, refining  the results of Ozsv\'ath-Szab\'o \cite{HolDiskFour} and Zemke \cite{ZemkeHF}, \cite{ZemkeHFK}. The cobordism maps are determined by homology orientations, in a manner reminiscent of what happens in monopole Floer homology \cite{KMBook}.

\medskip
\textbf{Organization of the paper.} In Section~\ref{sec:pinspin} we discuss Spin, Pin, and coupled Spin structures on vector bundles. In Section~\ref{sec:general} we explain the general construction of canonical orientations from Pin structures, in the context of Lagrangian Floer theory. In Section~\ref{sec:3m} we define Heegaard Floer homology over $\Z$ and prove that its isomorphism class is a $3$-manifold invariant. In Section~\ref{sec:naturality} we prove  Theorem~\ref{thm:main} in its full strength, by establishing naturality. In Section~\ref{sec:exact} we prove the surgery exact triangle. In Section~\ref{sec:other} we discuss signs in a few other settings, and in particular prove Theorems~\ref{thm:main_sutured} and Theorem~\ref{thm:main_links}.

\medskip
\textbf{Acknowledgements.} We thank Irving Dai, Andr\'as Juh\'asz, Tye Lidman, Robert Lipshitz, Maggie Miller, Tom Mrowka, Qianhe Qin, Sucharit Sarkar, Eha Srivastava, and Ian Zemke for helpful conversations. 
MA was partially supported by NSF award DMS-2103805. CM was partially supported by a Simons Investigator Award and the Simons Collaboration Grant on New Structures in Low-Dimensional Topology.

%%% Local Variables:
%%% mode: LaTeX
%%% TeX-master: "signs"
%%% End:

\section{Spin, Pin, and coupled Spin}
\label{sec:pinspin}

In this section we review some generalities about Spin and Pin structures, and then introduce the new notion of a coupled Spin structure.

\subsection{Pin structures}
\label{sec:pins}
We start by listing some relevant facts about Pin structures, comparing them to the better-known Spin structures. For more details, we refer to \cite{KirbyTaylorPin} and \cite{SeidelBook}.

Throughout the paper let $\Pin(n)$ denote the {\em positive} Pin group in dimension $n$. This is the central extension of $\On$ by $\Z/2$ determined (as a group extension) by the element $w_2 \in H^2(\B\On; \Z/2)$. It can be constructed explicitly as follows. Let $\Cl(n)$ be the real Clifford algebra on $n$ generators $e_1, \dots, e_n$, subject to the relations $e_i^2=1$ for all $i$, and $e_i e_j = - e_j e_i$ for all $i\neq j$. Then, $\Pin(n)$ is the subset of $\Cl(n)$ consisting of elements of the form $v_1v_2 \cdots v_k$ where $v_i$ are unit vectors in $\R^n = \Span(e_1, \dots, e_n).$ The subgroup $\Spin(n) \subset \Pin(n)$ consists of those $v_1v_2 \cdots v_k$ with $k$ even, and the double cover map $\Pin(n) \to O(n)$ is given by composing reflections across the hyperplanes $v_i^\perp$.

We obtain a commutative diagram of Lie group homomorphisms
$$\begin{xymatrix}{
1 \ar[r] & \Z/2 \ar[r] \ar[d]^{\id} & \Spin(n) \ar[d] \ar[r] & \operatorname{SO}(n) \ar[d] \ar[r] &1\\
1 \ar[r] & \Z/2 \ar[r] & \Pin(n) \ar[r] & \On \ar[r] &1}
\end{xymatrix}
$$

\begin{remark}
The group $\Pin(n)$ is denoted $\Pin^+_n$ in \cite{KirbyTaylorPin}. It is to be distinguished from the {\em negative} Pin group, denoted $\Pin^-_n$ in \cite{KirbyTaylorPin}, which is the central extension of $\On$ by $\Z/2$ classified by $w_2+w_1^2 \in H^2(\B\On; \Z/2)$. For example, for $n=1$ we have 
$$\Pin^+_1 \cong \Z/2 \times \Z/2, \ \ \ \Pin^-_1 \cong \Z/4,$$
so that the surjection $\Pin(1) \to \O(1)$ splits.

For $n=2$, we have $\Pin^+_2 \cong \operatorname{O}(2)$, whereas $\Pin^-_2$ is the group appearing in gauge theory from the symmetries of the Seiberg-Witten equations \cite{Furuta, beta}.
\end{remark}

From now on, whenever we talk about a base $B$ for a vector bundle, we will assume that $B$ is paracompact.

\begin{definition}
  Let $E \to B$ a real vector bundle of rank $n$ with an inner product, and $\Fr(E) \to B$ the principal $\On$-bundle of orthonormal frames in $E$. A {\em Pin structure} on $E$ is defined to be a lift of $\Fr(E)$ to a principal $\Pin(n)$-bundle. In particular, if $M$ is a smooth manifold, a Pin structure on $M$ is a Pin structure on the tangent bundle $TM$.
  \end{definition}
  
   \begin{definition}
An {\em isomorphism} between two Pin structures on a bundle $E \to B$ is an isomorphism of principal $\Pin(n)$-bundles commuting with the projections to $\Fr(E)$.
  \end{definition}
  
  If $E$ is oriented, then the choice of a Pin structure on $E$ is equivalent to that of a Spin structure. However, Pin structures can be defined without choosing orientations, and they can exist on non-orientable bundles.

\begin{proposition} $($\cite[p.185-186]{KirbyTaylorPin}$) $ 
A vector bundle $E \to B$ admits a Pin structure if and only if $w_2(E)=0$. If one exists, the set of Pin structures on $E$ up to isomorphism forms an affine space over $H^1(B; \Z/2)$.
\end{proposition}

Thus, given a Pin structure $\Ps$ on $E$ and an element $\eta \in H^1(B; \Z/2)$, we can twist $\Ps$ by $\ell$ and obtain a new Pin structure on $E$ (up to isomorphism), denoted $\Ps \otimes \eta$.

A  concept closely related to isomorphism is that of homotopy. 
   
  \begin{definition}
 A {\em homotopy} between two Pin structures $P_0^\#$, $P_1^\#$ on a bundle $E \to B$ consists of a Pin structure on the pullback of $E$ to $B \times [0,1]$, whose restriction to the two ends agrees with $P_0^\#$ and $P_1^\#$, respectively. Two homotopies are {\em equivalent} if the corresponding Pin structure over the product of $B$ with the boundary of $[0,1] \times [0,1]$ extends to a Pin structure on the interior.
\end{definition}

\begin{remark}
\label{rem:classifyingPin}
 Alternatively, one may formulate these notions as follows:  A Pin structure on a vector bundle $E$ over $B$ is a lift of a representative classifying map $B \to \B\O(n)$ to $ \B\Pin(n)$. A homotopy of Pin structures is given by a homotopy between the lifts (through other such lifts), and the notion of equivalence corresponds to a homotopy of homotopies. 
 \end{remark}
 
 The classification of Pin structures up to homotopy is the same as up to isomorphism. 
 \begin{lemma}
 \label{lem:homotopic}
 Two Pin structures on a vector bundle $E \to B$ are isomorphic if and only if they are homotopic.
 \end{lemma}
 \begin{proof}
 If the two Pin structures are homotopic, then after choosing a Pin connection in the bundle over $B \times [0,1]$ we get an isomorphism by parallel transport. Conversely, if we have an isomorphism, we can construct a Pin bundle over $B \times [0,1]$ by gluing two trivial bundles over $B \times [0,2/3)$ and $B \times (1/3, 1]$ using that isomorphism.
\end{proof}

The advantage of talking about homotopies is that we can iterate them. We can talk about homotopies of homotopies (what we called equivalences), and also about higher homotopies. Using this data we can construct a simplicial set whose vertices are Pin structures on $E$, and whose $n$-simplices correspond to $n$-homotopies. We call this the {\em space of Pin structures on $E$}.
 
\begin{remark}
\label{rem:noinner}
Up to homotopy equivalence, the space of Pin structures on a vector bundle $E$ does not depend on the inner product on $E$. For simplicity, we will suppress the inner products from our future discussions.
\end{remark}

\begin{proposition}
\label{prop:pinspace}
Let $E \to B$ be a vector bundle that admits a Pin structure. Then the space of such structures is (non-canonically) homotopy equivalent to the mapping space $\Map(B, \RP^\infty)$.
\end{proposition}

\begin{proof}
We think of Pin structures in terms of maps to the classifying space, as in Remark~\ref{rem:classifyingPin}. The fibration
$$ \Map(B, \text{B}\Z/2) \to \Map(B, \B\Pin(n)) \to \Map(B, \B\O(n))$$
gives the desired result, using the fact that $\text{B}\Z/2 = \RP^{\infty}$.
\end{proof}

 The above result has various generalizations one of which we will use: given a point $b \in B$, the space of Pin structures on $E$ which are fixed at $b$ is homotopy equivalent to the space of based maps from $B$ to $\B\O(n)$.

\begin{example}
\label{ex:S1Pin}
Let $M$ be the circle $S^1$, equipped with either orientation. It is well-known that $M$ admits two isomorphism classes of Spin structures: 
\begin{itemize}
\item the {\em bounding} one, which is obtained by restricting a Spin structure on $D^2$, and hence represents the zero element in the Spin bordism group $\Omega^1_{\Spin}\cong \Z/2$. In this case the oriented frame bundle of $TS^1$ is just $S^1$ itself, and this Spin structure corresponds to its nontrivial (connected) double cover;
\item the {\em Lie group} one, which is obtained by choosing a Spin structure on the tangent space to $S^1$ at a single point, and extending it to the whole circle in an $S^1$-equivariant way. This represents the nontrivial element in $\Omega^1_{\Spin}$, and corresponds to the trivial double cover of $S^1$.
\end{itemize}
See for example \cite[p.35-36]{KirbyBook}.

We will use the same terminology ({\em bounding} and {\em Lie group}) for the corresponding isomorphism classes of Pin structures on the circle, which are independent of the orientation. Note, however, that this is a slight misnomer: the Pin bordism group $\Omega^1_{\Pin}$ is trivial, because a Lie group Pin structure also bounds (a M\"obius band equipped with its own Pin structure).
\end{example}

\begin{definition}
\label{def:S1Pin}
Fix a splitting of the extension $\Pin(1) \to \O(1)$. This determines a canonical Pin structure on any one-dimensional vector space, and hence (by equivariant extension) a canonical representative of the Lie group Pin structures on the circle. When we are interested in a Pin structure on $S^1$ on the nose (not just up to isomorphism), this is what we call the {\em Lie group Pin structure}.
\end{definition}

By contrast, one can show that the action of rotation on the space of Pin structures on $D^2$ generates a non-trivial loop, so that there is no natural choice of bounding Pin structure on $S^1$.

\begin{remark}
In Seidel's book \cite{SeidelBook}, a Lie group Pin structure on the circle is called trivial, and a bounding one nontrivial. We will not use this terminology here, to prevent confusion---since either kind of structure can be considered trivial from a different perspective: that of bordism, or that of equivariant trivializations.
\end{remark}

\begin{example} \label{rem:non-equivalent_iso}
 Let $V$ be a vector space, viewed as a vector bundle over a point.  Then, there are non-equivalent homotopies between Pin structures on $V$. The simplest case to consider is when the two Pin structures that we are comparing are the same; in that case, such isomorphisms correspond bijectively to Pin structures on the circle which are fixed at a point, as can be seen by gluing the endpoints of the interval. In fact, by Proposition~\ref{prop:pinspace}, the space of Pin structures on $V$ is (non-canonically) homotopy equivalent to $\text{B}\Z/2 = \RP^{\infty}$, and $\pi_1(\RP^\infty)=\Z/2$.
\end{example}

\begin{example}
\label{ex:TnPin}
On a torus $T^n$ we will be interested in the {\em Lie group Pin structures}, those that are equivariant with respect to the Lie group multiplication on $T^n$. The space of such structures is the same as that of Pin structures on the tangent space at any point, and thus homotopy equivalent to $\RP^{\infty}$. Unlike in the case $n=1$ considered in Example~\ref{ex:S1Pin} and Definition~\ref{def:S1Pin}, for $n \geq 2$ we do not have a canonical representative anymore. 
\end{example}

\begin{remark}
\label{rem:productpin}
One disadvantage of Pin (as opposed to Spin) structures is that they do not behave well with respect to direct sums. Indeed, suppose that $E$ and $E'$ are two vector bundles over the same base $B$. Then
$$ w_2(E \oplus E') = w_2(E) + w_2(E') + w_1(E)w_1(E').$$
Thus, if $E$ and $E'$ are Spin (have $w_1=w_2=0$) then so is $E \oplus E'$, and in fact we get an induced Spin structure on $E \oplus E'$. By contrast, if $E$ and $E'$ are Pin (have $w_2=0$), this does not guarantee that $E \oplus E'$ is Pin.

The source of this issue is the lack of a natural multiplication map $\Pin(n) \times \Pin(n') \to \Pin(n+n')$. If we use the explicit description of the Pin groups as subsets of the Clifford algebras, we could try to define the map by 
$$(v_1 \dots v_k,  v'_1 \dots v'_{k'}) \mapsto v_1 \dots v_k v'_1 \dots v'_{k'}.$$ 
However, the left hand side is equal to
$$(1, v'_1 \dots v'_{k'}) \cdot (v_1 \dots v_k, 1) \in \Pin(n) \times \Pin(n')$$
which should be mapped to $v'_1 \dots v'_{k'}v_1 \dots v_k \in \Pin(n+n')$. Thus, we run into the problem that the elements $v_1 \dots v_k$ and $v'_1 \dots v'_{k'}$ commute in $\Pin(n+n')$ only up to a sign $(-1)^{kk'}$. On the other hand, there is a natural multiplication $ \Spin(n) \times \Spin(n') \to \Spin(n+n')$, since in that case we only work with even values of $k$ and $k'$.
\end{remark}

\subsection{Coupled orientations}
\label{sec:co}
The notion of coupled orientation that we introduce here will play a role in our take on Lagrangian Floer homology in Section~\ref{sec:LFH}. We will then come back to it in Section~\ref{sec:definition}.

\begin{definition}
For $n, m \geq 0$, we define the {\em coupled special orthogonal group}\footnote{Warning: the similar notation $SO(n, m)$ is  often used in the literature to denote a different group, that of transformations that preserve an indefinite bilinear form of signature $(n, m)$ and have determinant one.} $\SO(n; m)$ to be the group of pairs of orthogonal matrices with the same determinant:
$$ \SO(n; m) = \{(A, B) \in \O(n) \times \O(m) \mid \det(A) = \det(B) \}.$$
\end{definition}

\begin{definition}
\label{def:coupled-o}
Let $E$, $F$ be two real vector bundles (with inner products) over the same base $B$, of ranks $n$ and $m$ respectively.  A {\em coupled orientation} on $(E, F)$ is a lift of $\Fr(E) \times \Fr(F) \to B$ from a principal $\O(n) \times \O(m)$ bundle to a principal $\SO(n, m)$ bundle. 
\end{definition}

\begin{lemma}
\label{lem:coupled-o}
A coupled orientation on $(E, F)$ is equivalent to an orientation on the direct sum $E \oplus F$. It exists if and only if $w_1(E) = w_1(F)$.
\end{lemma}

\begin{proof}
For the first part, observe that there is a pull-back diagram
\[
\xymatrix{
\SO(n; m) \, \ar[d] \ar@{^{(}->}[r] &\SO(n+m)\ar[d] \\
\O(n) \times \O(m) \, \ar@{^{(}->}[r] &\O(n+m)
}
\]
Thus, lifting a principal $\O(n) \times \O(m)$ bundle to $\SO(n; m)$ is the same as lifting the induced $\O(n+m)$ bundle to $\SO(n+m)$. The latter is the data of an orientation on the direct sum.

The second part follows from the relation $w_1(E \oplus F) = w_1(E) + w_1(F)$.
\end{proof}

\begin{definition}
\label{lem:EFEFo}
Let $E, F, E', F'$ be four real vector bundles over a base $B$. Given coupled orientations on $(E, F)$ and $(E', F')$, we can form their {\em direct sum}. This is the coupled orientation on $(E \oplus E', F \oplus F')$  induced by the map
\begin{align}
\label{eq:dsum}
 \SO(n; m) \times \SO(n', m') &\to \SO(n+n'; m+m'), \\ ((A, B), (C, D)) &\mapsto \left( \left(\begin{array}{c|c} A & 0 \\ \hline 0 & C \end{array} \right), \left(\begin{array}{c|c} B & 0 \\ \hline 0 & D \end{array} \right) \right).\notag
 \end{align}
\end{definition}

%Alternatively, in terms of the description from Lemma~\ref{lem:coupled-o}, orientations on $E \oplus F$ and $E' \oplus F'$ give one on 
%\begin{equation}
%\label{eq:efe}
%(E \oplus F) \oplus (E' \oplus F') \cong (E \oplus E') \oplus (F \oplus F'),
%\end{equation}
%where the isomorphism interchanges the second and third summands.

\begin{definition}
\label{def:coupledoo}
For any vector bundle $E$, there is a {\em canonical coupled orientation} on $(E, E)$, induced from the natural map 
\begin{equation}
\label{eq:ee}
\O(n) \to \SO(n; n), \ \ A \mapsto (A, A).
\end{equation}
\end{definition}

\begin{lemma}
Let $E$ and $E'$ be two vector bundles over the same base $B$. Equip $(E, E)$ and $(E', E')$ with their canonical coupled orientations, then take their direct sum. The result is the canonical coupled orientation on $(E \oplus E', E \oplus E')$.
\end{lemma}

\begin{proof}
This is a consequence of the commutativity of the diagram
\[
\xymatrix{
\O(n) \times \O(n') \, \ar[d] \ar[r] &\SO(n; n) \times \SO(n'; n')\ar[d] \\
\O(n+n') \, \ar[r] &\SO(n+n'; n+n')
}
\]
which involves maps of the form \eqref{eq:dsum} and \eqref{eq:ee}.
\end{proof}

Observe that the equivalence in Lemma~\ref{lem:coupled-o} is based on the map
\begin{equation}
\label{eq:OO}
\O(n) \times \O(m) \to \O(n+m), \ \ (A, B) \mapsto \left(\begin{array}{c|c} A & 0 \\ \hline 0 & B \end{array} \right).
\end{equation}
Under this equivalence, the canonical coupled orientation on $(E, E)$ corresponds to the orientation on $E \oplus E$ given by an ordered basis of the form 
$$((v_1,0), \dots, (v_n,0), (0, v_1), \dots, (0,v_n)),$$ where $\{v_1, \dots, v_n\}$ is any basis of $E$. Let us call this the {\em concatenated orientation} on $E \oplus E$. 

There is another natural choice of orientation on $E \oplus E$, from bases of the form 
$$((v_1,0), (0, v_1), \dots, (v_n, 0), (0, v_n)).$$
We call this the {\em shuffled orientation}. It differs from the concatenated one by $(-1)^{n(n-1)/2}$. The shuffled orientation behaves better with respect to direct sums, in the following sense: Given shuffled orientations on $E \oplus E$ and $E' \oplus E'$, their direct sum is an orientation on $(E \oplus E) \oplus (E' \oplus E')$ which corresponds to the shuffled orientation on $(E \oplus E') \oplus (E \oplus E')$ under the isomorphism that interchanges the second and third summands. 

To make everything be compatible with direct sums, it is desirable to have the canonical coupled orientation on $(E, E)$ correspond to the shuffled orientation on $E \oplus E$. We can achieve this by adjusting the map \eqref{eq:OO} in the particular case where $n = m$. In that case we consider the map
\begin{equation}
\label{eq:OOC}
\O(n) \times \O(n) \to \O(2n), \ \ (A, B) \mapsto C^{-1}\left(\begin{array}{c|c} A & 0 \\ \hline 0 & B \end{array} \right) C,
\end{equation}
where $C$ is the permutation matrix taking the standard basis $(e_1, e_2, \dots, e_{2n-1}, e_{2n})$ to the basis $(e_1, e_{n+1},  e_2, e_{n+2}, \dots, e_{n}, e_{2n})$.

\begin{convention}
\label{conv0}
In this paper we will only use coupled orientations on pairs $(E, F)$, where $E$ and $F$ have the same rank. When doing so, to go from a coupled orientation on $(E, F)$ to an orientation on $E\oplus F$ we will use the map \eqref{eq:OOC} instead of \eqref{eq:OO}. Therefore, the canonical coupled orientation on a pair $(E, E)$ will correspond to the shuffled orientation on $E \oplus E$.
\end{convention}

So far we have discussed coupled orientations on vector bundles. Let us turn to the analogous notion for manifolds.

\begin{definition}
\label{def:co-manifolds}
Let $M_0$ and $M_1$ be two smooth manifolds. For $i=0,1$, let $\pi_i: M_0 \times M_1 \to M_i$ be the projection. A {\em coupled orientation} on $(M_0, M_1)$ is a coupled orientation on the pair of vector bundles $(\pi_0^*TM_0, \pi_1^* TM_1)$ on $M_0 \times M_1$.
\end{definition}

\begin{remark}
\label{rem:coco}
By Lemma~\ref{lem:coupled-o}, a coupled orientation on $(M_0, M_1)$ corresponds to an orientation on the bundle $$T(M_0 \times M_1) = \pi_0^*TM_0 \oplus \pi_1^* TM_1.$$
In other words, we can think of a coupled orientation on $(M_0, M_1)$ as an orientation of the product $M_0 \times M_1$. We will only use this notion when $\dim M_0 = \dim M_1$; then, the correspondence to the orientation on $M_0 \times M_1$  is defined using the map \eqref{eq:OOC}, as in Convention~\ref{conv0}.
\end{remark}

\subsection{Coupled Spin structures}
\label{sec:coupledspin}
We now define some notions that will be used in Section~\ref{sec:3m} and later. 

\begin{definition}
\label{def:coupled-tspin}
We introduce the group $\tSpin(n; m)$ to be the $(\Z/2 \times \Z/2)$-cover of $\SO(n; m)$ given as the pull-back
\begin{equation}
\label{eq:coupledspin1}
\xymatrix{
\tSpin(n; m) \, \ar[d] \ar@{^{(}->}[r] &\Pin(n) \times \Pin(m)\ar[d] \\
\SO(n, m) \, \ar@{^{(}->}[r] &\O(n) \times \O(m)
}
\end{equation}
We let the {\em coupled Spin group} $\Spin(n; m)$ be the quotient of $\tSpin(n; m)$ by the diagonal subgroup $\Z/2 \subset \Z/2 \times \Z/2$. 
\end{definition}

Explicitly, $\tSpin(n; m)$ is the subgroup of $\Pin(n) \times \Pin(m)$ consisting of pairs of the form $(v_1 \dots v_k, u_1 \dots u_l)$ with $k+l$ even, and $\Spin(n; m)$ is its quotient by the equivalence relation
$$(v_1 \dots v_k, u_1 \dots u_l) \sim (-v_1 \dots v_k, -u_1 \dots u_l).$$

\begin{definition}
\label{def:coupled-spin}
Let $E$, $F$ be two real vector bundles (with inner products) over the same base $B$, of ranks $n$ and $m$ respectively.  A {\em $\widetilde{\mathit{Spin}}$ structure} on $(E, F)$ is a lift of $\Fr(E) \times \Fr(F) \to B$ to a principal $\tSpin(n, m)$ bundle. A {\em coupled Spin structure} on $(E, F)$ is a lift of $\Fr(E) \times \Fr(F) \to B$ to principal $\Spin(n, m)$ bundle.
\end{definition}

\begin{lemma}
\label{lemma:EF}
A coupled Spin structure on a pair $(E, F)$ exists if and only if $w_1(E) = w_1(F)$ and $w_2(E) = w_2(F)$. If one exists, then the space of such coupled Spin structures is (non-canonically) homotopy equivalent to $\Map(B, \RP^\infty)$. In particular, isomorphism classes of coupled Spin structures form a torsor over $H^1(B; \Z/2)$.
\end{lemma}

\begin{proof}
The pull-back diagram \eqref{eq:coupledspin1} shows that a $\tSpin$ structure is the same as the data of a coupled orientation together with Pin structures on the two bundles. From Lemma~\ref{lem:coupled-o} and the definition of the Pin group, a $\tSpin$ structure exists if and only if $w_1(E)=w_1(F)$ and $w_2(E)=w_2(F)=0$. 

In fact, the same diagram shows that we can characterize $\tSpin(n; m)$ as the central extension of $\SO(n; m)$ by $\Z/2 \times \Z/2$ determined by the pull-back of the element $\pi_1^*w_2 \oplus \pi_2^* w_2$ from  $$H^2(\B\O(n) \times \B\O(m); \Z/2 \times \Z/2) \cong H^2(\B\O(n) \times \B\O(m); \Z/2) \oplus H^2(\B\O(n) \times \B\O(m); \Z/2)$$ to $H^2(\B \SO(n; m); \Z/2 \times \Z/2)$, where $\pi_1$ and $ \pi_2$ are the projections from $\B\O(n) \times \B\O(m)$ to the two factors. Dividing $\tSpin(n;m)$ by the diagonal $\Z/2$ action gives $\Spin(n;m)$ as the central extension of $\SO(n; m)$ by $\Z/2$ determined by the pull-back of
$$ \pi_1^*w_2 + \pi_2^* w_2 \in H^2(\B\O(n) \times \B\O(m); \Z/2)$$
to $H^2(\B \SO(n; m); \Z/2)$. This says that, once we have a coupled orientation (which exists iff $w_1(E)=w_1(F)$), the remaining obstruction to the existence of a coupled Spin structure is $w_2(E) + w_2(F)$.

The classification statement about $\Map(B, \RP^\infty)$ follows from the same argument as in the proof of Proposition~\ref{prop:pinspace}. Furthermore, isomorphism is equivalent to homotopy (by the same argument as in Lemma~\ref{lem:homotopic}), and we have $\pi_0(\Map(B, \RP^\infty)) = H^1(B; \Z/2)$.
\end{proof}

\begin{remark}
\label{rem:EF}
We will mostly be interested in coupled Spin structures on a pair $(V, W)$ of vector spaces, viewed as vector bundles over a point. Suppose that the pair $(V, W)$ is equipped with a coupled orientation. Then, Pin structures on the two vector spaces determine a $\tSpin$ structure and hence (in a non-injective way) a coupled Spin structure on $(V, W)$. If we fix the isomorphism class of the Pin structure on $V$, the space of its representatives is homotopy equivalent to $\RP^\infty$; the same goes for $W$. The map to coupled Spin structures is modeled on the  multiplication $\RP^\infty \times \RP^\infty \to \RP^\infty$ which comes from viewing $\RP^\infty$ as $\B\Z/2$. In particular, suppose we have a loop of Pin structures on $(V, W)$ that gives the standard generator of $\pi_1(\RP^\infty) \cong \Z/2$ for both $V$ and $W$. Then, the induced loop in the space of coupled Spin structures is trivial, representing $1+1 =0$ in $\pi_1(\RP^\infty) \cong \Z/2$.
\end{remark}

\begin{remark}
We also have a commutative diagram
\begin{equation}
\label{eq:coupledspin2}
\xymatrix{
\Spin(n) \times \Spin(m) \, \ar[d] \ar[r] &\tSpin(n; m) \ar[d] \\
\SO(n) \times \SO(m) \, \ar[r] &\SO(n, m).
}
\end{equation}
This implies that Spin structures on two bundles $E$ and $F$ determine a $\tSpin$ and hence a coupled Spin structure of $(E, F)$. 
\end{remark}

Given a coupled orientation, Lemma~\ref{lemma:EF} and Remark~\ref{rem:EF} show that coupled Spin structures are a weaker condition that having Pin structures on each bundle. However, one advantage they have over Pin structures is that they behave well with respect to direct sums:
\begin{lemma}
\label{lem:EFEF}
Let $E, F, E', F'$ be four real vector bundles over a base $B$. Given coupled Spin structures on $(E, F)$ and $(E', F')$, there is an induced coupled Spin structure on $(E \oplus E', F \oplus F').$
\end{lemma}

\begin{proof}
This follows from the existence of a natural multiplication map
$$ \Spin(n, m) \times \Spin(n', m') \to \Spin(n+n', m + m')$$
as follows. If we denote the elements of $\Pin(n)$ by $v_1 \dots v_k$, and those of $\Pin(m)$ by $u_1\dots u_l$, then the elements of $\Spin(n, m)$ are equivalence classes of pairs $[(v_1 \dots v_k, u_1\dots u_l)]$ with $k+l$ even. Similarly, the elements of $\Spin(n', m')$ are equivalence classes $[(v'_1 \dots v'_{k'}, u'_1\dots u'_{l'})]$ with $k'+l'$ even. We let the multiplication  be
$$ \bigl( [(v_1 \dots v_k, u_1\dots u_l)],  [(v'_1 \dots v'_{k'}, u'_1\dots u'_{l'})] \bigr)  \mapsto [(v_1 \dots v_kv'_1 \dots v'_{k'}, u_1\dots u_lu'_1\dots u'_{l'})].$$
This lands in the right place because $k+k'+l+l'$ is even. Furthermore, we do not encounter the same commutativity problem as the one for Pin discussed at the end of Section~\ref{sec:pins}. In the current setting, when we commute the two factors, the result is
$$
 [(v'_1 \dots v'_{k'}v_1 \dots v_k, u'_1\dots u'_{l'}u_1\dots u_l)]= [((-1)^{kk'} v_1 \dots v_kv'_1 \dots v'_{k'}, (-1)^{ll'} u_1\dots u_lu'_1\dots u'_{l'})].
$$
The right hand side is same equivalence class as $[(v_1 \dots v_kv'_1 \dots v'_{k'}, u_1\dots u_lu'_1\dots u'_{l'})]$. 
Indeed, we have $(-1)^{kk'} = (-1)^{l l'}$ because $k \equiv k' \!\pmod{2}$ and $l \equiv l' \!\pmod{2}$. 
\end{proof}

We end with a few more properties of coupled Spin structures.

\begin{lemma}
\label{lem:EE}
Given any vector bundle $E$, the pair $(E, E)$ has a canonical coupled Spin structure.
\end{lemma}

\begin{proof}
This comes from the existence of a natural map $\O(n) \to \Spin(n, n)$ defined as follows. We view $\O(n)$ as $\Pin(n)/\pm 1.$ Given an element $v_1 \dots v_k \in \Pin(n)$, we map its image in $\O(n)$ to  
$$ [(v_1 \dots v_k, v_1 \dots v_k)] \in \Spin(n, n).$$
Observe that changing $v_1 \dots v_k$ by a sign produces the same equivalence class in $\Spin(n, n)$.
\end{proof}

\begin{lemma}
\label{lem:EFS}
Given vector bundles $E$, $F$ and $S$ over the same base $B$, a coupled Spin structure on $(E, F)$ naturally induces one on $(E \oplus S, F \oplus S).$
\end{lemma}

\begin{proof}
This follows from Lemmas~\ref{lem:EFEF} and \ref{lem:EFS}.
\end{proof}

\begin{lemma}
\label{lem:CK}
Let $f: E \to F$ be a bundle map between vector bundles over the same base $B$ (covering the identity on $B$). Suppose $f$ has fixed rank, so that $K=\ker(f)$ and $C=\coker(f)$ form vector bundles over $B$. Then, a coupled $\Spin$ structure on $(C, K)$ induces one on $(F, E)$.
\end{lemma}

\begin{proof}
Recall that our vector bundles are implicitly equipped with inner products; see Remark~\ref{rem:noinner}. Therefore, if we let $I = \text{im}(f)$ then by taking orthogonal complements we get isomorphisms $E \cong K \oplus I$ and $F \cong C \oplus I$. The conclusion follows from Lemma~\ref{lem:EFS}.
\end{proof}

\begin{lemma}
\label{lem:pinLag}
Let $(E, \omega)$ be a symplectic vector bundle, and $L_0, L_1 \subset E$ be Lagrangian sub-bundles such that $L_0 \cap L_1$ is also a sub-bundle (i.e., has fixed rank). Then:
\begin{enumerate}
\item An inner product on $E$ induces a canonical bundle isomorphism $\tau_{L_0, L_1} : L_0 \to L_1$;
\item There is a canonical coupled Spin structure on $(L_0, L_1)$.
 \end{enumerate}
\end{lemma}

\begin{proof}
$(1)$ Let $(L_0 \cap L_1)^\omega$ be the symplectic complement to $L_0 \cap L_1$ in $E$. Then $E'=(L_0 \cap L_1)^\omega/(L_0 \cap L_1)$ is symplectic, and inside it we have Lagrangian sub-bundles $L_0' = L_0/(L_0 \cap L_1)$ and $L_1' = L_1/(L_0 \cap L_1)$. Using the inner product we can identify each quotient bundle with the corresponding orthogonal complement. Furthermore, since $L_0'$ and $L_1'$ intersect in the zero section,  the symplectic form gives an identification between $L_0'$ and the dual to $L_1'$. The inner product then gives an identification between $L_0'$ and $L_1'$. Taking the direct sum of this with the identity on $L_0 \cap L_1$ we get the desired identification between $L_0$ and $L_1$. 

$(2)$ Apply part (a) and Lemma~\ref{lem:EE}.
\end{proof}

\begin{remark}
In the situation of Lemma~\ref{lem:pinLag} (1), note that $\tau_{L_1, L_0}$ is not the inverse to $\tau_{L_0, L_1}$. For example, when $L_0$ and $L_1$ are transverse lines in $\R^2$, the isomorphism $\tau_{L_0, L_1}$ is counterclockwise rotation, and its composition with $\tau_{L_1, L_0}$ is $-\id$ on $L_0$. 
\end{remark}

\begin{remark}
\label{rem:cansum}
On the other hand,  the isomorphisms from Lemma~\ref{lem:pinLag} (1) behave well with respect to direct sums. If $L_0, L_1 \subset E$ are as in the lemma, and $L_0', L_1' \subset E'$ is another pair of the same kind, then
$$ \tau_{L_0 \oplus L_0', L_1 \oplus L_1'} = \tau_{L_0, L_1} \oplus \tau_{L_0', L_1'}.$$
This implies that if we take the direct sum of the canonical coupled Spin structures on $(L_0, L_1)$ and $(L_0', L_1')$, we get the canonical coupled Spin structure on $(L_0 \oplus L_0', L_1 \oplus L_1')$. 
\end{remark}

\section{Canonical orientations in Lagrangian Floer homology} 
\label{sec:general}
In this section we explain how Pin structures produce canonical orientations on the moduli spaces appearing in Lagrangian Floer homology.

\begin{convention}
\label{conv}
Our exposition is inspired from Seidel's book \cite{SeidelBook}, but we will use the usual conventions in Heegaard Floer theory \cite{HolDisk}. Specifically, we work with Floer homology instead of Floer cohomology, and when we talk about polygon maps, the Lagrangian boundary conditions are ordered clockwise rather than counterclockwise. Our Floer chain groups $\CF_*(L_0, L_1)$ correspond to Seidel's Floer cochain groups $\CF^{n-*}(L_1, L_0)$. 
\end{convention}

\subsection{Orientation spaces} 
\label{sec:orspaces}
We review here the definition of orientation spaces associated to linear Lagrangian branes, following \cite[Section 11]{SeidelBook}. We describe a slight modification that is convenient for our purposes.

Given a $(2n)$-dimensional symplectic vector space $V$, we let $\Gr(V)$ be the Lagrangian Grassmannian, whose points are the Lagrangian subspaces of $V$. In \cite[Section 11h] {SeidelBook}, Seidel considers a natural map to a product of Eilenberg-MacLane spaces
\begin{equation}
\label{eq:muw}
(\mu, w_2): \Gr(V) \to K(\Z, 1) \times K(\Z/2, 2),
\end{equation}
and pulls back the product of the universal fibrations to obtain a $\Z \times \RP^{\infty}$ bundle over $\Gr(V)$, denoted $\Grp(V).$ The points $\Lambdap \in \Grp(V)$ are called {\em abstract linear Lagrangian branes}.

The first component of \eqref{eq:muw}, $$\mu: \Gr(V) \to K(\Z, 1) \cong S^1$$ can be described explicitly as the squared phase map associated to a compatible complex structure $I_V$ and a quadratic complex volume form on $(V, I_V)$; see \cite[Section 11j]{SeidelBook}. 

In our setting we shall mostly focus on the second component 
$$ w_2: \Gr(V) \to K(\Z/2, 2).$$
We let $\Grd(V)$ be the pullback of the universal fibration, which is an $\RP^{\infty}$ bundle over $\Gr(V)$. Furthermore, $\Grp(V)$ is a $\Z$-cover of $\Grd(V)$.

An alternate description of $\Grd(V)$ is as follows. Let 
$$ i_V : \Gr(V) \to \B\O(n)$$
be the composition of the forgetful map from $\Gr(V)$ to the ordinary Grassmannian of $n$-planes in $V$, with the map to $\B\O(n)$ induced by some inclusion $V \hookrightarrow \R^{\infty}$. Then, $\Grd(V)$ is the pullback of the bundle $\B\Pin(n) \to \B\O(n)$ under $i_V$; compare \cite[p.60]{SeidelBook}. Thus, we can think of a point $\Lambdad \in \Grd(V)$ as a Pin structure on the underlying Lagrangian subspace $\Lambda \in \Gr(V)$; i.e., a principal homogeneous $\Pin(n)$-space equipped with an isomorphism $P^\# \times_{\Pin(n)} \R^n \cong \Lambda$. This interpretation is  particularly useful in families: a family of Lagrangian subspaces in $V$ parametrized by a base space $B$ (i.e., a Lagrangian bundle over $B$) is described by a map $B \to \Gr(V)$, and a lift of this map to $\Grd(V)$ is the same as a Pin structure on the Lagrangian bundle. 

In the notation of \cite[Section 11j]{SeidelBook}, an element of $\Grp(V)$ is a triple 
$$\Lambdap = (\Lambda, \alpha, P^\#),$$
where $\Lambda \in \Gr(V)$, $\alpha \in \R$ is such that $e^{2\pi i \alpha}=\mu(\Lambda) \in S^1$, and $P^\#$ is a Pin structure on $\Lambda$. The pair $\Lambdad=(\Lambda, P^\#)$ is the projection of $\Lambdap$ to $\Grd(V)$.

Next, consider two points $\Lambdap_0, \Lambdap_1 \in \Grp(V)$, such that the underlying Lagrangians $\Lambda_0, \Lambda_1 \subset V$ intersect transversely. Choose a path $$\rho^\#=(\Lambdap_t)_{t \in [0,1]}$$ connecting the two points, and let $\rho=(\Lambda_t)_{t \in [0,1]}$ be its projection to $\Gr(V)$. There is a technical condition we need to impose on $\rho$, that it has negative definite crossing form at $t=1$ with the constant path $\Lambda_1$; see \cite[Section 11g]{SeidelBook}. Then, the Maslov index of the path $\rho$ depends only on $\Lambdap_0$ and $\Lambdap_1$, and is denoted
$$i(\Lambdap_0, \Lambdap_1) \in \Z.$$
Furthermore, one can consider a Fredholm problem on a capped half-infinite strip $H$, with moving Lagrangian boundary conditions $\{\Lambda_t\}$ that are constant near the infinite end,  as in \cite[Section 11g]{SeidelBook} or \cite[p.143]{FOOO}, but reflected to be in line with our Conventions~\ref{conv}.:
$$ {
\fontsize{10pt}{11pt}\selectfont
   \def\svgwidth{2.2in} 
   %% Creator: Inkscape 1.3.2 (091e20e, 2023-11-25), www.inkscape.org
%% Accompanies image file 'cap.pdf' (pdf, eps, ps)
\begingroup%
  \makeatletter%
  \providecommand\color[2][]{%
    \errmessage{(Inkscape) Color is used for the text in Inkscape, but the package 'color.sty' is not loaded}%
    \renewcommand\color[2][]{}%
  }%
  \providecommand\transparent[1]{%
    \errmessage{(Inkscape) Transparency is used (non-zero) for the text in Inkscape, but the package 'transparent.sty' is not loaded}%
    \renewcommand\transparent[1]{}%
  }%
  \providecommand\rotatebox[2]{#2}%
  \newcommand*\fsize{\dimexpr\f@size pt\relax}%
  \newcommand*\lineheight[1]{\fontsize{\fsize}{#1\fsize}\selectfont}%
  \ifx\svgwidth\undefined%
    \setlength{\unitlength}{201.61538432bp}%
    \ifx\svgscale\undefined%
      \relax%
    \else%
      \setlength{\unitlength}{\unitlength * \real{\svgscale}}%
    \fi%
  \else%
    \setlength{\unitlength}{\svgwidth}%
  \fi%
  \global\let\svgwidth\undefined%
  \global\let\svgscale\undefined%
  \makeatother%
  \begin{picture}(1,0.48848057)%
    \lineheight{1}%
    \setlength\tabcolsep{0pt}%
    \put(0,0){\includegraphics[width=\unitlength,page=1]{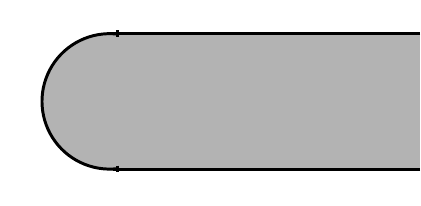}}%
    \put(0.50022782,0.0113657){\makebox(0,0)[lt]{\lineheight{1.25}\smash{\begin{tabular}[t]{l}$\Lambda_0$\end{tabular}}}}%
    \put(0.50022782,0.43886192){\makebox(0,0)[lt]{\lineheight{1.25}\smash{\begin{tabular}[t]{l}$\Lambda_1$\end{tabular}}}}%
    \put(0.00657083,0.29572711){\makebox(0,0)[lt]{\lineheight{1.25}\smash{\begin{tabular}[t]{l}$\Lambda_t$\end{tabular}}}}%
  \end{picture}%
\endgroup%

}$$
We call $H$ a {\em cap}. The $\delbar$ operator on $H$ associated to this set-up is denoted $\delbar_\rho$. Consider its determinant index bundle 
$$\det(\delbar_{\rho})=  \ltop ((\coker \delbar_\rho)^*)\otimes \ltop (\ker \delbar_\rho).$$

The {\em orientation space} $$o(\Lambdap_0, \Lambdap_1)$$ is the abelian group generated by the two orientations $\omega, \bar \omega$ on $\det(\delbar_{\rho})$, modulo the relation $\bar \omega = - \omega$. Since this is a free abelian group of rank-$1$,  it is non-canonically isomorphic to $\Z$.
\begin{remark}
\label{rem:conv}
Rotating our cap by $180^\circ$, we see that it is equivalent to the cap in Seidel's book with boundary conditions given by the reflected path $\Lambda_{1-t}$ from $\Lambda_1$ to $\Lambda_0$. By Equation (11.27) in \cite{SeidelBook}, we have 
\begin{equation}
\label{eq:inddetrho}
\index(\det(\delbar_{\rho})) = i(\Lambdap_1, \Lambdap_0) = n-i(\Lambdap_0, \Lambdap_1).
\end{equation}
Also, the orientation space $o(\Lambdap_0, \Lambdap_1)$ is canonically isomorphic to the one in Seidel's book tensored with $\ltop \Lambda_t$ (for any $t$). Note that, compared with \cite{SeidelBook}, we have not changed the definition of the Maslov index $i(\Lambdap_0, \Lambdap_1)$; but we have changed the definition of $o(\Lambdap_0, \Lambdap_1)$. We have also changed the meaning of the notation $H$: in \cite{SeidelBook}, our cap is denoted $\bar H$, and its reflection across a vertical axis is called $H$.
\end{remark}

\begin{lem}
\label{lem:nodep}
  The orientation space $$o(\Lambdap_0, \Lambdap_1)$$ depends, up to canonical isomorphism, only on $\Lambdap_0$ and $ \Lambdap_1$ (in particular, not on the choice of path $\rho^\#$ with these endpoints).
\end{lem}
\begin{proof}  
This is discussed in \cite[p.164]{SeidelBook}. The essential point in the following:  the space of paths in $\Gr(V)$, with endpoints $\Lambda_0$ and $\Lambda_1$, has $\Z$-many components, each of which has fundamental group $\Z/2$. The orientation space of the operator $\det(\delbar_{\rho})$  associated to paths $\rho$ defines a non-trivial local system over this space, as shown by de Silva in \cite{DeSilva}. Imposing the grading condition distinguishes a homotopy class of paths. More importantly, considering paths of Lagrangians equipped with Pin structures yields a simply connected space of paths, which can be thought of as  a $2$-fold cover of the space of paths in $\Gr(V)$.  Its components over a fixed path in $\Gr(V)$ are distinguished by the homotopy class of the path of Pin structures relative the endpoints. Thus, parallel transport through paths of paths in $\Gr^{\#}(V)$ yields isomorphisms of determinant lines that are independent of all choices. 
\end{proof}

\begin{lemma}
\label{lem:nontrivial} 
The local system formed by the orientation spaces $o(\Lambdap_0, \Lambdap_1)$ over the space of all Pin structures on $\Lambda_1$ is non-trivial (and similarly for $\Lambda_0$). 
\end{lemma}

\begin{proof}
Recall that the space of Pin structures on $\Lambda_1$ has the homotopy type of $\RP^\infty$, and thus $\pi_1=\Z/2$; see Example~\ref{rem:non-equivalent_iso}. Consider a homotopically non-trivial loop $\{P_{1,s}^{\#}\}_{s \in [0,1]}$ of Pin structures on $\Lambda_1$, restricting to $P^{\#}_1$ at $t \in \{0,1\}$. Concatenating this with the path $(\Lambda_t)$ gives a family of boundary conditions for the cap, parametrized by $s \in [0,1]$. There is a corresponding family of Cauchy-Riemann operators. Appealing once again to the non-triviality of $\det(\delbar_{\rho})$ over the space of paths in $\Gr(V)$ \cite{DeSilva}, we conclude that the monodromy on the orientation space is $-1$.  Compare \cite[Remark 11.19]{SeidelBook}. 
\end{proof}

The following result is ultimately derived from the fact that every even index loop in the Grassmanian of Lagrangians is homotopic to a loop in the image of the action by the unitary group; that equips the orientation operator with a trivialization coming from deforming the corresponding Fredholm problem to one which is complex linear. This will not be apparent in our proof, which uses instead formal properties established in \cite{SeidelBook}.

\begin{proposition}
\label{prop:epsL}
Suppose  we are given $\Lambdad_0, \Lambdad_1 \in \Grd(V)$, and a value $\epsilon \in \Z/2$. Pick lifts $\Lambdap_0, \Lambdap_1 \in \Grp(V)$ of $\Lambdad_0, \Lambdad_1$ so that 
\begin{equation}
\label{eq:epsL}
\epsilon = i(\Lambdap_0, \Lambdap_1) \mod{2}.
\end{equation}
Then, the orientation spaces $o(\Lambdap_0, \Lambdap_1)$ are canonically isomorphic, for all choices of $\Lambdap_0, \Lambdap_1$ satisfying \eqref{eq:epsL}.
\end{proposition}

\begin{proof}
A shift operation $S$ on the elements $\Lambdap=(\Lambda, \alpha, P^\#) \in \Grp(V)$ is introduced in \cite[Section 11k]{SeidelBook}:
$$ S\Lambdap=(\Lambda, \alpha-1, P^\# \otimes \ltop (\Lambda)).$$
Lemma 11.21 in \cite{SeidelBook} says that
\begin{equation}
\label{eq:indexshift}
 i(\Lambdap_0, S\Lambdap_1)= i( \Lambdap_0, \Lambdap_1)-1
 \end{equation}
and 
\begin{equation}
\label{eq:orshift}
 o(\Lambdap_0, S\Lambdap_1) \cong o(\Lambdap_0, \Lambdap_1).
\end{equation}
For $n=1$, the isomorphism \eqref{eq:orshift} is determined by clockwise rotation by $\pi$ in the plane. In higher dimensions, it is determined by asking for it to be compatible with direct sums. 

Similar arguments to those in the proof of Lemma 11.21  in \cite{SeidelBook} show that
 \begin{equation}
\label{eq:indexshift2}
 i(S\Lambdap_0, \Lambdap_1)= i( \Lambdap_0, \Lambdap_1)+1
 \end{equation}
and 
\begin{equation}
\label{eq:orshift2}
 o(S\Lambdap_0, \Lambdap_1) \cong o(\Lambdap_0, \Lambdap_1).
 \end{equation}
 
 For us, a more relevant shift operation $\Sigma$ is the one given by
 $$ \Sigma \Lambdap = (\Lambda, \alpha-1, P^\#),$$
since this preserves the underlying $\Lambdad \in \Grd(V)$. Observe that the corresponding double shift is the same: 
\begin{equation}
\label{eq:doubleshift}
\Sigma \circ \Sigma=S \circ S.
\end{equation}

The index does not depend on the Pin structure, so from \eqref{eq:indexshift} and \eqref{eq:indexshift2} we deduce that
$$ i( \Lambdap_0, \Lambdap_1)=i(\Lambdap_0, \Sigma \Lambdap_1)+1= i(\Sigma\Lambdap_0, \Lambdap_1)-1.$$

Therefore, to establish what we need for the proposition, it suffices to find canonical isomorphisms
$$ o( \Lambdap_0, \Lambdap_1)\cong o(\Lambdap_0, \Sigma^2 \Lambdap_1) \cong o(\Sigma^2\Lambdap_0, \Lambdap_1) \cong  o(\Sigma\Lambdap_0, \Sigma \Lambdap_1) .$$

The first two isomorphisms follow from \eqref{eq:orshift} and \eqref{eq:orshift2}, together with \eqref{eq:doubleshift}. The final isomorphism, $o( \Lambdap_0, \Lambdap_1)\cong o(\Sigma\Lambdap_0, \Sigma \Lambdap_1)$, is clear because by using the family $\{\Sigma \Lambdap_t\}$ for the index problem on the cap $H$, the index problem does not change: the Lagrangian paths $\{\Lambda_t\}$ are the same. 
\end{proof}

In view of Proposition~\ref{prop:epsL}, we introduce the notation
\begin{equation}
\label{eq:newo}
o(\Lambdad_0, \Lambdad_1, \epsilon) := o(\Lambdap_0, \Lambdap_1)
\end{equation}
for $\Lambdad_0, \Lambdad_1 \in \Grd(V)$, and  $\epsilon \in \Z/2$. Here, $\Lambdap_0$ and $ \Lambdap_1$ are any lifts satisfying \eqref{eq:epsL}.

\subsection{Lagrangian Floer homology}
\label{sec:LFH}
Lagrangian Floer homology over $\Z$ is developed in \cite{FOOO}, \cite{FOOO2} using Spin structures on the Lagrangians. In Seidel's book \cite{SeidelBook}, it was noted that only Pin structures are necessary, but the exposition there was under the assumption that the symplectic manifold $M$ satisfies $2c_1(TM)=0$ (so that the Floer complex is $\Z$-graded). In the context of Heegaard Floer theory, it will be convenient to use Pin structures, but the manifold does not have $2c_1(TM)=0$.  Nevertheless, since the only Lagrangians considered in Heegaard Floer theory are orientable, this setting admits a canonical relative $\Z/2$ grading. In fact, one can specify an absolute $\Z/2$ grading, which then suffices for the sign discussion in \cite{SeidelBook} to go through, using what we set up in Section~\ref{sec:orspaces}.

  We present here an adaptation of the techniques in \cite[Chapter 8]{FOOO2} and \cite[Section 11]{SeidelBook} to the setting of interest to us.

\begin{comment}
Given a (real) Fredholm linear operator $D$, we let $\det(D)$ be the determinant line
$$ \det(D):= \ltop ((\coker D)^*)\otimes \ltop (\ker D).$$
More generally, if we have a family of Fredholm operators $\{D_b\}_{b \in B}$ parametrized by a base space $B$, we let $\det(D)$ be the corresponding line bundle over $B$.
\end{comment}

Let $(M, \omega)$ be a closed symplectic manifold of dimension $2n$. Let $L_0, L_1 \subset M$ be two connected, transversely intersecting Lagrangians, such that $L_0$ and $L_1$ are orientable, and equipped with Pin structures $P_0^\#$ and $P_1^\#$. We do not choose orientations on the Lagrangians, but we will assume that they are coupled oriented, in the sense of Definition~\ref{def:co-manifolds}; i.e., that we have an orientation on $L_0 \times L_1$. (See Remark~\ref{rem:coco}.)

Given a coupled orientation on $(L_0, L_1)$, every intersection point $\x \in L_0 \cap L_1$ admits an absolute mod $2$ grading
$$ \gr(\x) \in \Z/2,$$
obtained as follows. We compare the orientation of $T_{\x}L_0 \oplus T_{\x}L_1 \cong T_{\x, \x}(L_0 \times L_1)$ with the orientation of $T_{\x}M$ coming from $\omega^n$. If they are the same, we set $\gr(\x) = n \pmod{2}$. If they are different, we set $\gr(\x) = n+1 \pmod{2}.$

%We will also assume that $L_0$ and $L_1$ are monotone, i.e., that there exists $\lambda > 0$ such that
%$$ [\omega]|_{\pi_2(M, L_i)} = \lambda \cdot \mu|_{\pi_2(M, L_i)},$$
%where $\mu$ denotes the Maslov class, and $i \in \{0,1\}$. 

Given $\x \in L_0 \cap L_1$, in the notation of Section~\ref{sec:orspaces}, we have elements
 $$\Lambdad_i = (T_{\x}L_i, P_i^\#|_{T_{\x} L_i}) \in \Grd(T_{\x}M), \ \ i=0, 1.$$ 
We consider the orientation space
 \begin{equation}
 \label{eq:ox}
  o(\x) \coloneqq o(\Lambdad_0, \Lambdad_1, \gr(\x)).
  \end{equation}
We denote a generator of $o(\x)$, i.e. an orientation of the corresponding determinant line, by $\omega_\x$.

The Lagrangian Floer complex is 
$$ \CF_*(L_0, L_1) = \bigoplus_{\x \in L_0 \cap L_1} o(\x),$$
 with the differential 
\begin{equation}
\label{eq:del}
\del \omega_\x = \sum_{\y \in L_0 \cap L_1} \sum_{\substack{\phi \in \pi_2(\x, \y)\\ \mu(\phi)=1}} \bigl( \# \Mhat(\phi) \bigr) \cdot \omega_\y.
\end{equation}
Here, $\pi_2(\x, \y)$ denotes the space of relative homotopy classes between $\x$ and $\y$ (with boundary on the two Lagrangians), and $\mu(\phi)$ denotes the Maslov index of such a class. Furthermore, $\Mhat(\phi)=\M(\phi)/\R$ is the count of $J$-holomorphic strips (solutions to Floer's equation) in the class $\phi$, after dividing by translation by $\R$. It remains to explain what we mean by the signed count $ \# \M(\phi) \in \Z$, which depends on the  orientations $\omega_\x\in o(\y)$ and $\omega_\y \in o(\y)$.

\begin{remark}
In Heegaard Floer theory, one considers several variants of the Floer complex, keeping track of the intersections of the holomorphic disks with a (real codimension two) symplectic hypersurface $D \subset M$ which is disjoint from the Lagrangians. This additional complication does not affect our discussion of signs. 
\end{remark}

\begin{remark}
The coupled orientation on $(L_0, L_1)$ makes the Lagrangian Floer complex absolutely $\Z/2$-graded. In certain situations, the grading can be lifted to a relative or absolute $\Z/2N$- or $\Z$-grading. Again, this issue is independent from our discussion of signs.
\end{remark}

To define the signed count $ \# \Mhat(\phi)$, we need to orient the moduli space. Let
$$\Path(L_0, L_1)=\{\gamma\in C^{\infty}([0,1], M) \mid \gamma(0) \in L_0, \gamma(1) \in L_1\}$$
and let $\Omega(\phi)$ be the space of (non necessarily holomorphic) Whitney disks from $\x$ to $\y$ in the class $\phi \in \pi_2(\x, \y)$; that is, the space of smooth paths in $\Path(L_0, L_1)$ from the constant path $c_\x$ at $\x$ to the constant path $c_\y$ at $\y$, in the class $\phi$. By picking a homeomorphism from $[0,1]$ to $\R$, we can view the moduli space $\M(\phi)$ as a subset of $\Omega(\phi)$.

The orientation on $\M(\phi)$ is governed by orienting the determinant index bundle $\det(D\delbar)$ over $\Omega(\phi)$. Recall that the orientation spaces $o(\x)$ and $o(\y)$ also come from determinant index bundles, for operators on the cap $H$ associated to paths of linear Lagrangian branes. We will denote these by $\delbar_\x$ and $\delbar_\y$, for convenience (instead of using the path of branes in the subscript, as we did in Section~\ref{sec:orspaces}). Once we have an orientation on $\M(\phi)$, we get one on $\Mhat(\phi)=\M(\phi)/\R$ using the ordered convention $\M(\phi) \cong \R \times \Mhat(\phi).$

To see that orientations $\omega_\x\in o(\x)$ and $\omega_\y \in o(\y)$ induce one on $\det(D\delbar)$, it suffices to prove the following lemma.

\begin{lemma}
\label{lemma:Dd}
The Pin structures and coupled orientation on the Lagrangians $(L_0, L_1)$ induce a canonical isomorphism: 
\begin{equation}
\label{eq:Dd}
 \det(\delbar_{\x}) \otimes \det(D\delbar) \cong  \det(\delbar_{\y}).
\end{equation}
\end{lemma}

\begin{proof}
The arguments are as in Proposition 11.13 and Section (12b) in \cite{SeidelBook}, with minor modifications. Given  $u \in \Omega(\phi)$, we view it as a strip $u: [0,1] \times \R \to M$, and trivialize $u^*TM$ over its domain. Over each of the boundaries $\{0\} \times \R$ and $\{1\} \times \R$ we get a family of linear Lagrangians of the same symplectic vector space $V=u^*TM$. These families interpolate between $T_{\x} L_i$ and $T_{\y}L_i$, and are equipped with Pin structures. 

After some deformations of Fredholm operators, we can glue together the operators $\delbar_\x$ and $D\delbar_u$ to obtain a new operator $\delbar_{\y}^{\new}$ on the cap $H$, with boundary conditions interpolating between $T_{\y} L_0$ and $T_{\y} L_1$:
$$ {
\fontsize{10pt}{11pt}\selectfont
   \def\svgwidth{4.2in} 
   %% Creator: Inkscape 1.3.2 (091e20e, 2023-11-25), www.inkscape.org
%% Accompanies image file 'disk.pdf' (pdf, eps, ps)
\begingroup%
  \makeatletter%
  \providecommand\color[2][]{%
    \errmessage{(Inkscape) Color is used for the text in Inkscape, but the package 'color.sty' is not loaded}%
    \renewcommand\color[2][]{}%
  }%
  \providecommand\transparent[1]{%
    \errmessage{(Inkscape) Transparency is used (non-zero) for the text in Inkscape, but the package 'transparent.sty' is not loaded}%
    \renewcommand\transparent[1]{}%
  }%
  \providecommand\rotatebox[2]{#2}%
  \newcommand*\fsize{\dimexpr\f@size pt\relax}%
  \newcommand*\lineheight[1]{\fontsize{\fsize}{#1\fsize}\selectfont}%
  \ifx\svgwidth\undefined%
    \setlength{\unitlength}{393.18922148bp}%
    \ifx\svgscale\undefined%
      \relax%
    \else%
      \setlength{\unitlength}{\unitlength * \real{\svgscale}}%
    \fi%
  \else%
    \setlength{\unitlength}{\svgwidth}%
  \fi%
  \global\let\svgwidth\undefined%
  \global\let\svgscale\undefined%
  \makeatother%
  \begin{picture}(1,0.27721273)%
    \lineheight{1}%
    \setlength\tabcolsep{0pt}%
    \put(0,0){\includegraphics[width=\unitlength,page=1]{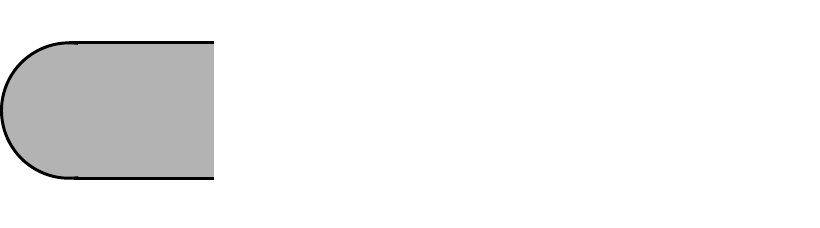}}%
    \put(0.18692564,0.00665131){\makebox(0,0)[lt]{\lineheight{1.25}\smash{\begin{tabular}[t]{l}$T_{\x}L_0$\end{tabular}}}}%
    \put(0.35976379,0.00665131){\makebox(0,0)[lt]{\lineheight{1.25}\smash{\begin{tabular}[t]{l}$T_{\x}L_0$\end{tabular}}}}%
    \put(0.77779898,0.00665131){\makebox(0,0)[lt]{\lineheight{1.25}\smash{\begin{tabular}[t]{l}$T_{\y}L_0$\end{tabular}}}}%
    \put(0.185952,0.25176981){\makebox(0,0)[lt]{\lineheight{1.25}\smash{\begin{tabular}[t]{l}$T_{\x}L_1$\end{tabular}}}}%
    \put(0.36189378,0.25176981){\makebox(0,0)[lt]{\lineheight{1.25}\smash{\begin{tabular}[t]{l}$T_{\x}L_1$\end{tabular}}}}%
    \put(0.78186239,0.25176981){\makebox(0,0)[lt]{\lineheight{1.25}\smash{\begin{tabular}[t]{l}$T_{\y}L_1$\end{tabular}}}}%
    \put(0,0){\includegraphics[width=\unitlength,page=2]{disk.pdf}}%
  \end{picture}%
\endgroup%

}$$
Under the gluing operation, the indices are added and the determinant lines are tensored:
\begin{align}
\label{eq:indexdet}
\index(\delbar_{\x}) + \index(D\delbar)  &=  \index (\delbar_{\y}^{\new}),\\
\det(\delbar_{\x}) \otimes \det(D\delbar) &\cong \det(\delbar_{\y}^{\new}).\label{eq:indexdetdet}
\end{align}

From the construction of the orientation spaces in Proposition~\ref{prop:epsL}, together with Equation~\eqref{eq:inddetrho}, we know we must have
$$ \index (\delbar_\x) \equiv n- \gr(\x) \!\!\!\!\pmod{2}.$$
Equation~\eqref{eq:indexdet} implies that
$$ \index (\delbar_\y^{\new}) \equiv (n - \gr(\x)) + (\gr(\x) - \gr(\y)) \equiv n-\gr(\y)   \!\!\!\!\pmod{2}.$$
 
 This means that the boundary conditions for $\delbar_\y^{\new}$ are of the kind needed to define the orientation space $o(\y)$, so  the operator $\delbar_\y^{\new}$ can play the role of $\delbar_\y$. 
Equation~\eqref{eq:indexdetdet} gives the desired isomorphism \eqref{eq:Dd}.
\end{proof}

This completes the description of the differential $\del$ on $\CF_*(L_0, L_1)$. To make sure that the sum in \eqref{eq:del} is finite, and that  $\del^2=0$, we need extra assumptions. The usual requirement is that the Lagrangians are monotone with Maslov numbers greater than $2$. However, this is not the case in Heegaard Floer theory; instead, some {\em ad hoc} arguments are used there. In particular, one shows that the contributions from disk and sphere bubbles to $\del^2$ are zero. (The signs of these  contributions are discussed in the next section.) Assuming we are in a situation where the bubble counts are zero, the only other place in the proof that $\del^2=0$ where signs play a role is in the fact that the orientations on $\Mhat(\phi)$ are compatible with gluing. For orientations coming from Pin structures as above, the proof of compatibility with gluing is as in \cite[Section 12f]{SeidelBook}. Thus, we obtain Lagrangian Floer homology groups, denoted $\HF_*(L_0, L_1)$. 

%\begin{remark}
%In Seidel's book \cite{SeidelBook}, as in most of the symplectic geometry literature, one defines Floer cohomology as opposed to homology. We have chosen here to define Floer homology, following the conventions in Heegaard Floer theory.
%\end{remark}
\begin{proposition}
Given Lagrangians $L_0, L_1 \subset M$ equipped with Pin structures and a coupled orientation (and satisfying the assumptions above), the Floer homology groups $\HF_*(L_0, L_1)$ are well-defined up to canonical isomorphism.
\end{proposition}

\begin{proof}
 Proposition~\ref{prop:epsL} ensures that the orientation spaces $o(\x)$ are well-defined (up to canonical isomorphisms), and these isomorphisms commute with the differentials.
\end{proof}

\subsection{Bubbles}
\label{sec:bubbles}
When showing $\del^2=0$ in the Lagrangian Floer complex, we may encounter moduli spaces of stable disk bubbles, consisting of sphere bubbles attached to disks. For future reference, let us discuss their orientations.

For sphere bubbles, the linearization of the $\delbar$ operator is complex, so those moduli spaces are canonically oriented.

For disk bubbles, the ones relevant to $\del^2=0$ are those in relative homotopy classes $\phi \in \pi_2(M, L_i)$ ($i=0$ or $1$) with one fixed boundary point:
\begin{equation}
\label{eq:u}
 u: (D^2, \del D^2) \to (M, L_i), \ \ u(1) = \x  \in L_0 \cap L_1,
 \end{equation}
and with Maslov index $\mu(\phi) \leq 2$. In Heegaard-Floer theory, topological constraints preclude all disks of non-positive Maslov index, so only the case of Maslov index $2$ disks is relevant. Following the notation in \cite[Section 3.7]{HolDisk}, we let $\cN(\phi)$ be the moduli space of such $J$-holomorphic disks, and we let $\cNhat(\phi)$ be its quotient by the (two-dimensional) automorphism group $\Aut(D^2, 1)$. To define a signed count $\# \cNhat(\phi)$, we need the following fact.

\begin{proposition}
\label{prop:signdisk}
Fix $i \in \{0, 1\}$. The Pin structure on the Lagrangian $L_i$ induces a canonical orientation of $\cNhat(\phi)$.
\end{proposition}

\begin{proof}
To orient $\cN(\phi)$, we need to trivialize the determinant index bundle $\det(D\delbar)$ over the space of disks $u$ satisfying \eqref{eq:u}. Given such a disk, we trivialize $u^*TM$ over the disk $D^2$. The pullback $u^*TL_i$ gives a loop $\rho: S^1 \to \Gr(\C^n)$ of Lagrangian subspaces along $\del D^2$. Let us denote the Cauchy-Riemann operator on $D^2$, with these boundary conditions, by $\delbar_{\rho}$. This is studied in \cite{SeidelBook}, where it is denoted $D_{D,\rho}$. Since the Maslov index of the path $\rho$ is even (equal to $2$), Lemma 11.17 in \cite{SeidelBook} says that the Pin structure induces an isomorphism
$$ \det(\delbar_\rho) \cong \ltop(\rho(1)).$$

The actual operator we are interested in, $D\delbar_u$, differs from $\delbar_{\rho}$ in the fact that its value at $1$ is fixed to be zero (because $u(1)=\x$ is fixed). Thus, we have
$$\ind(D\delbar_u) = \ind(\delbar_\rho) - \rho(1)$$
as virtual vector spaces. It follows that 
$$\det(D\delbar_u) \cong \det(\delbar_\rho) \otimes \ltop(\rho(1))^{-1} \cong \R,$$
so the bundle $\det(D\delbar)$ is trivialized, as claimed.

The orientation on $\cN(\phi)$ induces one on the quotient $\cNhat(\phi)$ after fixing an orientation on the automorphism group $\Aut(D^2, 1)$, which we do as follows. We identify $D^2$ with the upper half-space $\Hs$ (preserving their orientations as subsets of $\C$), such that $1 \in \del D^2$ gets mapped to infinity. Then $\Aut(\Hs, \infty)$ is generated by translations by $t \in \R$ and dilations by $e^s, s \in \R$. We choose the orientation on its tangent space to be given by the ordered basis $(\del/\del s, \del /\del t).$ Finally, to relate the orientations on $\cNhat(\phi)$ and $\cN(\phi)$, we use the convention $ \cN(\phi) \cong  \Aut(D^2, 1) \times \cNhat(\phi)$.
\end{proof}

\begin{example}
\label{ex:disk}
Let $L$ be the unit circle in the plane $M=\R^2$, and let $\phi$ be the class of the unit disk. This has Maslov index two, and the moduli space $\cNhat(\phi)$ consists of a single point. Proposition~\ref{prop:signdisk} assigns to this point the positive orientation if $L$ is equipped with the Lie group Pin structure, and the negative one if it is equipped with the bounding Pin structure. This can be read from the proof of Lemma 11.17 in \cite{SeidelBook}, in which the study of orientations on $\det(\delbar_{\rho})$ is reduced to the case where the Lagrangian loop $\rho$ is constant. In that case, the trivialization rule is spelled explicitly, depending on the Pin structure as noted above.
\end{example}

Given a class $\phi \in \pi_2(M, L_i)$, there is a corresponding homotopy class of strips (still denoted $\phi$) in $\pi_2(\x, \x)$. If $\mu(\phi)=2$, then the moduli space of strips $\Mhat(\phi)$ is one-dimensional, and the moduli space of disk bubbles $\cNhat(\phi)$ is part of its (zero-dimensional) boundary.

In this case, observe that Lemma~\ref{lemma:Dd} provides a canonical orientation on $\Mhat(\phi)$, because when $\x=\y$, the tensor product $\det(\delbar_{\x}) \otimes \det(\delbar_{\y})$ is trivial. 

\begin{proposition}
\label{prop:bubbly}
For disk bubbles with boundary on $L_0$, the orientation on $\cNhat(\phi)$ from Proposition~\ref{prop:signdisk} agrees with the boundary orientation of $\Mhat(\phi)$ from Lemma~\ref{lemma:Dd}. For disk bubbles with boundary on $L_1$, the two orientations disagree.
\end{proposition}

\begin{proof}
Since $\mu(\phi)=2$, the moduli space $\M(\phi)$ is one-dimensional, with some part of its boundary coming from disk bubbles in $\cN(\phi) \cong \Aut(D^2, 1) \times  \cNhat(\phi)$. Near that part, the strips in $\M(\phi)$ are obtained by gluing a disk bubble to the constant trajectory at $\x$. The gluing involves a parameter $T \to \infty$. Given a disk $u$ in $\cN(\phi)$, gluing it with the parameter $T$ is equivalent to gluing some dilation of $u$ by $e^s$ with parameter $1$, where $s \to -\infty$ as $T \to \infty$. Thus, 
the dilation coordinate $-s$  on $ \Aut(D^2, 1)\cong \Aut(\Hs, \infty)$ corresponds to the gluing parameter. On the other hand, the translation coordinate $t$ becomes (after gluing) either the $\R$-translation coordinate on $\M(\phi) \cong \R \times \Mhat(\phi)$, or its opposite, depending on whether we consider bubbles with boundary on $L_0$ or $L_1$. See Figure~\ref{fig:diskbubbles}. 

Since the $\R$ factor comes first in the identifications $\M(\phi) \cong \R \times \Mhat(\phi)$ and $\cN(\phi) \cong \Aut(D^2, 1) \times  \cNhat(\phi)$, the claim follows from the way we chose the orientation on $\Aut(D^2, 1)$: we used $(\del/\del s, \del /\del t)$, which has the same orientation as $(\del/\del t, -\del /\del s).$
\end{proof}

\begin{figure}
{
\fontsize{10pt}{11pt}\selectfont
   \def\svgwidth{4.5in}
   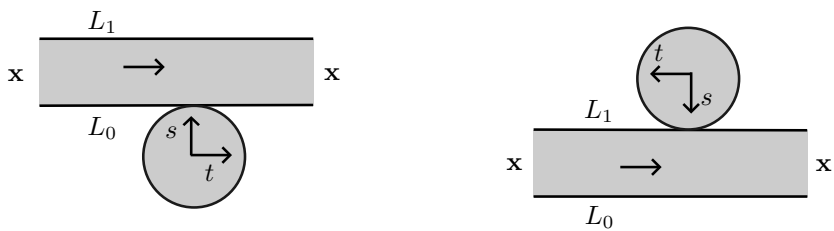
}
\caption{Orientations for gluing a disk bubble to a constant trajectory.}
\label{fig:diskbubbles}
\end{figure}

The count of disk bubbles with boundary on a Lagrangian $L$ appears in the calculation of the Floer homology of $L$ with itself, which can be formulated using  the pearl complex \cite{BiranCornea}. We state here the result when the count is zero.

\begin{proposition}
\label{prop:PSS}
Under the assumptions from Section~\ref{sec:LFH}, suppose $L_0=L_1$ (with the same Pin structure), and denote this Lagrangian by $L$. Pick the coupled orientation on $(L, L)$ to be the product of either orientation on $L$ with itself. Assume the count of stable disk bubbles vanishes. Then, we have a canonical isomorphism
$$ \HF_*(L, L) \cong H_*(L; |\ltop TL |),$$
where $|\ltop TL |$ denotes the orientation local system on $L$.
\end{proposition}

\begin{proof}
This follows from the PSS isomorphism \cite{PSS}, which is also discussed in \cite[Section (12e)]{SeidelBook} and \cite{Zapolsky}. The PSS isomorphism involves counts of configurations consisting of holomorphic caps joined to Morse trajectories, and these can be graded using Pin structures by an analogue of Lemma~\ref{lemma:Dd}. The isomorphism is usually stated in terms of Floer cohomology, saying that
$$ \HF^*(L, L) \cong H^*(L; \Z).$$
The statement for Floer homology can be deduced from this in view of our Conventions~\ref{conv} and Poincar\'e duality for $L$. 
\end{proof}

\begin{remark}
\label{rem:PSS}
Since $L$ is assumed to be orientable, the local system $|\ltop TL |$ is trivial, so Proposition~\ref{prop:PSS} implies the existence of an isomorphism $\HF_*(L, L) \cong H_*(L; \Z)$. However, this isomorphism is not canonical, because it depends on the choice of an orientation on $L$.
\end{remark}

\subsection{Decompositions}
\label{sec:decompose}
Inside the path space $\Path(L_0, L_1)$, we have the constant paths $c_\x$ for every $\x \in L_0 \cap L_1$. Recall that $\pi_2(\x, \y)$ denotes the set of relative homotopy classes of Whitney disks from $\x$ to $\y$, i.e., relative homotopy classes of paths from $c_\x$ to $c_\y$ in $\Path(L_0, L_1)$. There is a concatenation operation:
$$ *: \pi_2(\x, \y) \times \pi_2(\y, \z) \to \pi_2(\x, \z).$$

In case $\Path(L_0, L_1)$ is disconnected, some of the sets $\pi_2(\x, \y)$ could be empty. We define an equivalence relation on $L_0 \cap L_1$ by
\begin{equation}
\label{eq:equiv}
 \x \sim \y \iff \pi_2(\x, \y) \neq \emptyset.
 \end{equation}
We let $\Set$ be the set of equivalence classes. Then, the Lagrangian Floer complex splits into a direct sum
$$ \CF_*(L_0, L_1) = \bigoplus_{\s \in \Set} \CF_*(L_0, L_1, \s),$$
where 
$$\CF_*(L_0, L_1, \s) = \bigoplus_{\x \in \s} o(\x)$$
with the differential as in \eqref{eq:del}.

\subsection{Twisted coefficients}
\label{sec:twisted}
There is a well-known variant of Lagrangian Floer homology, which uses twisted coefficients. In the context of Heegaard Floer theory, this was introduced by Ozsv\'ath and Szab\'o in \cite[Section 8]{HolDiskTwo}. We present it here in a more general framework.

We keep the set-up from Sections~\ref{sec:LFH} and \ref{sec:decompose}. Fix $\s \in \Set$ and a base intersection point $\x_0 \in \s$. Let 
$$G := \pi_1( \Path(L_0, L_1), c_{\x_0}) = \pi_2(\x_0, \x_0).$$ 
We denote the elements of the group ring $\Z[G]$ by $e^\psi$, where $\psi \in G$. 

\begin{definition}
\label{def:complete}
A {\em complete set of paths for $\s$} is a choice of relative homotopy classes $\theta_\x \in \pi_2(\x_0, \x)$, one for each $\x \in \s$.
\end{definition}

Let $A$ be a $\Z[G]$-module, and fix a complete set of paths $\{\theta_\x\}$ for $\s$. Then, for any $\x, \y \in \s$, we get an identification 
\begin{equation}
\label{eq:idG}
\pi_2(\x, \y) \xrightarrow{\cong} G, \ \ \ \phi \mapsto \theta_\x * \phi *\theta_\y^{-1}.
\end{equation}
Using this identification, for every $\phi \in \pi_2(\x, \y)$, we can make sense of the action of an element $e^\phi \in \Z[\pi_2(\x, \y)]$ on $A$.

 We define the Lagrangian Floer complex with coefficients in $A$ by
$$\CFtw_*(L_0, L_1, \s; A) = \bigoplus_{\x \in \s} o(\x) \otimes_{\Z}  A,$$
with the differential
$$ \del (\omega_\x \otimes a)  = \sum_{\y \in \s} \sum_{\substack{\phi \in \pi_2(\x, \y)\\ \mu(\phi)=1}} \bigl( \# \Mhat(\phi) \bigr) \cdot \omega_\y \otimes (e^{\phi} \cdot a).$$
The moduli spaces $\Mhat(\phi)$ are oriented just as in Section~\ref{sec:LFH}. The resulting Floer homology with twisted coefficients is denoted $\HFtw_*(L_0, L_1, \s; A)$.

\begin{remark} \label{rem:canonical_twisted}
A more canonical way of defining Floer complexes with twisted coefficients in $A$, without choosing a complete set of paths (but still choosing the basepoint $\x_0$), is to set
$$\CFtwcan_*(L_0, L_1, \s; A) = \bigoplus_{\x \in \s} o(\x) \otimes_{\Z} \bigl( \Z[\pi_2(\x_0, \x)] \otimes_{\Z[G]} A\bigr),$$
where $ \Z[\pi_2(\x_0, \x)]$ is the free abelian group generated by $\pi_2(\x_0, \x)$. The group ring $\Z[G]=\Z[\pi_2(\x_0, \x_0)]$ acts on $\Z[\pi_2(\x_0, \x)]$ via concatenation of paths. We denote a typical generator of $\CFtwcan_*(L_0, L_1, \s; A)$ by $\omega_\x \otimes (e^\psi \otimes a)$, where $\omega_\x \in o(\x),$ $\psi \in \pi_2(\x_0, \x)$, and $a \in A.$ Then, the differential on $\CFtw_*(L_0, L_1, \s; A)$ is given by
$$ \del (\omega_\x \otimes (e^\psi \otimes a))  = \sum_{\y \in \s} \sum_{\substack{\phi \in \pi_2(\x, \y)\\ \mu(\phi)=1}} \bigl( \# \Mhat(\phi) \bigr) \cdot \omega_\y \otimes (e^{\psi * \phi} \otimes a).$$
The complete set of paths determines an isomorphism between $\Z[\pi_2(\x_0, \x)]$ and $\Z[G]$, and hence between $\CFtwcan_*(L_0, L_1, \s; A)$ and $\CFtw_*(L_0, L_1, \s; A)$. In particular, this shows that the complexes $\CFtw_*(L_0, L_1, \s; A)$, for different complete sets of paths, are all isomorphic. 
\end{remark}

\begin{remark}
When $A=\Z$ is the trivial $\Z[G]$-module, the complex $\CFtw_*(L_0, L_1, \s; \Z)$ becomes the usual Lagrangian Floer complex. On the other hand, $\CFtwcan_*(L_0, L_1, \s; \Z)$ is only non-canonically isomorphic to $\CF_*(L_0, L_1, \s; \Z)$; the isomorphism depends on the complete sets of paths. For our purposes, it is clearer to just choose such a complete set of paths from the very beginning, and work with the non-canonical complexes $ \CFtw_*(L_0, L_1, \s; A)$.
\end{remark}

We will mostly be interested in the following kind of modules $A$. Let
$$ (p_0, p_1): \Path(L_0, L_1) \to L_0 \times L_1$$
be the map taking a path to its endpoints. Given $\eta \in H^1(L_0; \Z/2)$ and $\zeta \in H^1(L_1; \Z/2)$, the pull-backs $p_0^*\eta$ and $p_1^*\zeta$ are elements of $H^1(\Path(L_0, L_1); \Z/2)$ or, equivalently, morphisms $G \to \Z/2$. We let $A_{\eta, \zeta}$ be the $\Z[G]$-module which is $\Z$ as an abelian group, and such that elements $e^\psi \in \Z[G]$ act on $A_{\eta, \zeta}$ by multiplication with 
\begin{equation}
\label{eq:Metazeta}
(-1)^{p_0^*\eta(\psi)} \cdot (-1)^{p_1^*\zeta(\psi)}\in \{\pm 1\}.
\end{equation}

We can define these modules for any $\s \in \Set$, so we could combine the resulting Floer complexes into one:
$$\CFtw_*(L_0, L_1; A_{\eta, \zeta}) = \bigoplus_{\s \in \Set}  \CFtw_*(L_0, L_1, \s; A_{\eta, \zeta}).$$
These twisted coefficients govern the changes in Pin structures on the Lagrangians.

\begin{proposition}
\label{prop:changePin}
Fix $\s \in \Set$. For $\eta \in H^1(L_0; \Z/2)$, let $L_0^\eta$ denote the Lagrangian $L_0$ with its Pin structure changed from $P_0^\#$ to $P_0^\# \otimes \eta.$ Similarly, for $\zeta \in H^1(L_1; \Z/2)$, let $L_1^\zeta$ denote $L_1$ with its Pin structure changed from $P_1^\#$ to $P_1^\# \otimes \zeta.$ Then, we have a canonical isomorphism of Floer complexes
$$ \CF_*(L_0^\eta, L_1^\zeta, \s) \cong \CFtw_*(L_0, L_1, \s; A_{\eta, \zeta}).$$
\end{proposition}

\begin{proof}
For simplicity, let us just consider a change in the Pin structure on $L_0$; i.e., we assume $\zeta=0$, and we denote $A_{\eta, 0}$ by $A_\eta$. (The effect of a change on the Pin structure on $L_1$ can be analyzed in a similar way.)

Recall that we have a base point $\x_0 \in \s$. We fix an isomorphism between the restrictions of the Pin structures $P_0^\#$ and $P_0^\# \otimes \eta$ to $T_{\x_0}L_0$. 

Recall that
\begin{equation}
\label{eq:CFtwisted}
  \CFtw_*(L_0, L_1, \s) = \bigoplus_{\x \in L_0 \cap L_1} o(\x) \otimes_{\Z} A_{\eta}.
  \end{equation}
On the other hand, 
\begin{equation}
\label{eq:CFeta}
 \CF_*(L_0^\eta, L_1, \s) = \bigoplus_{\x \in L_0 \cap L_1} o_{\eta}(\x), 
 \end{equation}
where $o_{\eta}(\x)$ is the orientation space at $\x$ coming from the determinant index bundle $\det(\delbar_\x^\eta)$ associated to the new Pin structure $P_0^\# \otimes \eta.$ To construct an isomorphism from the chain complex \eqref{eq:CFtwisted} to \eqref{eq:CFeta}, we need isomorphisms
\begin{equation}
\label{eq:Psix}
 \Psi_\x : o(\x) \otimes_{\Z} A_{\eta} \to o_{\eta}(\x)
 \end{equation}
for all $\x \in \Set$. 

 Let $\omega_\x \in o(\x)$,  $a \in A_\eta$, and $\omega^\eta_\x \in o_{\eta}(\x)$. Note that $a$ is just an integer, although the elements of $\Z[G]$ act on it nontrivially. We set
\begin{equation}
\label{eq:Psixdef}
 \Psi_\x(\omega_\x \otimes a) = \sigma_\x \cdot a \cdot \omega^\eta_\x,
 \end{equation}
with the sign $\sigma_\x \in \{\pm 1\}$ is determined as follows. As part of the complete set of paths, we have a class $\theta_\x \in \pi_2(\x_0, \x)$. Let $D \delbar(\theta_\x)$ be the linearized Cauchy-Riemann operator for a path in that class. By Lemma~\ref{lemma:Dd}, we have isomorphisms: 
\begin{align}
\label{eq:ddd}
\det(\delbar_{\x_0}) \otimes \det(D \delbar(\theta_\x))  &\cong \det(\delbar_\x),  \\
 \det(\delbar_{\x_0}^\eta) \otimes \det(D \delbar(\theta_\x))  &\cong  \det(\delbar_\x^\eta). \notag
\end{align}
Our identification of the Pin structures on $T_{\x_0}L_0$ gives an isomorphism between $\det(\delbar_{\x_0})$ and $ \det(\delbar_{\x_0}^\eta)$. We obtain an isomorphism
\begin{equation}
\label{eq:delbarxx}
 \det(\delbar_\x) \cong \det(\delbar_\x^\eta).
 \end{equation}
We let the sign $\sigma_\x$ in \eqref{eq:Psixdef} be $+1$ if $\omega_\x$ gets taken to $\omega^\eta_\x$ under this isomorphism, and we let it be $-1$ otherwise.

Let us check that the  isomorphisms $\Psi_\x$ commute with the chain complex differentials. Consider the contribution of a generator $\omega_\y$ to $\del \omega_\x$ in  $\CFtw_*(L_0, L_1, \s)$ coming from moduli spaces in a class $\phi \in \pi_2(\x, \y)$, and compare it to the contribution of some $\omega^\eta_\y$ to $\del \omega^\eta_\x$ in $\CF_*(L_0^\eta, L_1, \s)$ from the same $\phi$. Let $\psi=\theta_\x * \phi *\theta_\y^{-1}  \in \pi_2(\x_0, \x_0)$ be the class that corresponds to $\phi$ under the identification \eqref{eq:idG}. Because the differential on  $\CFtw_*(L_0, L_1, \s)$ involves twisted coefficients, it picks up a sign of $(-1)^{p_0^*\eta(\psi)}$. Thus, what we need to check is that the diagram
$$
\xymatrixcolsep{3cm}
 \xymatrix{
\det(\delbar_\x)\otimes \det(D\delbar(\phi)) \ar[r]^\cong \ar[d]_\cong & \det(\delbar_\x^\eta) \otimes \det(D\delbar(\phi))\ar[d]^\cong \\
\det(\delbar_\y)  \ar[r]^\cong & \det(\delbar_\y^\eta) }$$
commutes up to the sign $(-1)^{p_0^*\eta(\psi)}$. In this diagram, the top horizontal isomorphism is the one in \eqref{eq:delbarxx}, using the class $\theta_\x$; the bottom horizontal isomorphism is its analogue using the class $\theta_\y$, tensored with the identity on the $ \det(D\delbar(\phi))$ factors;  the vertical isomorphisms are from Lemma~\ref{lemma:Dd} and use the class $\phi$.

Going back to the definition of the horizontal isomorphisms from \eqref{eq:ddd}, we see that we can consider instead the diagram
\begin{equation}
\label{eq:xym}
\xymatrixcolsep{3cm}
\xymatrix{
\det(\delbar_{\x_0})  \otimes \det(D\delbar(\psi))\ar[r]^\cong \ar[d]_\cong & \det(\delbar_{\x_0}^\eta) \otimes \det(D\delbar(\psi))\ar[d]^\cong\\
\det(\delbar_{\x_0}) \ar[r]^\cong &\det(\delbar_{\x_0}^\eta) }
\end{equation}
where the horizontal maps come from the identification of the Pin structures at $\x_0$, and the vertical isomorphisms use Lemma~\ref{lemma:Dd} for the class $\psi$. We are left to show that \eqref{eq:xym} commutes up to the sign $(-1)^{p_0^*\eta(\psi)}$.

The left vertical isomorphism in \eqref{eq:xym} is induced by the original Pin structures on the Lagrangians, and the right one by the new Pin structures (where the structure on $L_0$ is tensored with $\eta$). The construction of the isomorphisms in Lemma~\ref{lemma:Dd} is based on identifying the orientation spaces $o(\y)$ (where in our case $\y = \x_0$) coming from $\delbar_{\y}^{\new}$ and $\delbar_{\y}$. This relies on Proposition~\ref{prop:epsL}, which identifies orientation spaces associated to paths in  $\Grd(V)$ with fixed endpoints and fixed index. In our setting, we compare two such paths from $T_{\x_0} L_0$ to $T_{\x_0} L_1$. These differ by concatenating with a loop from $T_{\x_0} L_0$ to itself, and whether this loop is nontrivial in $\pi_1(\Grd(T_{\x_0}L_0)$ is determined by the sign $(-1)^{p_0^*\eta(\psi)}$. If the sign is positive, the two paths are homotopic and produce the same isomorphism, so the diagram \eqref{eq:xym} commutes. If the sign is negative, Lemma~\ref{lem:nontrivial} shows that the isomorphisms differ by $-1$, so the diagram \eqref{eq:xym} anti-commutes.

We have now shown that the formula \eqref{eq:Psix} produces an isomorphism of chain complexes. 
Recall that we started our construction by choosing an isomorphism between the restrictions of the Pin structures $P_0^\#$ and $P_0^\# \otimes \eta$ to $T_{\x_0}L_0$. There is a connected space of such isomorphisms, so this choice will not affect the canonical nature of the final isomorphism between chain complexes, which is a discrete object.
\end{proof}

\begin{corollary}
\label{cor:changePin}
Let $\s \in \Set$ and $\alpha \in H^1(M; \Z/2)$. Consider the inclusions $\iota_0: L_0 \hookrightarrow M$ and $\iota_1: L_1 \hookrightarrow M$. For $i \in \{0, 1\}$, let $L_i^\alpha$ denote the Lagrangian $L_i$ with its Pin structure changed by tensoring with $\iota_i^*\alpha$. Then, there is a canonical isomorphism:
$$ \CF_*(L_0, L_1, \s) \cong \CF_*(L_0^\alpha, L_1^\alpha, \s).$$
\end{corollary}

\begin{proof}
This is a consequence of Proposition~\ref{prop:changePin}, by taking $\eta = \iota_0^*\alpha$ and $\zeta =\iota_1^*\alpha$. Indeed, note that in the formula \eqref{eq:Metazeta} we have $\pi_0^*\eta(\psi)=\pi_1^*\zeta(\psi)$, so the module $A_{\eta, \zeta}$ is $\Z$ with the trivial $\Z[G]$-action.
\end{proof}

Corollary~\ref{cor:changePin} shows that orientations of the moduli spaces in Floer theory are induced by Pin structures on the Lagrangians modulo those coming from the ambient space; that is, by an affine space on the cokernel of the map
$$H^1(M; \Z/2) \to H^1(L_0; \Z/2) \oplus H^1(L_1; \Z/2).$$

\subsection{Polygon maps}
\label{sec:polygon}
Our discussion of orientations on moduli spaces extends to holomorphic polygons, much as in Seidel's book \cite{SeidelBook}. We state the results here, and leave the proofs to the reader.

For $m \geq 2$, suppose we have monotone Lagrangians $L_0, L_1, \dots, L_m \subset M$, equipped with Pin structures. We also assume that we are given coupled orientations on the pairs $(L_0, L_1)$, $(L_1, L_2)$, $\dots$, $(L_{m-1}, L_m).$ These induce a coupled orientation on $(L_0, L_m)$. We define a map
\begin{equation}
\label{eq:polygonmap}
f_{L_0, \dots, L_m}: \CF_*(L_0, L_1) \otimes \CF_*(L_1, L_2) \otimes \dots \otimes \CF_*(L_{m-1}, L_m) \to \CF_*(L_0, L_m)
\end{equation}
by
$$ f_{L_0, \dots, L_m}(\omega_{\x_1} \otimes \omega_{\x_2} \otimes \dots  \otimes \omega_{\x_m}) =  \sum_{\y \in L_0 \cap L_m} \sum_{\substack{\phi \in \pi_2(\x_1,\x_2, \dots, \x_m, \y)\\ \mu(\phi)=0}} \bigl( \# \M(\phi) \bigr) \cdot \omega_\y.$$
Here, $\x_i \in L_{i-1} \cap L_i$ for $i=1, \dots, m$, and $\y \in L_0 \cap L_m$. Also, $\omega_{\x_i} \in o(\x_i)$ and $\omega_\y \in o(\y)$ are orientations for the respective determinant lines. The notation $\pi_2(\x_1,\x_2, \dots, \x_m, \y)$ indicates the set of (relative) homotopy classes of $(k+1)$-gons in $M$ with boundaries on $L_0, L_1, \dots, L_m$ and vertices at $\x_1, \x_2, \dots, \x_m, \y$, in clockwise order. (See Conventions~\ref{conv}.) The moduli space $\M(\phi)$ consists of holomorphic polygon maps in the class $\phi$, of index $\mu(\phi)=0$. Its orientation is induced from the Pin structures and coupled orientations, using a straightforward generalization of  Lemma~\ref{lemma:Dd}. 

When $m=1$, we let $f_{L_0, L_1}$ denote the differential $\del$ on the complex $\CF_*(L_0, L_1)$.

Let us adjust the signs to define
$$ \tilde f_{L_0, \dots, L_m}(\omega_{\x_1} \otimes  \dots  \otimes \omega_{\x_m})= (-1)^{\gr(\x_m) + 2\gr(\x_{m-1}) + \dots + m\gr(\x_1)} f_{L_0, \dots, L_m}(\omega_{\x_1} \otimes  \dots  \otimes \omega_{\x_m}).$$
Then, these new polygon maps satisfy the $A_{\infty}$ relations:
$$ \sum_{0 \leq i < j \leq m} (-1)^{*} \tilde f_{L_0, \dots, L_i, L_j, \dots, L_m}\bigl(\omega_{\x_1} \otimes \dots \otimes \omega_{\x_i} \otimes \tilde f_{L_i, L_{i+1}, \dots, L_j} (\omega_{\x_{i+1}} \otimes \dots \otimes \omega_{\x_j}) \otimes \omega_{\x_{j+1}} \otimes \dots \otimes \omega_{\x_m}\bigr ) =0$$
where $* = \gr(\x_{j+1}) + \dots + \gr(\x_m) + (m-j)$. Note that the signs are slightly different from those in \cite{SeidelBook}, because we use homological rather than cohomological conventions.

When $m=1$, the $A_{\infty}$ relation says that $\del^2=0$ on $\CF_*(L_0, L_1)$. When $m=2$, it says that 
$$f_{L_0, L_1, L_2}: \CF_*(L_0, L_1) \otimes \CF_*(L_1, L_2) \to \CF_*(L_0, L_2)$$ 
is a chain map. Hence, it induces a product map on Floer homology:
\begin{equation}
\label{eq:productmap}
 F_{L_0, L_1, L_2} : \HF_*(L_0, L_1) \otimes \HF_*(L_1, L_2) \to \HF_*(L_0, L_2).
 \end{equation}

In Section~\ref{sec:decompose} we decomposed the Floer complexes according to equivalence classes of intersection points. Pick such equivalence classes $\s_i \subset  L_{i-1} \cap L_i$ for $i=1, \dots, m$, and $\t \subset L_0 \cap L_m$. We then define an equivalence relation on tuples $(\x_1, \dots, \x_m, \y, \phi)$ with $\x_i \in \s_i$, $\y \in \t$, and $\phi \in \pi_2(\x_1, \dots, \x_m, \y)$, where 
$$ (\x_1, \dots, \x_m, \y, \phi) \sim (\x'_1, \dots, \x'_m, \y', \phi')$$
if $\phi'$ differs from $\phi$ by concatenation with classes in $\pi_2(\x_1, \x_1'), \dots, \pi_2(\x_m, \x'_m)$ and $\pi_2(\y, \y')$. We let $\Set$ be the set of such equivalence classes. For every $\s \in \Set$, we have a polygon map
$$f^\s_{L_0, \dots, L_m}: \CF_*(L_0, L_1, \s_1)  \otimes \dots \otimes \CF_*(L_{m-1}, L_m, \s_m) \to \CF_*(L_0, L_m, \t)$$
which counts only polygons in classes $\phi$ such that $[(\x_1, \dots, \x_m, \y, \phi)] =\s \in \Set$. The map \eqref{eq:polygonmap} splits as the sum of these maps, over all $\s \in \Set$.

There is also a variant of the polygon maps with twisted coefficients; compare \cite[Section 8.2]{HolDiskTwo}.  It can be described using bimodule tensor products in the formalism of Remark \ref{rem:canonical_twisted}, but we present instead the explicit definition: Fix base intersection points $\x_{i, 0} \in \s_i$ for $i=1, \dots, m$, as well as $\y_0 \in \t$. 
Let 
$$ G_i = \pi_1(\Path(L_{i-1}, L_i), c_{\x_{i, 0}}) = \pi_2(\x_{i, 0}, \x_{i, 0}), \ i=1, \dots, m$$
and let also $$G = \pi_1(\Path(L_0, L_m), c_{\y_0}) = \pi_2(\y_0, \y_0).$$ Denote by $\bSet=\pi_2(\x_{1,0}, \dots, \x_{m,0}, \y_0)$. Observe that $\Set$ is the quotient of $\bSet$ by the actions of all the groups $G_i$ and $G$; for $\s \in \Set$, we let $\bSet(\s) \subset \bSet$ be its equivalence class.  

Fix also complete sets of paths for all $\s_i$ and $\t$, as in Definition~\ref{def:complete}. These induce identifications between $\bSet$ and $ \pi_2(\x_1, \dots, \x_m, \y)$, for all $\x_i \in \s_i$ and $\y \in \t$. Given $\phi \in \pi_2(\x_1, \dots, \x_m, \y)$, we let $\bs(\phi)$ denote the corresponding element in $\bSet$.

Fix $\s \in \Set$. Suppose we have modules $A_i$ over $\Z[G_i]$. These induce a module $A^\s$ over $\Z[G]$ given by
\begin{equation}
\label{eq:As}
 A^\s = \frac{\{ (a_1, \dots, a_m, \bs) \in A_1 \times \dots \times A_m \times \bSet(\s) \}}{ (a_1, \dots, a_m, \bs) \sim (e^{\psi_1}a_1, \dots, e^{\psi_m} a_m, ({\psi_1} \times \dots \times {\psi_m}) \bs) } 
 \end{equation}
where $\psi_i$ are arbitrary elements of $G_i$ for $i=1, \dots, m$. For $\psi \in G$, the element $e^\psi \in \Z[G]$ acts on $[(a_1,  \dots, a_m, \bs)]$ by taking it to $[(a_1,  \dots, a_m, \psi \bs)]$.

We get a polygon map
\begin{equation}
\label{eq:polytwisted}
\ftw^\s_{L_0, \dots, L_m}: \CFtw_*(L_0, L_1, \s_1; A_1) \otimes  \dots \otimes \CFtw_*(L_{m-1}, L_m, \s_m; A_m) \to \CFtw_*(L_0, L_m, \t; A^\s)
\end{equation}
given by
\begin{multline*}
\ftw^\s_{L_0, \dots, L_m} \bigl(  (\omega_{\x_1} \otimes   a_1\bigr) \otimes \dots \otimes  (\omega_{\x_m} \otimes a_m) \bigr) =\\
\sum_{\y \in \t} \sum_{\bs \in \bSet(\s)} \sum_{\substack{\phi \in \pi_2(\x_1,\x_2, \dots, \x_m, \y) \\ [(\x_1, \dots, \x_m, \y, \phi)] =\s \\ \mu(\phi)=0}} \bigl( \# \M(\phi) \bigr) \cdot \bigl(\omega_\y \otimes  [(a_1, \dots, a_m, \bs(\phi))]\bigr).
\end{multline*}

These polygon maps satisfy $A_{\infty}$ relations just like the untwisted ones. In particular, the triangle maps give chain maps and hence maps on the Floer homology with twisted coefficients.

As in Section~\ref{sec:twisted}, we are particularly interested in modules that govern the change in Pin structures. Pick elements $\eta_i \in H^1(L_i; \Z/2)$ for $i=0, \dots, m$. From these we obtain modules $A_i : = A_{\eta_{i-1}, \eta_i}$ over $\Z[G_i]$ for $i=1, \dots, m$, as well as $A_{\eta_0, \eta_m}$ over $\Z[G]$. For $\s \in \Set$, let $A^\s$ be as in \eqref{eq:As}. 

Pick any (not necessarily holomorphic) polygon in a class in $\bSet(\s) \subset \pi_2(\x_{1,0}, \dots, \x_{m,0}, \y_0)$, and let $\Gamma$ be its boundary. This consists of paths $\gamma_i$ on $L_i$ from $\x_{i,0}$ to $\x_{i+1, 0}$ for $i=1, \dots, m-1$, as well as $\gamma_0$ on $L_0$ from $\y_0$ to $\x_{1, 0}$ and $\gamma_m$ on $L_m$ from $\x_{m, 0}$ to $\y_0$. Since the class $\s$ is fixed, the homotopy class of $\Gamma$ (relative to its vertices) is fixed up to concatenation with boundaries of disks in $G_i = \pi_2(\x_{i, 0}, \x_{i, 0})$ and $G = \pi_2(\y_0, \y_0)$.

For $i=0, \dots, m$, restriction to the part of the boundary in $L_i$ produces maps $q_i$ from $\bSet=\pi_2(\x_{1,0}, \dots, \x_{m,0}, \y_0)$ to the set of relative homotopy classes of paths between the two relevant basepoints on $L_i$. The latter space can be identified with $\pi_1(L_i)$ by concatenating with the path $\gamma_i$. Viewing the class $\eta_i$ as a morphism $\pi_1(L_i) \to \Z/2$, we obtain a map
$$ \eta_i \circ q_i : \bSet \to \Z/2.$$ 

We now define a map
\begin{equation}
\label{eq:AsA}
 A^\s \to A_{\eta_0, \eta_m}, \ \ [(a_1, a_2, \dots, a_m, \bs)] \to a_1a_2 \cdots a_m \cdot (-1)^{\sum_{i=0}^m (\eta_i \circ q_i)(\bs)}.
 \end{equation}
 It is straightforward to check that this is well-defined, and a morphism of $\Z[G]$-modules. 
  
The morphism \eqref{eq:AsA} induces a map $\CFtw_*(L_0, L_m, \t; A^\s) \to \CFtw_*(L_0, L_m, \t; A_{\eta_0, \eta_m}).$ Composing the polygon maps in \eqref{eq:polytwisted} with this yields a map
\[ \ftw^\s_{L_0, \dots, L_m}: \CFtw_*(L_0, L_1, \s_1; A_{\eta_0, \eta_1}) \otimes  \dots \otimes \CFtw_*(L_{m-1}, L_m, \s_m; A_{\eta_{m-1}, \eta_m}) \to \CFtw_*(L_0, L_m, \t; A_{\eta_0, \eta_m}).
\]
\begin{proposition}
For $i=0, \dots, m$, let $L_i^{\eta_i}$ denote $L_i$ with its Pin structure changed by tensoring with $\eta_i$.  Under the identifications from Proposition~\ref{prop:changePin}, the map  $\ftw^\s_{L_0, \dots, L_m}$ corresponds to the ordinary polygon map
$$ f^\s_{L_0^{\eta_0}, \dots, L_m^{\eta_m}}: \CF_*(L_0^{\eta_0}, L_1^{\eta_1},\s_1) \otimes \dots \otimes \CF_*(L_{m-1}^{\eta_{m-1}}, L_m^{\eta_m}, \s_m) \to \CF_*(L_0^{\eta_0}, L_m^{\eta_m}, \t).$$
\end{proposition}

The proof is routine (along the same lines as that of Proposition~\ref{prop:changePin}), so we omit it.

\subsection{Examples}
We give a few simple examples of sign counts of polygons, starting with the case when $M$ is a surface, and the Lagrangians are curves. Recall from Definition~\ref{def:S1Pin} that in this case, there are canonical Lie group Pin structures on the Lagrangians (not just up to isomorphism). 

If $V$ is a one-dimensional vector space, we denote by $|V|$ the abelian group generated by orientations $\omega$ on $V$, modulo the relation $\bar \omega = -\omega$.

\begin{lemma}
  \label{lem:dim2}
Let $M$ be a two-dimensional symplectic manifold, and $L_0$, $L_1 \subset M$ be curves equipped with their Lie group Pin structures and a coupled orientation. Let $\x \in L_0 \cap L_1$ be a transverse intersection point.

(a) If $\gr(\x) = 1 \in \Z/2$, then the orientation space $o(\x)$ from \eqref{eq:ox} has a canonical trivialization.

(b) If $\gr(\x)=0$, then $o(\x)$ is canonically identified with $|T_\x L_0|$.
\end{lemma}

\begin{proof}
(a) The condition $\gr(\x) = 1$ means that the orientation on $T_\x L_0 \oplus T_\x L_1$ agrees with the one on $T_\x M$. Up to isotopy, there is a unique linear isomorphism from $T_\x M$ to $\C$ taking $T_\x L_0$ to $\R$ and $T_\x L_1$ to $i\R$ (respecting the coupled orientation). In view of Remark~\ref{rem:conv}, the $\delbar$ operator on the cap $H$ with boundary conditions interpolating from $\R$ to $i\R$ via $e^{i\theta}\R$, $\theta \in [0, \pi/2]$ is equivalent to the one discussed in \cite[proof of Lemma 11.11]{SeidelBook}, where it is shown to be bijective.  Hence, its determinant index bundle can be trivialized by picking the positive generator. 

(b) If $\gr(\x)=0$, then the orientation on $T_\x L_0 \oplus T_\x L_1$ disagrees with the one on $T_\x M$. We can again identify $o(\x)$ with a standard orientation space, this time using Lagrangians interpolating from $i\R$ to $\R$. This is the reverse of the path considered in (a). By Equation (11.27) in \cite{SeidelBook}, the orientation space differs from the one in part (a) by tensoring with $|T_{\x}L_0|$.
\end{proof}

\begin{remark}
\label{rem:gensdim2}
In view of Lemma~\ref{lem:dim2}, when $\dim(M)=2$ and the Lagrangians have the Lie groups Pin structures, choosing an orientation on $L_0$ (or, equivalently, on $L_1$, in a way compatible with the coupled orientation) allows us to trivialize $o(\x)$ for all intersection points $\x \in L_0 \cap L_1$. We can then view the Lagrangian Floer complex $\CF_*(L_0, L_1)$ as generated by the intersection points: we identify each $\x$ with the positive generator of $o(\x)\cong \R$. 
\end{remark}

\begin{example}
\label{ex:circles}
We review Example 13.5 in \cite{SeidelBook}, using our Convention~\ref{conv}. Let $M \cong S^1 \times [0,1]$ be an open annulus, which we view as the complement of a small disk $D^2(\epsilon)$ in $\C$.  Let $\alpha$ be the unit circle, and $\beta$ a Hamiltonian translate of $\alpha$ intersecting it in two points, as in Figure~\ref{fig:circles}. Equip $\alpha$ and $\beta$ with the Lie group Pin structures. Pick also counterclockwise orientations on $\alpha$ and $\beta$ as in the figure. The resulting coupled orientation on $(\alpha, \beta)$ gives $\gr(\x) = 1$ and $\gr(\y) =0$.

Following Remark~\ref{rem:gensdim2}, the Floer complex $\CF_*(\alpha, \beta)$ is generated by the intersection points $\x$ and $\y$. We have
$$\del \x = \pm \y \pm \y,$$
where the two signs come from the bigons $A$ and $B$ in Figure~\ref{fig:circles}. 
By Proposition~\ref{prop:PSS}, we know that $\HF_*(\alpha, \beta) \cong H_*(S^1)$, so we must have $\del \x=0$. Thus, the two bigons $A$ and $B$ must come with opposite signs. In fact, it is computed in \cite[Example 13.5]{SeidelBook} that $A$ contributes $+1$ and $B$ contributes $-1$.
\end{example}

\begin{figure}
{
\fontsize{10pt}{11pt}\selectfont
   \def\svgwidth{2.3in}
   %% Creator: Inkscape 1.3.2 (091e20e, 2023-11-25), www.inkscape.org
%% Accompanies image file 'circles.pdf' (pdf, eps, ps)
\begingroup%
  \makeatletter%
  \providecommand\color[2][]{%
    \errmessage{(Inkscape) Color is used for the text in Inkscape, but the package 'color.sty' is not loaded}%
    \renewcommand\color[2][]{}%
  }%
  \providecommand\transparent[1]{%
    \errmessage{(Inkscape) Transparency is used (non-zero) for the text in Inkscape, but the package 'transparent.sty' is not loaded}%
    \renewcommand\transparent[1]{}%
  }%
  \providecommand\rotatebox[2]{#2}%
  \newcommand*\fsize{\dimexpr\f@size pt\relax}%
  \newcommand*\lineheight[1]{\fontsize{\fsize}{#1\fsize}\selectfont}%
  \ifx\svgwidth\undefined%
    \setlength{\unitlength}{152.70890123bp}%
    \ifx\svgscale\undefined%
      \relax%
    \else%
      \setlength{\unitlength}{\unitlength * \real{\svgscale}}%
    \fi%
  \else%
    \setlength{\unitlength}{\svgwidth}%
  \fi%
  \global\let\svgwidth\undefined%
  \global\let\svgscale\undefined%
  \makeatother%
  \begin{picture}(1,0.72403826)%
    \lineheight{1}%
    \setlength\tabcolsep{0pt}%
    \put(0,0){\includegraphics[width=\unitlength,page=1]{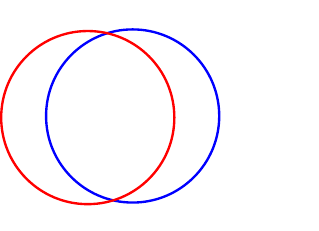}}%
    \put(0.67058197,0.5450462){\color[rgb]{0,0,1}\makebox(0,0)[lt]{\lineheight{1.25}\smash{\begin{tabular}[t]{l}$\beta$\end{tabular}}}}%
    \put(0.0040269,0.5450462){\color[rgb]{1,0,0}\makebox(0,0)[lt]{\lineheight{1.25}\smash{\begin{tabular}[t]{l}$\alpha$\end{tabular}}}}%
    \put(0,0){\includegraphics[width=\unitlength,page=2]{circles.pdf}}%
    \put(0.04474675,0.33157853){\makebox(0,0)[lt]{\lineheight{1.25}\smash{\begin{tabular}[t]{l}$A$\end{tabular}}}}%
    \put(0.57924927,0.33157853){\makebox(0,0)[lt]{\lineheight{1.25}\smash{\begin{tabular}[t]{l}$B$\end{tabular}}}}%
    \put(0.328033,0.65852876){\makebox(0,0)[lt]{\lineheight{1.25}\smash{\begin{tabular}[t]{l}$\x$\end{tabular}}}}%
    \put(0.3155611,0.01712554){\makebox(0,0)[lt]{\lineheight{1.25}\smash{\begin{tabular}[t]{l}$\y$\end{tabular}}}}%
    \put(0,0){\includegraphics[width=\unitlength,page=3]{circles.pdf}}%
  \end{picture}%
\endgroup%

}
\caption{Two circles in an annulus.}
\label{fig:circles}
\end{figure}

\begin{example}
\label{ex:abc}
Let $M$ be any surface, and consider three curves $\alpha$, $\beta$, $\gamma \subset M$, having a triple intersection point $\x$. Suppose $\alpha$, $\beta$, $\gamma$ appear in this cyclic order around $\x$, counterclockwise. We equip them with the Lie group Pin structures. There is a trivial $\alpha$-$\beta$-$\gamma$ triangle of index $0$ that contributes to the polygon map
$$ f_{\alpha, \beta, \gamma} : \CF_*(\alpha, \beta) \otimes \CF_*(\beta, \gamma) \to \CF_*(\alpha, \gamma).$$
Suppose we choose orientations on $\alpha$, $\beta$, $\gamma$ either as in Figure~\ref{fig:triangles} (a) or all three opposite to those shown there. This is equivalent to asking that the induced coupled orientations on each pair $(\alpha, \beta)$, $(\beta, \gamma)$ and $(\alpha, \gamma)$ are such that $\gr(\x)=1 \in \Z/2$ in all three Floer complexes. Then, the triangle map takes the form
$$ f_{\alpha, \beta, \gamma}(\x \otimes \x) = \pm \x + \cdots$$
where $\pm \x$ is the contribution of the trivial triangle. We claim that its sign is positive. This can be checked explicitly by following the definitions. Alternatively, we can determine the sign from a specific example: embed the picture into one with three circles in the annulus, all Hamiltonian translates of the other. Then, under the PSS isomorphism, the generators $\x$ become identified with $[S^1] \in H_*(S^1)$, and the triangle map becomes the intersection product on homology:
$$H_*(S^1) \otimes H_*(S^1) \to H_*(S^1).$$
See for example \cite{AbouzaidPlumbings} (where this is stated in terms of cohomology, with the triangle map becoming the cup product) or \cite{Zapolsky}. The intersection of $[S^1]$ and $[S^1]$ is $[S^1]$, with a positive sign.
\end{example}
\begin{figure}
{
\fontsize{10pt}{11pt}\selectfont
   \def\svgwidth{5.7in}
   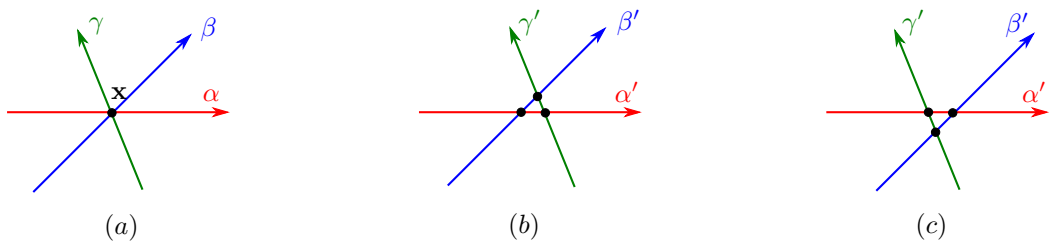
}
\caption{(a) A trivial triangle at a triple intersection. (b), (c): Its perturbations.}
\label{fig:triangles}
\end{figure}

\begin{example}
\label{ex:abcp}
If $\alpha$, $\beta$, $\gamma$ are as in Example~\ref{ex:abc}, we can slightly perturb them to get a triple $\alpha'$, $\beta'$, $\gamma'$ as in Figure~\ref{fig:triangles} (b) or (c). The small $\alpha'$-$\beta'$-$\gamma'$ triangle between the curves is a deformation of the trivial triangle from (a), so it also contributes positively to the polygon map. Note that the two triangles in Figure~\ref{fig:triangles}(b) and (c) are, in fact, isomorphic. They differ by a $180^\circ$ rotation, which reverses the orientations on the curves; but what matter are only the coupled orientations, which stay the same.
\end{example}

\begin{remark}
If $\alpha$, $\beta$, $\gamma$ are as in Example~\ref{ex:abc}, except in the cyclic order $\alpha$, $\gamma$, $\beta$, then the trivial $\alpha$-$\beta$-$\gamma$ triangle is obstructed, and thus does not contribute to the polygon map. Indeed, a slight perturbation of one of the curves produces a small $\alpha$-$\gamma$-$\beta$ triangle instead of an $\alpha$-$\beta$-$\gamma$ one.
\end{remark}

\begin{example}
\label{ex:product}
Suppose we have two symplectic manifolds $M$ and $M'$ (of dimensions $2n$ and $2n'$) and two pairs of Lagrangians as in Section~\ref{sec:LFH}: $(L_0, L_1)$ in $M$, and $(L_0', L_1')$ in $M'$. We obtain a product pair $(L_0 \times L_0', L_1 \times L_1')$ in $M \times M'$. 

Unfortunately, Pin structures do not behave well with respect to products; see Remark~\ref{rem:productpin}. We will use Spin structures instead. Assume $L_0, L_1, L_0', L_1'$ come equipped with Spin structures (and, in particular, are oriented). We consider the product Spin structures on $L_0 \times L_0'$ and $L_1 \times L_1'$. 
Note that Spin structures on a pair of Lagrangians produce both Pin structures and orientations, hence coupled orientations, so the Floer chain complex is well-defined.

%Further, given coupled orientations on $(L_0, L_1)$ and $(L_0', L_1')$ we obtain a coupled orientation on $(L_0 \times L_0', L_1 \times L_1')$ by using the identification
%$$ L_0 \times L_1 \times L_0 \times L_1' \cong L_0 \times L_0' \times L_1 \times L_1'.$$

%Observe that, whereas orientations on products depend on the order of the factors, coupled orientations do not. Indeed, if we swap $L_0$ with $L_0'$ and at the same time swap $L_1$ with $L_1'$, then the orientations on $L_0 \times L_0'$ and $L_1 \times L_1'$ both differ from the original ones by $(-1)^{nn'}$, so the coupled orientation is unchanged.

With this in mind, we consider the Floer chain complexes $\CF_*(L_0, L_1)$, $\CF_*(L_0', L_1')$ and $\CF_*(L_0 \times L_0', L_1 \times L_1')$. Given intersection points $\x \in L_0 \cap L_1$ and $\x' \in L_0' \cap L_1'$, there is an intersection point $\x \times \x' \in (L_0 
\times L_0') \cap (L_1 \times L_1').$ Given index problems on the factors, their direct sum is the index problem on the product. Thus, we have identifications between the orientation spaces $o(\x \times \x') \cong o(\x) \otimes o(\x')$. We also have an identification between the moduli spaces $\M(\phi \times \phi') \cong \M(\phi) \times \M(\phi').$ It follows that $\CF_*(L_0 \times L_0', L_1 \times L_1')$ is the tensor product of the chain complexes $\CF_*(L_0, L_1)$ and $\CF_*(L_0', L_1')$. (See \cite{WWOri,Amorim2017}.)

More generally, if we have an $(m+1)$-tuple of Spin Lagrangians on $M$ and an $(m+1)$-tuple of Spin Lagrangians on $M'$, the moduli space of holomorphic polygons on the product is the fiber product of the moduli spaces of holomorphic polygons on the factors over the abstract moduli spaces of $(m+1)$-polygons. In particular, for $m=2$, it is simply a product.
\end{example}

\section{Three-manifold invariants}
\label{sec:3m}
We present here a construction of Heegaard Floer homology, following \cite{HolDisk}, but with two departures from the original. The first (which is our main point) is that we define the signs in the differential using canonical orientations instead of coherent orientations. The second is that we apply Perutz's work \cite{Perutz} to equip the symmetric product with a symplectic form so that the Heegaard tori are Lagrangians, rather than only totally real as in \cite{HolDisk}. This allows us to be in the setting of Section~\ref{sec:general}. An alternative would have been to extend the results from that section to the totally real case; this can indeed be done, because the totally real Grassmannian deformation retracts onto the Lagrangian one. However, Perutz's perspective will also be useful to us when we discuss handleslides in Section~\ref{sec:handleslide}. 

\subsection{Heegaard Floer homology: a preliminary version}
\label{sec:definition}
Let $Y$ be a closed, connected, oriented $3$-manifold and $z\in Y$ a basepoint. A (pointed) {\em Heegaard diagram} for $Y$ consists of the data
$$ \H = (\Sigma, \alphas=\{\alpha_1, \dots, \alpha_g\}, \betas=\{\beta_1, \dots, \beta_g\}, z)$$
where $\Sigma$ is a surface of genus $g$ containing $z$, and splitting $Y$ into two handlebodies $U_{\alpha}$ and $U_{\beta}$. Further, $\alphas$ is a collection of $g$ disjoint simple closed curves on $\Sigma$ that are attaching curves for $U_{\alpha}$, in the sense that we obtain $U_{\alpha}$ by attaching disks with boundaries on the curves $\alpha_i$, and then a $3$-ball. Similarly, $\betas$ is a collection of attaching curves for $U_{\beta}$. We assume that the alpha and beta curves intersect transversely, and are disjoint from $z$. Note that $U_{\alpha}$ and $U_{\beta}$ inherit orientations from $Y$. We orient $\Sigma$ as $\del U_{\alpha} = -\del U_{\beta}$.

We consider the symmetric product 
$$M=\Sym^g(\Sigma)= \Sigma^{\times g} / S_g,$$ where $S_g$ is the symmetric group. Inside $H^2(\Sym^g(\Sigma); \Z)$ we have a class $\eta$ Poincar\'e dual to $\{z\} \times \Sym^{g-1}(\Sigma)$. For each oriented simple closed curve $\gamma$ on $\Sigma$, we also have a class  $\langle \gamma \rangle$ in $H^1(\Sym^g(\Sigma); \Z)$, dual to $\gamma \times \Sym^{g-1}(\Sigma)$. If we pick curves $\gamma_1, \dots, \gamma_g$
and duals $\gamma_1^*, \dots, \gamma_g^*$ so that they produce a symplectic basis of $H_1(\Sigma)$, the class
$$\theta := \sum_{i=1}^g \langle \gamma_i \rangle  \langle \gamma_i ^* \rangle \in H^2(\Sym^{g-1}(\Sigma); \Z)$$
is invariant under the action of the mapping class group of $\Sigma$. 

For $\epsilon > 0$ sufficiently small, Perutz \cite{Perutz} constructs a symplectic form $\omega$ on $\Sigma$ in the class $\eta + \epsilon \theta$, such that the tori
$$\Ta = \alpha_1 \times \cdots \times  \alpha_g, \ \ \ \Tb=\beta_1 \times \cdots \times \beta_g$$
are Lagrangians with respect to $\omega$. As an aside, we note that $c_1(TM) = \eta - \theta$.

In  \cite[Section 10.4]{HolDiskTwo}, Ozsv\'ath and Szab\'o define an absolute grading on the Heegaard Floer groups and hence on the generators $\x\in \Ta \cap \Tb$. We will denote their grading by $\grHF(\x) \in \Z/2$.

 Let $\AA$ and $\BB$ be  the subspaces of $H_1(\Sigma; \R)$ spanned by the alpha and beta curves respectively. We have canonical isomorphisms
\begin{equation}
\label{eq:Aspace}
\AA= \ker(H_1(\Sigma; \R) \to H_1(U_{\alpha}; \R)) \cong H_1(\Ta; \R),
\end{equation}
\begin{equation}
\label{eq:Bspace}
\BB= \ker(H_1(\Sigma; \R) \to H_1(U_{\beta}; \R))\cong H_1(\Tb; \R).
 \end{equation}
Further, $\AA$ can be canonically identified with the tangent space to $\Ta$ at any point, and $\BB$ to the tangent space to $\Tb$ at any point. 

We define orientations on the moduli spaces $\Mhat(\phi)$ as in Section~\ref{sec:LFH}, by making the following choices:
\begin{itemize}
\item We choose Lie group Pin structures on the tori $\Ta$ and $\Tb$, as in Example~\ref{ex:TnPin}. To specify these, it suffices to choose them on the tangent spaces to $\Ta$ and $\Tb$ at arbitrary points; i.e., we need Pin structures $\Pa$ on $\AA$ and $\Pb$ on $\BB$. For now, we let $\Pa$ and $\Pb$ be arbitrary;
\item For the coupled orientation on $(\Ta, \Tb)$, we choose it so that the resulting $\Z/2$-grading $\gr$ on the Heegaard Floer complex (constructed as in Section~\ref{sec:LFH})  is given by the formula:
\begin{equation}
\label{eq:grHF}
\gr(\x) = \grHF(\x) + g + b_1(Y)  \pmod{2}.
\end{equation}
\end{itemize}

The intersection points $\x \in \Ta \cap \Tb$ decompose according to $\Spinc$ structures on the $3$-manifold $Y$. A subset of those, $\Set \subseteq  \Spinc(Y)$, is the set of equivalence classes described in Section~\ref{sec:decompose}. (Not all the $\Spinc$ structures have to be represented in a given  Heegaard diagram.) 

Fix $\s \in \Spinc(Y)$. We assume that the Heegaard diagram is strongly $\s$-admissible in the sense of \cite[Definition 4.10]{HolDisk}. There are several versions of the Heegaard Floer complex, differing in how we keep track of the basepoint $z$. A preliminary infinity version $\CFiprel(\H, \s) = \CFiprel(\Ta, \Tb, \s)$ is generated by pairs $[\omega_\x, i]$, where $\omega_\x$ is an orientation of $\det(\delbar_\x)$ for $\x \in \s \subset \Ta \cap \Tb$, and $i \in \Z$. The differential is given by
$$\del [\omega_\x, i] =\sum_{\y \in \s}  \sum_{\substack{\phi \in \pi_2(\x, \y)\\ \mu(\phi)=1}} \bigl( \# \Mhat(\phi) \bigr) \cdot [\omega_\y, i - n_z(\phi)].$$
Here, $n_z(\phi)$ denotes the intersection product of $\phi$ with the class of the divisor $\{z\} \times \Sym^{g-1}(\Sigma)$. There is a $\Z[U]$-action on the complex, given by $U\cdot [\omega_{\x}, i] = [\omega_\x, i-1]$.

Other Heegaard Floer complexes are: 
\begin{itemize}
\item $\CFmprel(\H, \s)$, the subcomplex of $\CFiprel(\H, \s)$ generated by $[\omega_\x, i]$ with $i < 0$; 
\item $\CFpprel(\H, \s)$, the corresponding quotient complex, generated by $[\omega_\x, i]$ with $i \geq 0$; 
\item $\CFhatprel(\H, \s)$, generated directly by $\omega_\x$ and counting only disks in classes $\phi$ with $n_z(\phi) = 0.$
\end{itemize}

To ensure that the differentials are finite, we need to impose a strong admissibility condition on the Heegaard diagram, as in \cite[Section 4.2]{HolDisk}. If we only care about $\CFpprel(\H, \s)$ and $\CFhatprel(\H,\s)$, a weaker admissibility condition is sufficient; see \cite[Theorem 4.15]{HolDisk}. 

To show that the differentials square to zero, Ozsv\'ath and Szab\'o prove that the contributions from disk and sphere bubbles are zero; see \cite[Section 3.7]{HolDisk}. Their proof uses degeneration arguments and the fact that the counts of $0$-dimensional moduli spaces of bubbles are constant in families. 
As such, the proof works just as well for any signed counts with this property; in particular, for those induced by Pin structures as in Section~\ref{sec:bubbles}.

From here, we obtain preliminary Heegaard Floer homology groups:
$$\HFiprel(\H, \s), \ \ \HFmprel(\H, \s), \ \ \HFpprel(\H, \s), \ \ \HFhatprel(\H, \s).$$

\subsection{The absolute $\Z/2$-grading} 
\label{sec:abs}
Before moving forward, let us make some comments about the absolute $\Z/2$-grading $\grHF$ used in \eqref{eq:grHF}. In \cite[Section 10.4]{HolDiskTwo}, this grading is defined by asking that a version of Heegaard Floer homology with twisted coefficients, $\HFtw^\infty$, is supported in even degrees. 

 To see that the argument is not circular, note that  it suffices to first define and compute $\HFtw^\infty$ without signs (i.e., using $\F_2$ instead of $\Z$). This still defines the absolute $\Z/2$-grading, and hence a sign on the intersection points $\Ta \cap \Tb$ and a coupled orientation on $(\Ta, \Tb)$. We then use it to upgrade the Floer complex to $\Z$ coefficients.   

 For the purpose of studying naturality, we shall however need a more concrete definition of the absolute $\Z/2$ grading $\grHF$, which makes no reference to the calculation of $\HFtw^\infty$; see \cite[Section 2.4]{FJR} and \cite[Section 3]{Petkova}.  An orientation of $\Ta$ is equivalent to one of $\AA$, and one of $\Tb$ to one of $\BB$. Thus, to define $\grHF$ it suffices to specify a coupled orientation on $(\AA,\BB)$. 

By associating to each alpha or beta curve the disk it bounds in the respective handlebody, and combining this with Lefschetz duality, we get canonical isomorphisms
\begin{align}
\label{eq:Aspace2}
\AA &\cong H_2(U_{\alpha}, \del U_{\alpha}; \R) \cong H^1(U_{\alpha}; \R),\\
\label{eq:Bspace2}
\BB &\cong H_2(U_{\beta}, \del U_{\beta}; \R)  \cong H^1(U_{\beta}; \R).
 \end{align}

Therefore, $\AA$ is the dual vector space to $H_1(U_{\alpha}; \R)$ and $\BB$ to $H_1(U_{\beta}; \R)$. We have a Mayer-Vietoris sequence
\begin{equation}
\label{eq:MVseq}
 0 \to H_2(Y; \R) \to H_1(\Sigma; \R) \to H_1(U_{\alpha}; \R) \oplus H_1(U_{\beta}; \R) \to H_1(Y; \R) \to 0.
 \end{equation}
An orientation on a vector space determines one on its dual. Since $H_2(Y; \R)$ and $H_1(Y; \R)$ are related by Poincar\'e duality, there is a coupled orientation on them. Furthermore, $H_1(\Sigma; \R)$ is a symplectic vector space and has a canonical orientation. Combining these, we obtain an orientation on $H_1(U_{\alpha}; \R) \oplus H_1(U_{\beta}; \R)$ and hence on $\AA \oplus \BB$. This produces a coupled orientation on $(\AA, \BB)$ using Lemma~\ref{lem:coupled-o} and Conventions~\ref{conv0}. (In \cite[Section 2.4]{FJR}, at the end they multiply the orientation on  $\AA \oplus \BB$ by $(-1)^{g(g-1)/2}$. This final sign change is the same as that needed to go from the concatenated to the shuffled orientation; see Section~\ref{sec:co}. In our context, the sign change is automatically implemented by the transition to a coupled orientation, by Convention~\ref{conv0}.)

It is not hard to see that this definition of $\gr_{\HF}$ can be rephrased more simply as follows. Since $\AA$ and $\BB$ are Lagrangian subspaces in the vector space $H_1(\Sigma; \R)$, equipping the latter with an inner product produces an isomorphism 
$$ \tau_{\AA, \BB} : \AA \to \BB$$
as in Lemma~\ref{lem:pinLag} (1). The coupled orientation on $(\AA, \BB)$ is obtained by choosing any orientation on $\AA$ and pairing it with its image under $\tau_{\AA, \BB}$. We will take this as the working definition of $\gr_{\HF}$ in our paper.

\begin{remark}
\label{rem:HFSW}
There is also an absolute $\Z/2$ grading on monopole Floer homology, defined by Kronheimer and Mrowka in \cite[Subsection 22.4]{KMBook}. This corresponds to $\grHF + b_1(Y)$, which is an expression that will appear naturally in the statement of Proposition~\ref{prop:canonical} below. 
\end{remark}

\subsection{Adjusting the definition}
\label{sec:adjust}
The isomorphism classes of the preliminary Heegaard Floer groups turn out to be three-manifold invariants. However, in the definition we had to choose Pin structures $\Pa$ on $\AA$ and $\Pb$ on $\BB$. To make the Heegaard Floer groups themselves into {\em canonical} invariants, we need to get rid of their dependence on this data. We will do this by tensoring  with certain lines (rank one free abelian groups). 

Let us make use of the concept of a coupled Spin structure defined in Section~\ref{sec:coupledspin}. We view the vector spaces $\AA$ and $\BB$ as vector bundles over a point. Since they are Lagrangian subspaces of $H_1(\Sigma; \R)$, Lemma~\ref{lem:pinLag} (2) gives a canonical coupled Spin structure on $(\AA, \BB)$; we denote this structure by $\Scan$.

In Section~\ref{sec:definition} we also introduced a coupled orientation on $(\AA, \BB)$: the one defining the grading $\grHF$. Given the description in Section~\ref{sec:abs}, we see that the coupled Spin structure $\Scan$ produces this coupled orientation on $(\AA, \BB)$, under the natural map $\Spin(g, g) \to \SO(g, g)$. 

As noted in Remark~\ref{rem:EF}, once we fix the coupled orientation, a pair of Pin structures $\Pa$ on $\AA$ and  $\Pb$ on $\BB$ gives rise to a coupled Spin structure on $(\AA, \BB)$, which we denote simply by $(\Pa, \Pb)$. The space of coupled Spin structures on $(\AA, \BB)$ compatible with the given coupled orientation is homotopy equivalent to $\RP^\infty$. Since $\pi_1(\RP^\infty) =\Z/2$, there are exactly two homotopy classes of paths from $(\Pa, \Pb)$ to $\Scan$. If these classes are denoted $\gamma_1$ and $\gamma_2$, we let $\ell(\Pa, \Pb)$ be the rank one abelian group generated by $\gamma_1$ and $\gamma_2$ modulo the relation $\gamma_1 = -\gamma_2$.

For $\circ \in \{\widehat{\phantom{u}}, +, -, \infty\}$, we define the chain complex
$$ \CFcirc(\H, \s) = \CFcircprel(\H, \s) \otimes \ell(\Pa, \Pb)$$
with homology
$$ \HFcirc(\H, \s) = \HFcircprel(\H, \s) \otimes \ell(\Pa, \Pb).$$

As we shall soon see, these groups are natural invariants of the three-manifold and the $\Spinc$ structure. We will often denote them by $\HFcirc(Y, \s)$.

For now, let us prove the first step towards naturality.
\begin{proposition}
\label{prop:nopin}
For $\circ \in \{\widehat{\phantom{u}}, +, -, \infty\}$, the homology groups $\HFcirc(\H, \s)$ are independent of the choice of Pin structures $\Pa$ and $\Pb$, up to canonical isomorphism.
\end{proposition}

\begin{proof}
We will prove this at the level of chain complexes. Let $(\Ra, \Rb)$ be another choice for the pair of Pin structures on $(\AA, \BB)$. We claim that the complexes $\CFcirc(\H, \s)$ defined from $(\Pa, \Pb)$ are canonically isomorphic to those from $(\Ra, \Rb)$. 

Pick paths of Pin structures $\gamma$ from $\Pa$ to $\Ra$ and $\eta$ from $\Pb$ to $\Rb$. They induce an isomorphism of orientation spaces $o(\x)$ from those defined using $\Pa$ and $\Pb$ to those defined using $\Ra$ and $\Rb$. The isomorphism is compatible with the differentials, so we get an isomorphism between the corresponding preliminary Heegaard Floer chain groups $\CFcircprel(\H, \s)$. Moreover, combining $\gamma$ and $\eta$ we obtain a path $(\gamma, \eta)$ of coupled Spin structures from $(\Pa, \Pb)$ to $(\Ra, \Rb)$. Using concatenation with the paths to $\Scan$, this gives an isomorphism of lines $\ell(\Pa, \Pb) \cong \ell(\Ra, \Rb)$. Altogether, we get the desired isomorphism between the respective $\CFcirc(\H, \s)$.

We need to show that this last isomorphism does not depend on the choices of $\gamma$ and $\eta$. Since the space of Pin structures has $\pi_1 = \Z/2$ (cf. Remark~\ref{rem:non-equivalent_iso}), up to homotopy there are only two possible choices for $\gamma$ and two for $\eta$. Changing such a choice results in a sign change for the isomorphism on orientation spaces and hence on $\CFcircprel(\H, \s)$ (see Lemma~\ref{lem:nontrivial}), but this is cancelled by another sign change for the isomorphism between $\ell(\Pa, \Pb)$ and $\ell(\Ra, \Rb)$.
\end{proof}

\begin{convention}
\label{conv:Spinc}
We will sometimes drop the $\Spinc$ structure $\s$ from notation and consider preliminary chain groups $\CFcircprel(\H) = \CFcircprel(\Ta, \Tb)$ with homology $\HFcircprel(\H) = \HFcircprel(\Ta, \Tb)$, as well as adjusted chain groups $\CFcirc(\H) = \CFcirc(\Ta, \Tb)$ with homology $\HFcirc(\H) = \HFcirc(\Ta, \Tb)$. By these we always mean the direct sum over the set $\Set$ of $\Spinc$ structures that are represented by intersection points in the diagram $\H$. Thus, $\HFcirc(\Ta, \Tb)$ may be different from $\HFcirc(Y)$, the latter being the direct sum of $\HFcirc(Y, \s)$ over {\em all} $\Spinc$ structures $\s$. 
\end{convention}

\subsection{More on Pin structures}
\label{sec:morepin}
Our definition of Heegaard Floer homology relies on Pin structures $\Pa$ and $\Pb$. Although in Proposition~\ref{prop:nopin} we showed that in the end different choices produce canonically isomorphic groups, in practice it is useful to work with concrete Pin structures. One way to produce them is as follows.

\begin{definition}
\label{def:chooseab}
Choose some additional data $a$, consisting of an ordering of the alpha curves together with an orientation of each. This trivializes the tangent space $T_\x \Ta$, and lets us define a trivial Spin structure $S^\#_a$ on $\Ta$. We can think of it as the product of the trivial Lie group Spin structures on each $\alpha_i$ (with the product taken in the given order). We let $P^\#_a$ be the Pin structure on $\AA$ induced from $S^\#_a$. 

Similarly, we let $P^\#_b$ be the Pin structure on $\BB$ obtained from data $b$ that consists of an ordering of the beta curves together with an orientation of each. 
\end{definition}

\begin{remark}
Data $a$ and $b$ as above also specifies orientations of $\AA$ and $\BB$, and hence a coupled orientation. It is convenient to choose $a$ and $b$ so that this agrees with the coupled orientation described in Section~\ref{sec:abs}. 
\end{remark}

Going back to the case of general Pin structures, it is worth emphasizing the following fact.
\begin{proposition}
\label{prop:onlycoupled}
 Let $(\Pa, \Pb)$ and $(\Ra, \Rb)$ be two pairs of Pin structures that represent the same coupled Spin structure on $(\AA, \BB)$. Then, for each $\x \in \Ta \cap \Tb$, the orientation lines $o(\x)$ constructed from $(\Pa, \Pb)$ and $(\Ra, \Rb)$ are canonically isomorphic. 
\end{proposition}

\begin{proof}
In the proof of Proposition~\ref{prop:nopin} we argued that the groups $o(\x) \otimes \ell(\Pa, \Pb)$ are canonically isomorphic, for any choices of Pin structures. In the particular case when the Pin structures produce the same coupled Spin structure, the lines $\ell(\Pa, \Pb)$ are canonically isomorphic by definition. Thus, the same holds for $o(\x)$. 
\end{proof}

Since the isomorphisms in Proposition~\ref{prop:onlycoupled} are compatible with the differentials, it follows that the preliminary groups $\CFcircprel(\H, \s)$ constructed from those pairs are canonically isomorphic. Thus, what matters in the construction of $\CFcircprel(\H, \s)$ is only the coupled Spin structure, not its Pin representatives.

\subsection{Describing some generators}
Observe that the generators of $\CFcirc(\H, \s)$ are elements of the lines $o(\x) \otimes \ell(\Pa, \Pb)$ for $\x \in \Ta \cap \Tb$. The following result will be useful. 

\begin{proposition}
\label{prop:canonical}
If $\x \in \Ta \cap \Tb$ has $\gr(\x) = g \pmod{2}$, i.e. $\grHF(\x) = b_1(Y) \pmod{2}$, then the line $o(\x) \otimes \ell(\Pa, \Pb)$ admits a canonical trivialization.
\end{proposition}

\begin{proof}
In view of the definition of $\gr(\x)$ in Section~\ref{sec:LFH}, the condition $\gr(\x) = g \pmod{2}$ means that the orientation on $T_\x \Ta \oplus T_\x \Tb$ (coming from the coupled orientation) agrees with the one on $T_\x M$. Suppose $\x = \{x_1, \dots, x_g\}$ with $x_i \in \alpha_i \cap \beta_{\sigma(i)}$, for some permutation $\sigma$. We view each pair $(\alpha_i, \beta_{\sigma(i)})$ as Lagrangians in $\Sigma$ and we choose the coupled orientation on them so that it agrees with the one on $\Sigma$; that is, so that $\gr(x_i)=1$ for all $i$. The product of these coupled orientations gives the coupled orientation on $(\Ta, \Tb)$. 

Let us now choose some additional data $a$ as in Definition~\ref{def:chooseab}, consisting of an ordering of the alpha curves together with an orientation of each. With our $\x$ fixed, the data $a$ produces a similar data $b$ on the beta curves, consisting of an ordering of them together with orientations. Indeed, for the ordering, we simply transplant the ordering of the alpha curves using the permutation $\sigma$ (determined by the generator $\x$). For the orientations, we choose them so that we get the coupled orientation on each $(\alpha_i, \beta_{\sigma(i)})$ that we already specified. As in Definition~\ref{def:chooseab}, let $P^\#_a$ and $P^\#_b$ be the Pin structures induced from $a$ and $b$. 

Let us first prove the lemma in the case where the groups are defined from Pin structures $\Pa = P^\#_a$ and $\Pb = P^\#_b$ for some data $a$ (and the induced data $b$). In that case, recall that the orientation space $o(\x)$ is associated to an index problem on the cap $H$, with boundary conditions interpolating between $T_\x \Ta$ and $T_\x \Tb$. This is the direct sum of the index problems corresponding to each $x_i$, and therefore $o(\x)$ can be canonically identified with the tensor product of the orientation spaces $o(x_i)$ for $i=1, \dots, g$; compare Example~\ref{ex:product}. (Since we trivialized the tangent spaces, we have Spin and not just Pin structures, so Example~\ref{ex:product} is applicable here.) Because $\gr(x_i)=1$, Lemma~\ref{lem:dim2} gives a canonical trivialization for each of $o(\x_i)$. Tensoring them together we get a trivialization of $o(\x)$. Furthermore, using Remark~\ref{rem:cansum} about the behavior of canonical Spin structures under direct sum, we deduce that $(P^\#_a, P^\#_b)$ represents the canonical coupled Spin structure on $(\AA, \BB)$. Thus, by choosing the constant path we trivialize the line $\ell(P^\#_a, P^\#_b)$.

Let us now consider the case where $\Pa$ and $\Pb$ are arbitrary. From the proof of Proposition~\ref{prop:nopin} we know that there is a canonical isomorphism between the line $o(\x) \otimes \ell(\Pa, \Pb)$ and the similar line defined using $(P^\#_a, P^\#_b)$. Since the latter line is trivialized, so is the former. However, we need to make sure that this trivialization does not depend on the chosen data $a$; that is, if we have two sets of data $a$ and $a'$, their canonical trivializations (defined in the previous paragraph) are related by the isomorphism from the proof of Proposition~\ref{prop:nopin}. 

If the change in $a$ comes from a change in the orientation of one of the curves $\alpha_i$, note that there must be a corresponding change in the orientation of the curve $\beta_{\sigma(i)}$. Overall, the coupled orientation and the coupled Spin structure on the pair $(\alpha_i, \beta_{\sigma(i)})$ are preserved. It follows that we are in the situation of Proposition~\ref{prop:onlycoupled}: The coupled Spin structure being fixed, the  spaces $o(x_i)$ (before and after we made the orientation changes) are canonically identified. The origin of these identifications is the same as that in the proof of Proposition~\ref{prop:nopin}, so everything is compatible. 

It remains to study the case of a change in the ordering of the alpha curves. It suffices to consider a transposition,  say changing the ordering $(\alpha_1, \alpha_2)$ to $(\alpha_2, \alpha_1)$, which must come in tandem with changing $(\beta_{\sigma(1)}, \beta_{\sigma(2)})$ to $(\beta_{\sigma(2)}, \beta_{\sigma(1)})$. This produces a genuine change in the Pin structures. We obtain a local problem near $\x$, and we can restrict attention to small neighborhoods of each $x_i= \alpha_i \cap \beta_{\sigma_i}$ in $\Sigma$, whose product is  a neighborhood of $\x$ in $\Sym^g(\Sigma)$, away from the diagonal. Furthermore, without loss of generality, we can assume that $\sigma$ is the identity, and that we are in two dimensions (that is, we focus on the neighborhoods of $x_1$ and $x_2$, as the rest is not affected by the ordering change).

The problem we are left with is the following. For $i=1,2$, we are given curves $\alpha_i$ and $\beta_i$ in $\C\cong \R^2$, intersecting transversely at $x_i$. We may as well assume $\alpha_i$ and $\beta_i$ are lines through $x_i=0$. We denote $\x = (0, 0) \in \C^2$, $\Ta =\alpha_1 \times \alpha_2$, $\Tb=\beta_1 \times \beta_2$. Because the coupled orientations are fixed, and  we are otherwise free to choose orientations on $\alpha_i$, we can assume that $(\alpha_1, \beta_1)$ and $(\alpha_2, \beta_2)$  are both positive bases of $\R^2$. Let $a$ and be the data consisting of the chosen orientations of the curves, together with the ordering $(\alpha_1, \alpha_2)$; let $a'$ be the similar data with the ordering $(\alpha_2, \alpha_1)$. These produce Pin structures $P^\#_a$ and  $P^\#_{a'}$ on $T_\x \Ta$. Similarly, the orderings $(\beta_1, \beta_2)$ and $(\beta_2, \beta_1)$ give data $b$ and $b'$ and produce Pin structures $P^\#_b$ and $P^\#_{b'}$ on $T_\x\Tb$. The lines $o(\x) \otimes \ell(P^\#_\alpha, P^\#_\beta)$ defined using $(P^\#_a, P^\#_{b})$ and $(P^\#_{a'}, P^\#_{b'})$ are related by two isomorphisms, and our problem is to show that these isomorphisms are the same. 

Specifically, one isomorphism comes from viewing $o(\x)$ as a tensor product in two different ways: $o(x_1) \otimes o(x_2) \cong o(x_2) \otimes o(x_1)$. We combine this with the identity on $\ell(P^\#_\alpha, P^\#_\beta)$, as the coupled Spin structure is the same for both orderings. 

The other isomorphism is described in the proof of Proposition~\ref{prop:nopin}, where we had to choose paths $\gamma$ from $P^\#_a$ to $P^\#_{a'}$ and $\eta$ from $P^\#_b$ to $P^\#_{b'}$. We let $\gamma$ be given by the Pin structures coming from the trivializations of 
$$T_\x \Ta \cong T_{x_1} \alpha_1 \oplus T_{x_2} \alpha_2 \cong \R^2$$ using the bases
$$((\cos t, \sin t), (-\sin t, \cos t)), \ \ t\in [0,\pi/2].$$
Similarly, we let $\eta$ come from the trivializations of 
$$T_\x \Tb \cong T_{x_1} \beta_1 \oplus T_{x_2} \beta_2 \cong \R^2$$ using the same bases as above. 
 The coupled Spin structure is kept constant throughout this process, so $\ell(\Pa, \Pb)$ is  trivialized with the constant path. Furthermore, following these paths we see decompositions of $o(\x)$ as products  of two lines, starting with $o(x_1) \otimes o(x_2)$ and ending with $o(x_2) \otimes o(x_1)$. In other words, we get the same isomorphism as the first one we described.
\end{proof}

\subsection{The homology action}
Following \cite[Section 4.2.5]{HolDisk}, we can equip our Heegaard Floer homology groups $\HFcirc(Y, \s)$ with an action of the exterior algebra $\Lambda^*(H_1(Y; \Z)/\text{Tors})$. We have an isomorphism
$$ H^1(\Path(\Ta, \Tb); \Z) \cong \Z \oplus (H^1(Y; \Z)/\text{Tors}).$$
Given a cocycle $\zeta \in Z^1(\Path(\Ta, \Tb); \Z)$, we define its action on the preliminary Floer complex $\CFiprel(\H, \s)$ by
$$\del [\omega_\x, i] =\sum_{\y \in \s}  \sum_{\substack{\phi \in \pi_2(\x, \y)\\ \mu(\phi)=1}} \zeta(\phi) \cdot \bigl( \# \Mhat(\phi) \bigr) \cdot [\omega_\y, i - n_z(\phi)].$$
The same proof as in \cite[Proposition 4.17]{HolDisk} shows that $A_{\gamma}$ is a chain map, depends only on the homology class of $\gamma$, and satisfies $A_{\gamma}^2=0$ on homology. Restricting to $[\gamma]$ in the  $H^1(Y; \Z)/\text{Tors}$ yields an action of $\Lambda^*(H_1(Y; \Z)/\text{Tors})$ on $\HFiprel(\H, \s)$, which we then tensor with the identity on $\ell(\Pa, \Pb)$ to get one on $\HFi(Y, \s)$. The actions on the other versions $\HFcirc(Y, \s)$ are constructed similarly.

\subsection{Twisted coefficients}
\label{sec:twistedHF}
There are yet more variants of Heegaard Floer homology that can be considered; e.g., with twisted coefficients, as in Section~\ref{sec:twisted}. In the Heegaard Floer setting, the group $G=\pi_2(\x_0, \x_0)$ maps into $\Z \oplus H^1(Y; \Z)$, and this map is an isomorphism for $g > 1$; see \cite[Proposition 2.15]{HolDisk}. The $\Z$ summand captures the quantity $n_z(\phi)$, and is not relevant, because in the differential we already keep track of this quantity. The more relevant group is $\hat{\pi}_2(\x_0, \x_0)$, which only includes classes $\phi$ such that $n_z(\phi)=0$. This group maps to $H^1(Y; \Z)$, and produces genuine twisted coefficients. Thus, we obtain groups
$$\HFitw(Y, \s; A), \ \ \HFmtw(Y, \s; A), \ \ \HFptw(Y, \s; A), \ \ \HFhattw(Y, \s; A)$$
for any module $A$ over $\Z[H^1(Y; \Z)]$. Compare \cite[Section 8]{HolDiskTwo}.

In particular, we can consider changes in the Pin structures on the Lagrangians, which correspond to twisting coefficients by modules of the form $A_{\eta, \zeta}$, as in Proposition~\ref{prop:changePin}. 
As explained at the end of Section~\ref{sec:twisted}, different systems of orientations on the moduli spaces are parametrized by an affine space on 
\begin{equation}
\label{eq:coker}
\coker \bigl( H^1(M; \Z/2) \to H^1(\Ta; \Z/2) \oplus H^1(\Tb; \Z/2) \bigr).
\end{equation}
Let $M' = \Sym^g(\Sigma -\{z\})$. We have $\pi_1(\Omega M')=\pi_2(M') = H_2(M') = 0$, and $\pi_1(M') = \pi_1(M) = H_1(\Sigma)=\Z^{2g}$. Let also $\Pathhat(\Ta, \Tb)$ be the space of homotopy classes of Whitney disks inside $M'$. This has $\pi_1(\Pathhat(\Ta, \Tb)) \cong  H^1(Y; \Z)$ for $g > 1$.

We have the short exact sequence of a fibration
$$0=\pi_1(\Omega M') \to \pi_1(\Pathhat(\Ta, \Tb)) \to \pi_1(\Ta \times \Tb) \to \pi_1(M') \to 0.$$
After we abelianize, the exact sequence splits, because $H_1(M)=\Z^{2g}$ is free. Therefore we can dualize it and get a short exact sequence
$$ 0\to H^1(M; \Z/2) \to H^1(\Ta \times \Tb; \Z/2) \to H^1(\Pathhat(\Ta, \Tb); \Z/2) \to 0$$
so the cokernel in \eqref{eq:coker} is (at least for $g > 1$)
$$ H^1(\Pathhat(\Ta, \Tb); \Z/2)\cong \Hom(\pi_1(\Pathhat(\Ta, \Tb)), \Z/2) \cong \Hom(H^1(Y; \Z), \Z/2).$$
This exactly corresponds to the different {\em coherent orientation systems} considered by Ozsv\'ath and Szab\'o in \cite[Section 4.2.4]{HolDisk}. It was observed there that  such systems form an affine space over $\Hom(H^1(Y; \Z), \Z/2).$ 

Thus, on a $3$-manifold $Y$, there are $2^{b_1(Y)}$ orientation systems. Our choice of coupled Spin structures picks up a particular one. In \cite[Theorem 10.12]{HolDiskTwo}, Ozsv\'ath and Szab\'o also choose a particular one, but in a different way, using their study of $\HFitw$ as a $\Z[H^1(Y; \Z)]$-module. We do not attempt to show that our orientation system corresponds to theirs in general. Rather, let us focus on the case $Y=\#^g(S^1 \times S^2)$, which will become relevant to us in Section~\ref{sec:naturality}.

For $\#^g(S^1 \times S^2)$, we have a Heegaard diagram with the $\beta$ curves isotopic to the $\alpha$ curves; see \cite[Lemma 9.1]{HolDisk}. There, they choose the orientation system so that
\begin{equation} 
\label{eq:S1S2}
\HFhat(\#^g(S^1 \times S^2)) \cong H_*(T^g; \Z).
\end{equation}
We check that our construction produces the same answer.

\begin{lemma}
\label{lem:S1S2}
Equation~\eqref{eq:S1S2} holds for the Heegaard Floer groups defined over $\Z$ in our setting, with $\HFhat(\#^g(S^1 \times S^2))$ equipped with the $\grHF$ grading.
\end{lemma}

\begin{proof}
In the genus $g$ diagram where the $\alpha_i$ curve is Hamiltonian isotopic to $\beta_i$, the Lagrangian tori $\Ta$ is Hamiltonian isotopic to $\Tb$. Further, this isotopy preserves the Lie group Pin structures. Thus, we are simply computing $\HFhat(\Ta, \Ta)$, which is  isomorphic to $H_*(\Ta; \Z)$ by Proposition~\ref{prop:PSS}. This proves the statement as relatively graded groups. The absolutely graded version also holds, because we know from \cite{HolDiskTwo} that it does so when we work with $\Z/2$ coefficients.
\end{proof}

\subsection{Handleslides}
\label{sec:handleslide}
In proving the invariance of Heegaard Floer homology, a key role will be played by the maps induced by handleslides, which we discuss here.

There are two kinds of handleslides: among the alpha curves, and among the beta curves. We will describe the maps associated to beta handleslides; those associated to alpha handleslides are similar. Furthermore, we focus on the minus version of $\HF$ for concreteness.

Following \cite[Section 9]{HolDisk}, a beta handleslide consists of replacing a curve collection $\betas=\{\beta_1, \dots, \beta_g\}$ by another collection $\gammas=\{\gamma_1, \dots, \gamma_g\}$ such that:
\begin{itemize}
\item The curve $\gamma_1$ is obtained from $\beta_1$ by sliding it over $\beta_2$;
\item For $i > 1$, the curves $\gamma_i$ is  obtained from $\beta_i$ by a small Hamiltonian isotopy;
\item For all $i$, the curve $\gamma_i$ intersects $\beta_i$ in two points, and does not intersect any other beta curve.
\end{itemize}
See Figure~\ref{fig:slide}. We let $\Tg=\gamma_1 \times \dots \times \gamma_g$ denote the resulting Lagrangian torus.
\begin{figure}
{
\fontsize{10pt}{11pt}\selectfont
   \def\svgwidth{4.8in}
   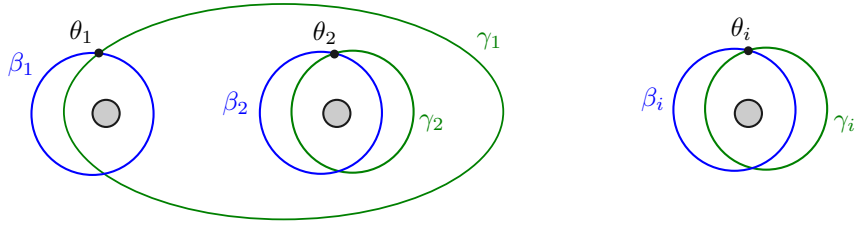
}
\caption{A handleslide. We draw the $\beta$ circles in blue and the $\gamma$ circles in green. The gray disks are feet of handles that can be connected to other parts of the diagram.}
\label{fig:slide}
\end{figure}

Recall that in the construction of the preliminary groups $\HFmprel(\Ta, \Tb)$ we used Lie group Pin structures on $\Ta$ and $\Tb$, coming from Pin structures $\Pa$ and $\Pb$ on the vector spaces $\AA$ and $\BB$ from \eqref{eq:Aspace} and \eqref{eq:Bspace}. Since the handleslide does not change the Heegaard splitting (only the diagram), the same vector space $\BB$ can be identified with the tangent space to $\Tg$, and we will use the same Pin structure $\Pb$ to get one on $\Tg$.

In \cite{Perutz}, Perutz shows that a handleslide induces a Hamiltonian isotopy $\psi: M \to M$ taking $\Tb$ to $\Tg$. The following lemma shows that our chosen Pin structures on the Lagrangian tori are preserved by $\psi$.

\begin{lemma}
\label{lemma:liepin}
The Hamiltonian isotopy $\psi$ induced by a handleslide takes the Lie group Pin structure on $\Tb$ from $\Pb$ on $\BB$ to the Lie group Pin structure on $\Tg$ coming from the same $\Pb$ (up to canonical homotopy of Pin structures).
\end{lemma}

\begin{proof} Pick a point $p_i$ on each curve $\beta_i$, so that $p=p_1 \times \dots \times p_g$ is a basepoint on $\Tb$. From the description of $\psi$ in \cite{Perutz}, we can arrange so that in a neighborhood of each $p_i$, the curve $\gamma_i$ is a small translate of the corresponding $\beta_i$ (including for $i=1$); and the isotopy $\psi$  is locally given by the product of these translations. Thus, if we identify $T_p\Tb$ and $T_{\psi(p)}\Tg$ to  $\BB$, we see that the Pin structures on these two tangent spaces both come from $\Pb$. To obtain the Lie group Pin structures on the whole Lagrangians $\Tb$ and $\Tg$, we translate the ones at the given points. 

Let us identify each $\beta_i$ to a curve (still denoted $\beta_i$) on $\Tb$, namely the product of $\beta_i$ and the basepoints $\{z_j\}$ for $j \neq i$.  For $i=1, \dots, g$, let $\beta_i' \subset \Tg$ be the image of $\beta_i \subset \Tb$ under the Hamiltonian isotopy $\psi$. The collection of curves $\{\beta_i'\}$ differ from $\{\gamma_i\}$ by a linear shear $L$ on $\Tg$. The transformation $L$ takes translations on the torus to other translations, and the image of the Pin structure on $\Tb$ under $\psi$ is invariant under these translations, just like the given Pin structure on $\Tg$ (the one coming from $\Pb$). Since these two Pin structures on $\Tg$ agree at a point, they agree everywhere.
  \end{proof}
  
 \begin{remark}
The fact that we chose Lie group Pin structures on the alpha and beta curves was key in the proof of handleslide invariance. Consider for example the alternative of equipping the Lagrangian tori with the Pin structure induced by the product of bounding Spin structures on each curve. (Recall that we can take products of Spin but not Pin structures; see Remark~\ref{rem:productpin}.) This ``bounding Pin structure'' would not have been invariant under reparametrizations of the torus by linear maps $L \in \GL(n, \Z)$. For example, up to isomorphism, the bounding Pin structure on $T^2$ differs from the Lie group Pin structure (which we know is independent of reparametrization) by the class $(1, 1) \in H^1(T^2; \Z/2) = \Z/2 \oplus \Z/2$. Clearly, this class is not preserved by the action of $\GL(n, \Z)$; e.g. the shear $(x, y) \mapsto (x, x+y)$ takes it to $(1, 0)$.
\end{remark}

Since handleslides give Hamiltonian isotopies of the Lagrangian tori, we can associate continuation maps to them, which are isomorphisms on Floer homology. In our case, {\em a priori} we get isomorphisms on the preliminary Heegaard Floer homology groups from Section~\ref{sec:definition}. To get the adjusted groups from Section~\ref{sec:adjust}, we tensor the preliminary groups with $\ell(\Pa, \Pb)$. Since this line is the same after doing the handleslide (the space $\BB$ stays fixed), we can simply tensor the preliminary map with the identity on $\ell(\Pa, \Pb)$ and obtain an isomorphism
$$ \Gamma^{\alphas}_{\betas \to \gammas}: \HFm(\Ta, \Tb, \s) \xrightarrow{\cong} \HFm(\Ta, \Tg,\s).$$

Nevertheless, for the proofs in the rest of the paper, it will be helpful to make a closer connection to the current Heegaard Floer literature. Therefore, we will work with triangle maps instead of continuation maps. Triangle maps are used in the original proof of handleslide invariance, in \cite[Section 9]{HolDisk}. 

The treatment of handleslides in \cite[Section 9]{HolDisk} starts by showing that $$\HFm(\Tb, \Tg, \s_0) \cong \Z[U] \otimes H_*(T^g; \Z),$$ and picking a generator $\Theta_{\beta, \gamma}$ of the top degree part of this group. (Here, $\mathfrak{s}_0$ is the torsion $\Spinc$ structure. This is the only $\Spinc$ supported by the $\beta$-$\gamma$ diagram, so we can drop it from the notation, following Convention~\ref{conv:Spinc}.) Then, we consider the product map
$$ F_{\Ta, \Tb, \Tg}: \HFm(\Ta, \Tb, \s) \otimes \HFm(\Tb, \Tg) \to \HFm(\Ta, \Tg, \s)$$
as in \eqref{eq:productmap}. We define the following map associated to the handleslide:
\begin{equation}
\label{eq:handleslidemap}
 \Psi^{\alphas}_{\betas \to \gammas}: \HFm(\Ta, \Tb, \s) \to \HFm(\Ta, \Tg, \s), \ \ \x \mapsto F_{\Ta, \Tb, \Tg}(\x \otimes \Theta_{\beta, \gamma}).
 \end{equation}
It is shown in \cite[Section 9]{HolDisk} that this map is an isomorphism (for suitable coherent orientation systems). 

Let us construct $ \Psi^{\alphas}_{\betas \to \gammas}$ in our setting, with the signs coming from Lie group Pin structures. The work of Perutz  \cite{Perutz} shows that $\Tg$ is Hamiltonian isotopic to $\Tb$. Therefore, by Proposition~\ref{prop:PSS} (keeping track of the basepoint as in $\HFm$), we have canonical isomorphisms
\begin{equation}
\label{eq:PSShere0}
 \HFmprel(\Tb, \Tg) \cong \HFmprel(\Tb, \Tb) \cong \Z[U] \otimes H_*(\Tb; |\ltop T\Tb|) \cong \Z[U] \otimes H_*(T^g).
 \end{equation}
 To go from the preliminary to the adjusted Heegaard Floer groups, we tensor them with $\ell(\Pb, \Pb)$. Note that the pair $(\Pb, \Pb)$ gives exactly the canonical coupled Spin structure $\Scan$ on $(\BB, \BB)$. It follows that the line $\ell(\Pb, \Pb)$ is canonically trivialized (by the constant path), so we also have canonical isomorphisms
 \begin{equation}
\label{eq:PSShere}
 \HFm(\Tb, \Tg) \cong \HFm(\Tb, \Tb) \cong \Z[U] \otimes H_*(\Tb; |\ltop T\Tb|) \cong \Z[U] \otimes H_*(T^g).
 \end{equation}

Further, note that $(\Sigma, \betas, \gammas)$ is a Heegaard diagram for the manifold $\#^g(S^1 \times S^2)$ with $b_1=g$, and the absolute mod $2$ grading $\grHF$ on its Heegaard Floer homology coincides with that on $\Z[U] \otimes  H_*(T^g)$ under the above isomorphisms. (This follows from the definition of the mod $2$ grading in terms of $\HFitw$ in \cite[Section 10.4]{HolDiskTwo}.) Moreover, in this case formula \eqref{eq:grHF} implies $\gr=\grHF$.

The chain complex $\CFm_*(\Tb, \Tg)$ has $2^g$ generators, the same as its homology, so the differential must be zero. Consider the intersection point in the top homological degree:
$$\Theta_{\beta, \gamma} = \{\theta_1, \dots, \theta_g\}$$
with $\theta_i \in \beta_i \cap \gamma_i$ as in Figure~\ref{fig:slide}. Since $\gr(\Theta_{\beta, \gamma} )=g$, Proposition~\ref{prop:canonical}  tells us that the line 
$$o(\Theta_{\beta, \gamma}) \otimes \ell(\Pb, \Pb)$$ is canonically trivialized. Since we already know the same about $\ell(\Pb, \Pb)$, we get that $o(\Theta_{\beta, \gamma})$ is canonically trivialized. Therefore, we can identify $\Theta_{\beta, \gamma}$ with the positive generator of this orientation space, and thus view it as a generator of the Floer homology $\HFmprel(\Tb, \Tg)$. We pick this $\Theta_{\beta, \gamma} $ to define the handleslide map on preliminary groups as a triangle map, as in \eqref{eq:handleslidemap}. We then tensor it with the identity on $\ell(\Pa, \Pb)$ to get the true handleslide map
\begin{equation}
\label{eq:Hmap}
 \Psi^{\alphas}_{\betas \to \gammas}: \HFm(\Ta, \Tb, \s) \to \HFm(\Ta, \Tg, \s).
 \end{equation}

\begin{remark}
\label{rem:notriangle}
In general, triangle maps are defined only on preliminary Heegaard Floer groups (following the recipe in Section~\ref{sec:polygon}). To define them on the adjusted groups $\HFm$, we would need maps
$$\ell(\Pa, \Pb) \otimes \ell(\Pb, \Pg) \to \ell(\Pa, \Pg).$$
A natural definition of these maps only exists in certain cases, for example when $\Pb =\Pg$ as in our situation.
\end{remark}

\begin{convention}
By a slight abuse of notation, when discussing handleslides, from now on we will not focus on the distinction between preliminary and adjusted Heegaard Floer groups. Since $\ell(\Pb, \Pb)$ is canonically trivialized we will write $\Theta_{\beta, \gamma}$ for the generator of $\HFm(\Tb, \Tg)$ as well, and describe the map in \eqref{eq:Hmap} as
$$\Psi^{\alphas}_{\betas \to \gammas} = F_{\Ta, \Tb, \Tg} (\cdot \otimes \Theta_{\beta, \gamma}).$$
The fact that we define it like this on the preliminary groups and then tensor with the identity is implicit.
\end{convention}

\begin{proposition}
\label{prop:hslideiso}
The map $\Psi^{\alphas}_{\betas \to \gammas}$ from \eqref{eq:Hmap} is an isomorphism.
\end{proposition}

\begin{proof}
This appears as Theorem 9.5 in \cite{HolDisk}. We do not repeat the proof, but rather explain what needs to be modified when we use canonical orientations instead of coherent orientations. To show that $\Psi^{\alphas}_{\betas \to \gammas}$ is an isomorphism, Ozsv\'ath and Szab\'o construct an inverse 
$$\Psi^{\alphas}_{\gammas \to \deltas}: \HFm(\Ta, \Tg, \s) \to \HFm(\Ta, \Td, \s), \ \ \x \mapsto F_{\Ta, \Tg, \Td}(\x \otimes \Theta_{\gamma, \delta}),$$
where $\deltas$ is a collection of curves Hamiltonian isotopic to $\betas$ (as in Figure~\ref{fig:slide2}), and $\HFm(\Ta, \Td)$ is identified with $\HFm(\Ta, \Td)$ using this isotopy. Showing that $\Psi^{\alphas}_{\gammas \to \deltas} \circ \Psi^{\alphas}_{\betas \to \gammas} = \id$ boils down to proving the relation
\begin{equation}
\label{eq:bgd}
 F_{\Tb, \Tg, \Td}(\Theta_{\beta, \gamma} \otimes \Theta_{\gamma, \delta}) = \Theta_{\beta, \delta}
 \end{equation}
which involves counting index zero holomorphic triangles in the diagram in Figure~\ref{fig:slide2}. There is only one such triangle (in the symmetric product), which is the product of the $g$ darkly shaded triangles on the Heegaard surface. 

Let us choose counterclockwise orientations on each of the $\beta$, $\gamma$, and $\delta$ curves in Figure~\ref{fig:slide2}. We also order the curves in each set with those indexed by $1$ and $2$ being the ones involved in the handleslide, again as shown in Figure~\ref{fig:slide2}. This equips the Lagrangians with orientations, in a way compatible with their  existing coupled orientations coming from the grading. In fact, we now have data that specifies a Pin structure (and even a Spin structure) on each Lagrangian, as in Definition~\ref{def:chooseab}. Furthermore, each curve becomes an oriented Lagrangian on the Heegaard surface, which we also endow with its Lie group Spin structure. Using the given ordering of the curves, we get product Spin structures on the tori. Then,  all the vertices of the darkly shaded triangles have $\gr = 1$. In the setting of coherent orientations, we know from Example~\ref{ex:abcp} that each of the darkly shaded triangles gets a sign of $+1$; see Figure~\ref{fig:triangles} (b), with $\alpha'$, $\beta'$ and $\gamma'$ replaced by $\beta_i$, $\gamma_i$, $\delta_i$, in this order. Since the three triangles come with a positive sign, so does their product; see Example~\ref{ex:product}. Relation \eqref{eq:bgd} follows from here.
\end{proof}

\begin{figure}
{
\fontsize{10pt}{11pt}\selectfont
   \def\svgwidth{4.8in}
   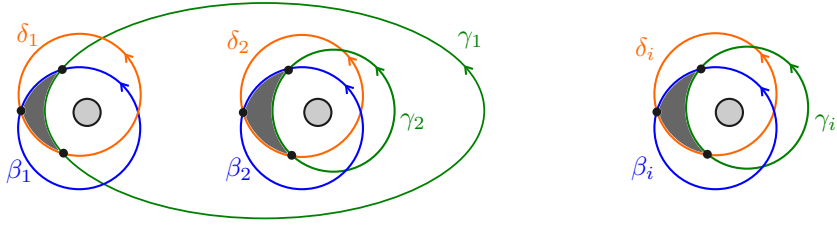
}
\caption{An index zero holomorphic triangle contributing $+1$.}
\label{fig:slide2}
\end{figure}

This concludes the description of the triangle map induced by a beta handleslide on $\HFm$.  The case of an alpha handleslide (changing $\alphas$ into a new curve collection $\gammas$) is similar, with the triangle map 
$$\Psi^{\alphas \to \gammas}_{\betas}:  \HFm(\Ta, \Tb, \s) \to \HFm(\Tg, \Tb, \s), \ \ \x \mapsto F_{\Tg, \Ta, \Tb}( \Theta_{\gamma, \alpha} \otimes \x)$$ involving the top degree generator $\Theta_{\gamma, \alpha} \in \HFm(\Tg, \Ta)$. 

The maps on the other versions of $\HF$ are defined similarly. In the case of $\HFi$ and $\HFhat$, in Equation~\ref{eq:PSShere} we need to replace $\Z[U]$ with $\Z[U, U^{-1}]$ and $\Z$, respectively, and use the corresponding $\Theta$ elements. In the case of $\HFp$, we still use the $\Theta$ element on $\HFm$ to construct the triangle map; see Theorem 8.12 and Section 9 in \cite{HolDisk}.

\subsection{Stabilizations}
\label{sec:stabs}
We now turn to maps induced by stabilizations. The stabilization move involves replacing a Heegaard diagram $\He=(\Sigma, \alphas, \betas)$ with $\He'=(\Sigma', \alphas', \betas')$, where $\Sigma'$ is the connected sum of $\Sigma$ with a (two-dimensional) torus $E$, and the curve collections $\alphas'$, $\betas'$ are obtained from $\alphas$ resp. $\betas$ by adding two new curves on $E$: $\alpha_E$ and $\beta_E$, intersecting in a single point $\x_E$. (Since here we work with embedded Heegaard diagrams, $E$ is a torus inside $Y$, disjoint from $\Sigma$, and we take their connect sum inside $Y$.) We equip the new curves with the coupled orientation induced from the one on $E\cong \alpha_E \times \beta_E$. We also equip them with the Lie group Pin structures $P^{\#}_{\alpha, E}$ and $P^{\#}_{\beta, E}$ (recall that canonical choices of Lie group Pin structures  exist in dimension $1$). Together, the coupled orientation and Pin structures produce the canonical coupled Spin structure $S^{\#}_{\can, E}$ on $(\alpha_E, \beta_E)$. It follows that the line $\ell(P^{\#}_{\alpha, E}, P^{\#}_{\beta, E})$ is canonically trivialized.

Starting from Pin structures $\Pa$ and $\Pb$ in the setting before stabilization, we can take the product of the coupled Spin structure $(\Pa, \Pb)$ with $S^{\#}_{\can, E}$ and obtain a coupled Spin structure $(P^{\#}_{\alpha'}, P^{\#}_{\beta'})$ after stabilization. Since canonical coupled Spin structures are preserved by products (cf. Remark~\ref{rem:cansum}), we have canonical isomorphisms
\begin{equation}
\label{eq:lll}
\ell(P^{\#}_{\alpha'}, P^{\#}_{\beta'}) \cong \ell(\Pa, \Pb) \otimes \ell(P^{\#}_{\alpha, E}, P^{\#}_{\beta, E}) \cong \ell(\Pa, \Pb).
\end{equation}

Recall from Proposition~\ref{prop:onlycoupled} that the orientation spaces $o(\x)$ depend on Pin structures only through their combined coupled Spin structure. Coupled Spin structures behave well with regard to products, and in fact we have isomorphisms
$$ o(\x) \otimes o(\x_E) \cong o(\x \times \x_E),$$
for all $\x \in \Ta \cap \Tb$. (Compare Example~\ref{ex:product}.)

The proof of stabilization invariance in \cite[Section 10]{HolDisk} is based on a neck stretching argument.  Ultimately, we identify the $J$-holomorphic strips on $\Sym^{g+1}(\Sigma')$ (contributing to the Heegaard Floer group defined from $\He'$) with products of those on $\Sym^g(\Sigma)$ (contributing to the Heegaard Floer group defined from $\He$) and on $E$ (with boundaries on $\alpha_E$ and $\beta_E$). We end up in the situation of  Example~\ref{ex:product} and obtain an isomorphism
\begin{equation}
\label{eq:lV}
 \HFmprel(\Tap, \Tbp ) \cong \HFmprel(\Ta, \Tb ) \otimes_{\Z[U]} \HFmprel(\alpha_E, \beta_E).
 \end{equation}
Strictly speaking, to be in the setting of Example~\ref{ex:product} one needs to lift the coupled Spin to Spin structures; but the resulting isomorphism does not depend on this lift.

Combining \eqref{eq:lV} with \eqref{eq:lll}, we also get an isomorphism
$$ \HFm(\Tap, \Tbp ) \cong \HFm(\Ta, \Tb ) \otimes_{\Z[U]} \HFm(\alpha_E, \beta_E).$$

Observe that $(E, \alpha_E, \beta_E)$ is a Heegaard diagram for $S^3$. The group $ \HFm(\alpha_E, \beta_E )\cong \HFm(S^3)$ is a rank one free module over $\Z[U]$ supported in grading $\grHF=0$, which by formula \eqref{eq:grHF} corresponds to $\gr=1$. Lemma~\ref{lem:dim2} (a) says that the line $o(\x_E) \otimes \ell(P^{\#}_{\alpha, E}, P^{\#}_{\beta, E})$ is canonically trivialized; so we can view $\HFm(\alpha_E, \beta_E )$ as being generated by $\x_E$. (Compare Proposition~\ref{prop:canonical}.) We obtain the desired stabilization isomorphism:
\begin{equation}
\label{eq:S}
 S=\cdot \otimes \x_E : \HFm(\Ta,\Tb) \xrightarrow{\cong} \HFm(\Tap, \Tbp).
 \end{equation}
 
The stabilization isomorphisms for the other variants of $\HF$ are constructed similarly.

\subsection{Invariance}
\label{sec:invariance}
We are now ready to prove a weak version of Theorem~\ref{thm:main}: we establish the existence of isomorphisms between Heegaard Floer homologies associated to different Heegaard diagrams for the same $3$-manifold. We leave the proof that this isomorphisms are canonical for the next section.

For simplicity, we will denote by $\HFcirc$ any of the variants of Heegaard Floer homology, $\circ \in \{\widehat{\phantom{u}}, +, -, \infty\}$.

\begin{proposition}
\label{prop:invariance}
Let $Y$ be a closed, oriented $3$-manifold equipped with a basepoint $z \in Y$ and a $\Spinc$ structure $\s$. Then, the isomorphism class of the Heegaard Floer homology $\HFcirc$ (defined using Pin structures on the Lagrangian tori) is an invariant of the pair $(Y, \s)$.
\end{proposition}

\begin{proof}
Following \cite{HolDisk}, we have to prove invariance under three kinds of moves on Heegaard diagrams: Hamiltonian isotopies of the alpha and beta curves, handleslides of the alpha and beta curves, and stabilizations. 

Hamiltonian isotopies of the alpha and beta curves induce Hamiltonian isotopies of the Lagrangian tori, and these preserve the Lie group Pin structures. Then, as usual in Floer theory, continuation maps associated to Hamiltonian isotopies give rise to isomorphisms between the Floer homologies.

Handleslides were discussed in Section~\ref{sec:handleslide}. We associate to them triangle maps of the form $\Psi^{\alphas}_{\betas \to \gammas}$, and Proposition~\ref{prop:hslideiso} shows that these are isomorphisms. Similarly, to stabilizations we associate the isomorphisms $S$ defined in Section~\ref{sec:stabs}.

So far we have only discussed pointed Heegaard moves, i.e. those supported away from the basepoint $z$. It is shown in \cite[Proposition 7.1]{HolDisk} that any two Heegaard diagrams representing the same $Y$ (possibly with different basepoints) become diffeomorphic after a finite sequence of pointed Heegaard moves. 
It follows that the isomorphism class of $\HFcirc$ is independent of $z$.
\end{proof}

\begin{remark}
If we were only interested in Proposition~\ref{prop:invariance} (and not in the naturality results discussed in the next section), there would be no need of coupled Spin structures or of adjusting the preliminary Heegaard Floer groups as in Section~\ref{sec:adjust}. We could have chosen any Lie group Pin structures on the Lagrangian tori, and worked with the preliminary Heegaard Floer groups. Indeed, up to (non-canonical) isomorphism, these are the same as the adjusted groups.
\end{remark}

%%% Local Variables:
%%% mode: latex
%%% TeX-master: "signs"
%%% End:

\section{Naturality}
\label{sec:naturality}
To complete the proof of Theorem~\ref{thm:main}, it remains to show that the isomorphisms induced by (pointed) Heegaard moves on $\HFcirc$ are natural. In other words, if we go from one Heegaard diagram $\H$ to another diagram $\H'$ by a sequence of moves, the resulting isomorphism between the Heegaard Floer homologies of $\H$ and $\H'$ does not depend on the sequence of moves we chose. 

In \cite{JTZ}, Juh\'asz, Thurston and Zemke gave a list of conditions that need to be checked to ensure that naturality holds for a $3$-manifold invariant defined from Heegaard diagrams. We review here some of their set-up. While they work with sutured manifolds, we will restrict here to based $3$-manifolds, for simplicity; this is the class $\mathcal{S}_{\operatorname{man}}$ in their notation.

\begin{remark}
When talking about naturality, it is important to fix the basepoint $z$. Even over $\F_2$, when one moves $z$ to another basepoint $z'$ by following a path on $Y$, there is an isomorphism between the Floer homologies that depends on this path. See \cite{ZemkeHat}, \cite{ZemkeHF}. 
\end{remark}

\subsection{Strong Heegaard invariants}
We define an {\em isotopy diagram} to be an equivalence class of Heegaard diagrams (for based $3$-manifolds), where two diagrams are equivalent if the underlying Heegaard surfaces are the same, their alpha curves differ by isotopies, and their beta curves differ by isotopies. Further, we say that two isotopy diagrams are {\em $\alpha$-equivalent} if they differ by a sequence of $\alpha$-handleslides. We define {\em $\beta$-equivalence} similarly. (Of course, all isotopies and handleslides here are supposed to avoid the basepoint.)

We let $\G$ be the oriented (multi-)graph whose vertices are all isotopy diagrams, and whose edges are associated to diagram moves: $\alpha$-equivalences, $\beta$-equivalences, stabilizations, destabilizations, and diffeomorphisms. If we restrict to only one kind of these moves, we denote the respective subgraphs (with the same vertices as $\G$) by $\G_\alpha$, $\G_{\beta}$, $\Gstab$ and $\Gdiff$.

\begin{definition}
\label{def:dr}
A {\em distinguished rectangle} in $\G$ is a subgraph
$$ \xymatrix{
\H_1 \ar[d]^f \ar[r]^e &\H_2\ar[d]^g\\
\H_3 \ar[r]^h &\H_4}$$
such that one of the following holds:
\begin{enumerate}
\item
Both $e$ and $h$ are $\alpha$-equivalences, whereas $f$ and $g$ are $\beta$-equivalences; 
\item
Both $e$ and $h$ are equivalences of the same type ($\alpha$ or $\beta$), whereas $f$ and $g$ are stabilizations;
\item
Both $e$ and $h$ are equivalences of the same type, while $f=g$ is a diffeomorphism.
\item
All of $e$, $f$, $g$ and $h$ are stabilizations, with $e$ and $h$ both consisting of replacing the same disk $D_1$ in the Heegaard surface with the same punctured torus $T_1$, and similarly $f$ and $g$ both consisting of replacing a disk $D_2$ with a punctured torus $T_2$. We require that $D_1 \cap D_2 = \emptyset$ and $T_1 \cap T_2 =\emptyset$;
\item Both $e$ and $h$ are stabilizations, whereas $f$ and $g$ are diffeomorphisms, taking one stabilization to the other.
\end{enumerate}
\end{definition}

We also need the notion of a {\em simple handleswap}, which is a triangle in $\G$ consisting of an $\alpha$-equivalence (a handleslide), a $\beta$-equivalence (another handleslide), and a diffeomorphism, all being the identity except on a punctured genus two surface, where they look like in Figure~\ref{fig:handleswap}. 

\begin{figure}
{
\fontsize{10pt}{11pt}\selectfont
   \def\svgwidth{5in}
   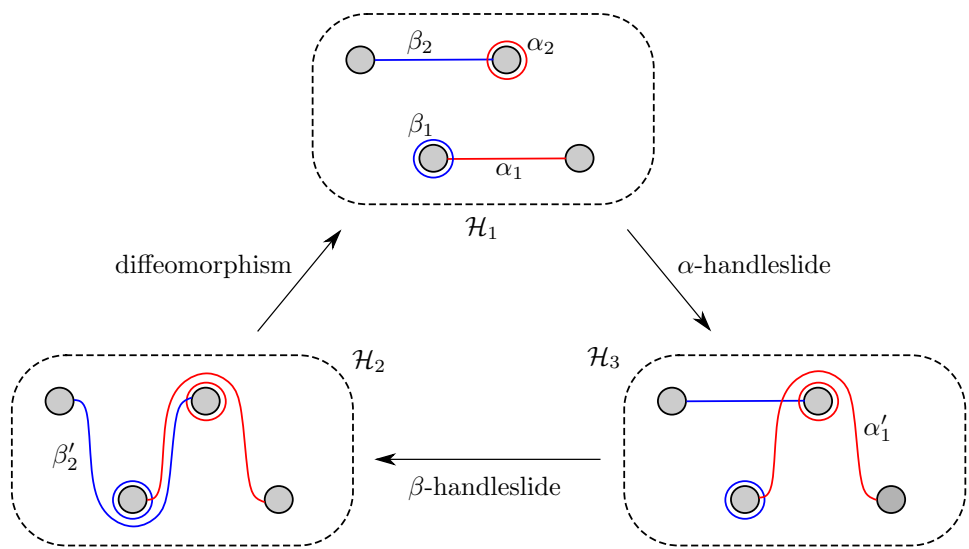
}
\caption{A simple handleswap. We reproduce here Figure 4 in \cite{JTZ}. The dashed curve is the boundary of a genus two surface obtained by identifying the boundaries of the grayed circles in pairs. (In each diagram, the circles at the same height are identified via reflection in a vertical axis). The $\alpha$ circles are in red and the $\beta$ circles in blue.}
\label{fig:handleswap}
\end{figure}

\begin{definition}
\label{def:shi}
Let $\Cat$ be a category. A {\em strong Heegaard invariant} is an assignment $F: \G \to \Cat$ that takes the vertices of $\G$ (isotopy diagrams) to objects in $\Cat$, and edges of $\G$ to isomorphisms in $\Cat$, with the following properties:
  \begin{enumerate}
  \item \label{functoriality} {\bf Functoriality:} The restrictions of $F$ to  $\G_\alpha$, $\G_{\beta}$, and $\Gdiff$ are functors. Further, if $e$ is a stabilization and $e'$ the corresponding destabilization, then $F(e') = F(e)^{-1}$.
  \item \label{commutativity} {\bf Commutativity}: For every distinguished rectangle in $\G$, applying $F$ yields a commutative diagram in $\Cat$;
  \item \label{continuity}{\bf Continuity:} If $e$ is an edge in $\G$ associated to a diffeomorphism $\psi$ from the same diagram to itself, and $\psi$ is isotopic to the identity, then $F(e)$ is the identity; 
  \item \label{handleswap} {\bf Handleswap invariance:}  For every simple handleswap 
  $$ \xymatrix{
\H_1 \ar[dr]^e & \\
\H_3 \ar[u]^g &\H_2 \ar[l]^f}$$
the composition $F(g) \circ F(f) \circ F(e)$ is the identity.
  \end{enumerate}
\end{definition}

Juh\'asz, Thurston and Zemke proved that strong Heegaard invariants are natural:
\begin{theorem}[Theorem 2.38 in \cite{JTZ}]
\label{thm:strongH}
Let $F: \G \to \Cat$ be a strong Heegaard invariant. If two isotopy diagrams $\H_1$ and $\H_2$ are related by a path of arrows in $\G$, then the isomorphism from $F(\H_1)$ to $F(\H_2)$ induced by composing the morphisms along the path is independent of the choice of this path. 
\end{theorem}

In \cite[Theorem 2.33 (2)]{JTZ}, they further proved that the $\HFm$ for $\circ \in \{\widehat{\phantom{u}}, +, -, \infty\}$, defined without signs, are strong Heegaard invariants into the category of $\F_2[U]$-modules. We need to show the same thing with signs, i.e., that they are strong Heegaard invariants into the category of $\Z[U]$-modules. (A weaker version of this claim, in the projective category of $\Z[U]$-modules where morphisms are defined up to a sign, was proved by Gartner in \cite{Gartner}.)

\subsection{Loops of handleslides}
In our context it is helpful to decompose each $\alpha$- and $\beta$-equivalence $f$ into a sequence of handleslides, and define $F(f)$ as the composition of maps associated to those handleslides. In \cite[Appendix A]{JTZ}, Juh\'asz, Thurston and Zemke describe a set of conditions that produce a strong Heegaard invariant, using maps induced by handleslides instead of equivalences. More details for their proofs appear in the work of Qin \cite{Qin}.

First, if $f$ is an equivalence, to prove that the map $F(f)$ is well-defined (that is, it does not depend on how we present $f$ as a composition of handleslides), we would need to show that a loop of handleslides produces the identity. Definition A.3 and Proposition A.5 in \cite{JTZ} give a finite set of types of such loops that generate all others. To describe them, let us first define an {\em attaching set} on a based surface $(\Sigma, z)$ to be an isotopy class of an unordered collection of $g$ disjoint simple closed curves $\beta_1, \dots, \beta_g$ on $\Sigma \setminus \{z\}$ that are linearly independent in $H_1(\Sigma)$. For example, the collections $\alphas$ and $\betas$ in a Heegaard diagram are attaching sets. (Unlike in \cite{JTZ}, here we pick $\beta$ instead of $\alpha$ as the notation for a typical attaching set. This is because $\beta$-handleslides, where we fix the $\alpha$ curves, are more commonly considered in Heegaard Floer theory; e.g. in \cite{HolDisk} and in Section~\ref{sec:handleslide} of this paper. Of course, this convention is of little importance.) 

\begin{definition}
\label{def:hloops}
A {\em handleslide loop} is one of the following sequences of attaching sets on $\Sigma$ connected by handleslides:
\begin{enumerate}
\item \label{slide1} A {\em slide triangle}, formed of three attaching sets of the form
$$ \xymatrix{ \{\beta_1, \beta_2\} \cup \vbeta \ar@{-}[rr] \ar@{-}[rd] & & \{\beta_2, \beta_3\} \cup \vbeta \ar@{-}[ld]\\
&  \{ \beta_1, \beta_3 \} \cup \vbeta  & }$$
where  $\beta_1, \beta_2, \beta_3$ bound a pair-of-pants, and $\vbeta$ is a fixed collection of $g-2$ curves.

\item \label{slide2} A {\em commuting slide square}, formed of four attaching sets of the form
$$ \xymatrix{ \{\beta_1, \beta_2, \beta_3, \beta_4 \} \cup \vbeta \ar@{-}[r] \ar@{-}[d] & \{\beta_1', \beta_2, \beta_3, \beta_4 \} \cup \vbeta \ar@{-}[d]\\
\{\beta_1, \beta_2, \beta_3', \beta_4 \} \cup \vbeta \ar@{-}[r] & \{\beta_1', \beta_2, \beta_3', \beta_4 \} \cup \vbeta }$$
where $\beta_1'$ is obtained by sliding $\beta_1$ over $\beta_2$, and $\beta_3'$ is obtained by sliding $\beta_3$ over $\beta_4$. Here, $\vbeta$ is a fixed collection of $g-4$ curves. 

\item \label{slide3} A square of the form 
$$ \xymatrix{ \{\beta_1, \beta_2, \beta_3 \} \cup \vbeta \ar@{-}[r] \ar@{-}[d] & \{\beta_1', \beta_2, \beta_3 \} \cup \vbeta \ar@{-}[d]\\
\{\beta_1'', \beta_2, \beta_3 \} \cup \vbeta \ar@{-}[r] & \{\beta_1''', \beta_2, \beta_3', \beta_4 \} \cup \vbeta }$$
where $\beta_1'$ is obtained by sliding $\beta_1$ over $\beta_2$, whereas $\beta_1''$ is obtained by sliding $\beta_1$ over $\beta_3$, and $\beta_1'''$ is obtained by sliding $\beta_1'$ over $\beta_3$ (or, equivalently, sliding $\beta_1''$ over $\beta_2$). Here, $\vbeta$ is a fixed collection of $g-3$ curves.

\begin{figure}
{
\fontsize{10pt}{11pt}\selectfont
   \def\svgwidth{6in}
   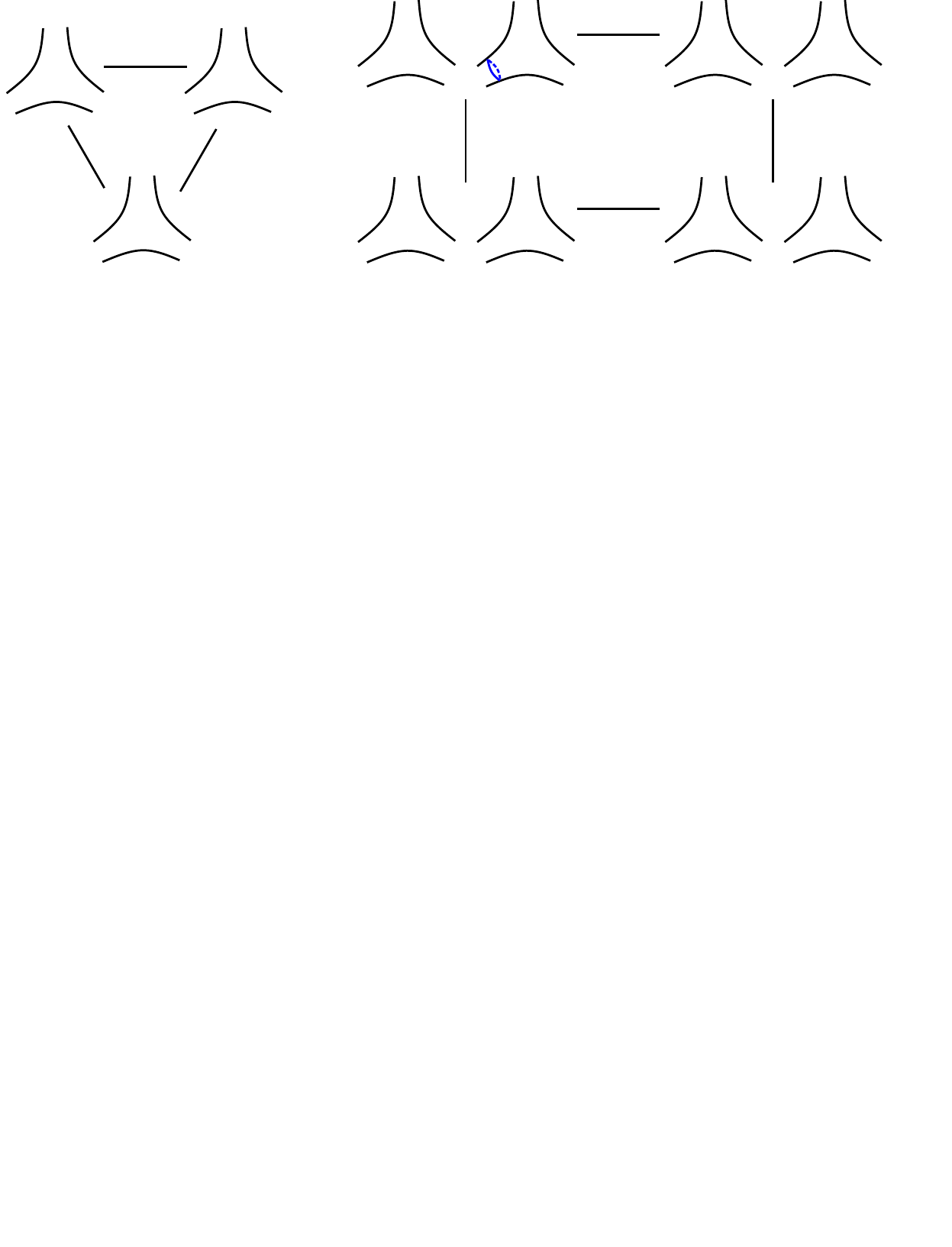
}
\caption{The six types of handleslide loops.}
\label{fig:loops}
\end{figure}

\item \label{slide4} A square of the form 
$$ \xymatrix{ \{\beta_1, \beta_2, \beta_3 \} \cup \vbeta \ar@{-}[r] \ar@{-}[d] & \{\beta_1', \beta_2, \beta_3  \} \cup \vbeta \ar@{-}[d]\\
\{\beta_1, \beta_2, \beta_3' \} \cup \vbeta \ar@{-}[r] & \{\beta_1', \beta_2, \beta_3' \} \cup \vbeta }$$
where $\beta_1'$ is obtained by sliding $\beta_1$ over $\beta_2$, and $\beta_3'$ is obtained by sliding $\beta_3$ over $\beta_2$, so that the two slides are happening along arcs that reach $\beta_2$ from opposite sides. Here, $\vbeta$ is a fixed collection of $g-3$ curves.

\item \label{slide5} A square of the form 
$$ \xymatrix{ \{\beta_1, \beta_2 \} \cup \vbeta \ar@{-}[r] \ar@{-}[d] & \{\beta_1', \beta_2  \} \cup \vbeta \ar@{-}[d]\\
\{\beta_1'', \beta_2 \} \cup \vbeta \ar@{-}[r] & \{\beta_1''', \beta_2 \} \cup \vbeta }$$
where $\beta_1'$ and $\beta_1''$ are both obtained by sliding $\beta_1$ over $\beta_2$, but from opposite sides; and $\beta_1'''$ is obtained from $\beta_1'$ by doing both of these slides over $\beta_2$. Here, $\vbeta$ is a fixed collection of $g-2$ curves.

\item \label{slide6} A pentagon of the form 
$$ 
\xymatrixcolsep{1mm}
\xymatrix{ &  \{\beta_1, \beta_2, \beta_3 \} \cup \vbeta \ar@{-}[dl] \ar@{-}[dr] &\\
 \{\beta_1', \beta_2, \beta_3  \} \cup \vbeta \ar@{-}[d] &  &\{\beta_1, \beta_2', \beta_3 \} \cup \vbeta \ar@{-}[d] \\
\{\beta_1'', \beta_2, \beta_3 \} \cup \vbeta  \ar@{-}[rr]&  & \{\beta_1'', \beta_2', \beta_3 \} \cup \vbeta 
}$$
where $\beta_1'$ is obtained from $\beta_1$ by sliding over $\beta_2$, whereas $\beta_2'$ is obtained from $\beta_2$ by sliding it over $\beta_3$, and $\beta_1''$ is obtained from $\beta_1'$ by sliding it over $\beta_3$ (or, equivalently, from $\beta_1$ by sliding it over $\beta_2'$). Here, $\vbeta$ is a fixed collection of $g-3$ curves.

\end{enumerate}
See Figure~\ref{fig:loops}, which is based on Figures 21 and 22 in \cite{JTZ}.
\end{definition}

We also need one other loop, which involves stabilizations in addition to handleslides. This is the stabilization slide from \cite[Definition 7.7]{JTZ}. It was described there as a triangle where two edges are stabilizations, and one is an equivalence. Since that equivalence is the composition of two handleslides, in our context the triangle becomes a square.

\begin{definition}
\label{def:sslide}
A {\em stabilization $\beta$-slide} is a square composed of four Heegaard diagrams that differ locally as in Figure~\ref{fig:sslide}. A {\em stabilization $\alpha$-slide} is similar, with the alpha and beta curves reversed.
\end{definition}

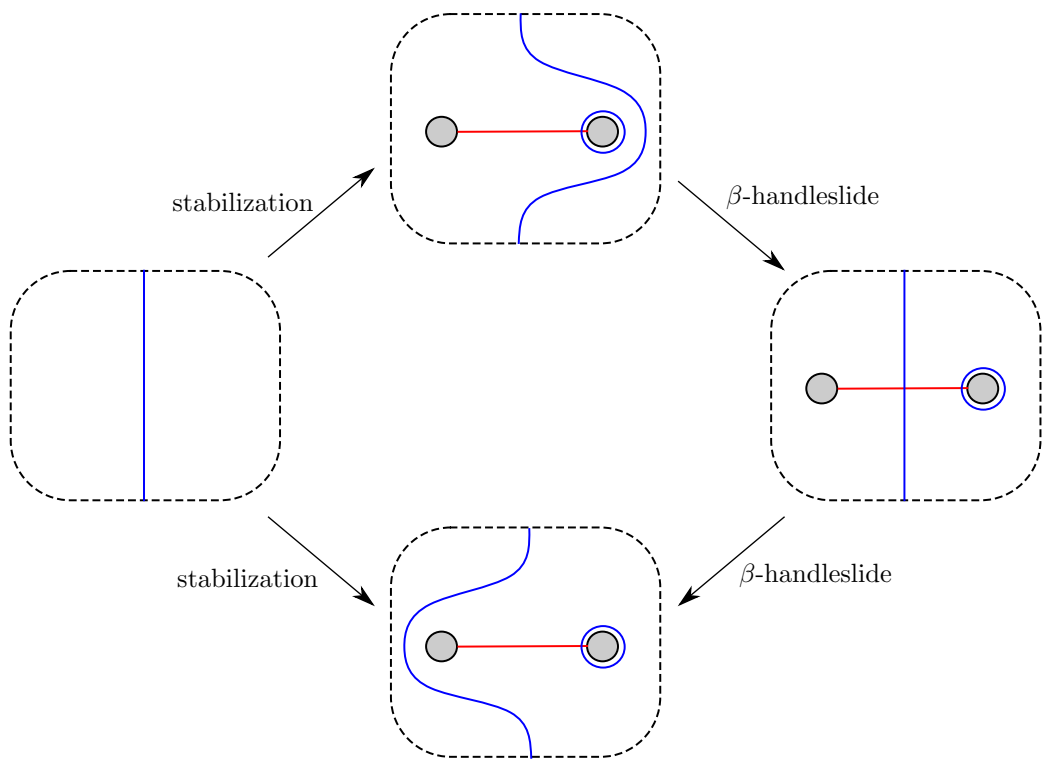
\begin{figure}
{
\fontsize{10pt}{11pt}\selectfont
   \def\svgwidth{5.4in}
   %% Creator: Inkscape 1.3.2 (091e20e, 2023-11-25), www.inkscape.org
%% Accompanies image file 'sslide.pdf' (pdf, eps, ps)
\begingroup%
  \makeatletter%
  \providecommand\color[2][]{%
    \errmessage{(Inkscape) Color is used for the text in Inkscape, but the package 'color.sty' is not loaded}%
    \renewcommand\color[2][]{}%
  }%
  \providecommand\transparent[1]{%
    \errmessage{(Inkscape) Transparency is used (non-zero) for the text in Inkscape, but the package 'transparent.sty' is not loaded}%
    \renewcommand\transparent[1]{}%
  }%
  \providecommand\rotatebox[2]{#2}%
  \newcommand*\fsize{\dimexpr\f@size pt\relax}%
  \newcommand*\lineheight[1]{\fontsize{\fsize}{#1\fsize}\selectfont}%
  \ifx\svgwidth\undefined%
    \setlength{\unitlength}{569.20490415bp}%
    \ifx\svgscale\undefined%
      \relax%
    \else%
      \setlength{\unitlength}{\unitlength * \real{\svgscale}}%
    \fi%
  \else%
    \setlength{\unitlength}{\svgwidth}%
  \fi%
  \global\let\svgwidth\undefined%
  \global\let\svgscale\undefined%
  \makeatother%
  \begin{picture}(1,0.72333806)%
    \lineheight{1}%
    \setlength\tabcolsep{0pt}%
    \put(0.15774622,0.5315058){\makebox(0,0)[lt]{\lineheight{1.25}\smash{\begin{tabular}[t]{l}stabilization\end{tabular}}}}%
    \put(0.16242987,0.16627399){\makebox(0,0)[lt]{\lineheight{1.25}\smash{\begin{tabular}[t]{l}stabilization\end{tabular}}}}%
    \put(0.69456629,0.53815418){\makebox(0,0)[lt]{\lineheight{1.25}\smash{\begin{tabular}[t]{l}$\beta$-handleslide\end{tabular}}}}%
    \put(0.7069024,0.17184862){\makebox(0,0)[lt]{\lineheight{1.25}\smash{\begin{tabular}[t]{l}$\beta$-handleslide\end{tabular}}}}%
    \put(0,0){\includegraphics[width=\unitlength,page=1]{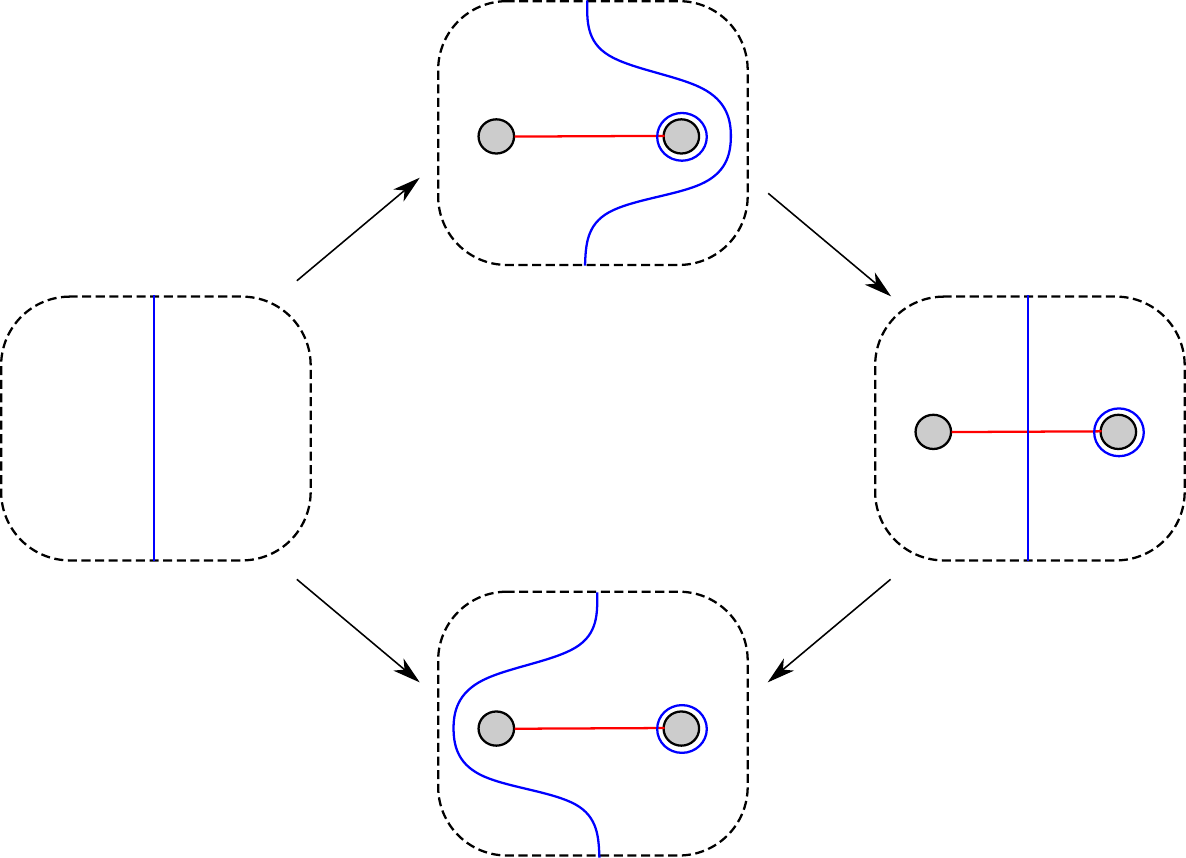}}%
  \end{picture}%
\endgroup%

}
\caption{A stabilization slide. The two handleslides are performed along the same $\beta$-curve, but the second is done through the handle.}
\label{fig:sslide}
\end{figure}

We let $\G'$ be the oriented (multi-)graph whose vertices are all isotopy diagrams, and whose edges are associated to diagram moves: $\alpha$-handleslides, $\beta$-handleslides, stabilizations, destabilizations, and diffeomorphisms. Note that $\G'$ is a subgraph of $\G$, with the same vertices but fewer edges. 

\begin{theorem}[Theorem A.6 in \cite{JTZ}; Theorem 1.3 in \cite{Qin}]
\label{thm:shi2}
Let $\Cat$ be a category. Suppose we are given an assignment $F: \G' \to \Cat$ that takes the vertices of $\G$ (isotopy diagrams) to objects in $\Cat$, and edges of $\G'$ to isomorphisms in $\Cat$, with the following properties:
  \begin{enumerate}
  \item {\bf Functoriality:} The restriction of $F$ to $\Gdiff$ (which is a subgraph of $\G'$) is a functor. Further, if $e$ is a stabilization and $e'$ the corresponding destabilization, then $F(e') = F(e)^{-1}$.
\item \label{cdr} {\bf Commutativity along distinguished rectangles}: For every distinguished rectangle in $\G'$ (i.e., such that the only equivalences involved are handleslides), applying $F$ yields a commutative diagram in $\Cat$;
 \item {\bf Continuity:} condition \eqref{continuity} in Definition~\ref{def:shi};
 \item {\bf Handleswap invariance:} condition \eqref{handleswap} in Definition~\ref{def:shi};
 \item {\bf Commutativity along handleslide loops:} $F'$ commutes along any of the handleslide loops in Definition~\ref{def:hloops};
 \item{\bf Commutativity along stabilization slides:}  $F'$ commutes along any stabilization slide (see Definition~\ref{def:sslide}).
 \end{enumerate}
 Then, $F'$ extends uniquely to a strong Heegaard invariant $F: \G \to \Cat$.
\end{theorem}

\begin{remark}
The original stabilization slide from \cite[Definition 7.7]{JTZ} is an example of a distinguished rectangle in $\G$, where one horizontal edge is trivial, the other is the composition of the two handleslides, and the vertical edges are the stabilizations. Since one edge is a composition of two handleslides rather than a single handleslide, the stabilization slide does not count as a distinguished rectangle in $\G'$. This is the reason for listing it separately from condition~\eqref{cdr}.
\end{remark}

\subsection{Heegaard Floer homology as a strong Heegaard invariant}

\begin{proof}[Proof of Theorem~\ref{thm:main}] We focus on the minus version. In view of Theorem~\ref{thm:strongH}, it suffices to show that $\HFm$ is a strong Heegaard invariant. Juh\'asz, Thurston and Zemke did this in \cite[Theorem 2.33 (2)]{JTZ} for $\HFm$ as an $\F_2[U]$-module, by directly checking the conditions in Definition~\ref{def:shi}. Here we work with $\Z[U]$-modules, and we will check the conditions in  Theorem~\ref{thm:shi2} instead. 

We view $\HFm$ as an assignment from $\G'$ to $\Z[U]$-modules. Notice that in Section~\ref{sec:definition}, $\HFm$ is defined on (admissible) Heegaard diagrams, rather than isotopy diagrams. We define it on an isotopy diagram as the colimit of the transitive system of Heegaard Floer homologies over all admissible Heegaard diagrams that produce the given isotopy diagram; see \cite[Definition 9.19]{JTZ}. To construct the colimit we use the continuation maps defined in \cite[Section 7.3]{HolDisk}, which we denote by $\Gamma^{\alphas}_{\betas\to \betas'}$ or $\Gamma_{\betas}^{ \alphas \to \alphas'}$. 

To have a well-defined colimit, we further need to ensure that the continuation maps relating Heegaard Floer homologies give rise to a transitive system in the sense of \cite[Definition 1.1]{JTZ}; i.e., that loops of isotopies induce the identity on Heegaard Floer homology. We can decompose such loops into smaller (triangular) loops such that the quadruple Heegaard diagram for each loop is admissible; this can be done as in the proofs of Lemmas 9.7, 9.10, 9.12 in \cite{JTZ}, which in turn are based on Lemma 9.5 in \cite{JTZ}. A typical such triangular loop is 
  $$ \xymatrix{
(\Sigma, \alphas, \betas, z) \ar[rr] &  & (\Sigma, \alphas, \betas', z) \ar[dl] \\
& (\Sigma, \alphas, \betas'', z) \ar[lu]& 
}$$
where the quadruple diagram $(\Sigma, \alphas, \betas, \betas', \betas'', z)$ is admissible. We claim that 
\begin{equation}
\label{eq:isotopycont}
\Gamma^{\alphas}_{\betas''\to \betas} \circ \Gamma^{\alphas}_{\betas'\to \betas''} \circ \Gamma^{\alphas}_{\betas\to \betas'} = \id.
 \end{equation}
For a punctured surface $(\Sigma, z)$ of genus $g > 0$, the identity component of its diffeomorphism group is contractible; see \cite{EarleEells}, \cite{EarleSchatz}. It follows that the composition of isotopies taking $\betas \to \betas'\to \betas'' \to \betas$ is homotopic to the identity. Admissibility ensures that we can associate a continuation map to this homotopy, which induces a chain homotopy between the maps on Heegaard Floer complexes and the identity; passing to homology, we obtain \eqref{eq:isotopycont}. 

The other kind of triangular loop, where the alpha curves vary, is entirely analogous. This completes the proof that we have a transitive system, and therefore $\HFm$ is well-defined on isotopy diagrams.

\begin{remark}
In \cite[Section 9.1]{JTZ}, to define $\HFm$ over  $\F_2[U]$ for isotopy diagrams, the authors proceeded differently: they assigned triangle maps to isotopies, and they also showed that these triangle maps  are the same as continuation maps. They proved that the triangle maps form a transitive system by making use of the uniqueness of top degree generators of the form $\Theta_{\beta,\beta'}$ or $\Theta_{\alpha,\alpha'}$ over $\F_2[U]$. Over $\Z[U]$, for arbitrary isotopies, it is unclear which triangle maps to consider, because there are two choices (differing by a sign) for the top degree generators.
\end{remark}

Next, we need to assign maps to the edges in $\G'$ (i.e. to moves between isotopy diagrams). For an $\alpha$- or $\beta$-handleslide, we use the maps $\Psi^{\alphas \to \gammas}_{\betas}$ and $\Psi^{\alphas}_{\betas \to \gammas}$, constructed in Section~\ref{sec:handleslide}.  To a diffeomorphism between diagrams we assign the obvious identification of Floer homologies. To a stabilization we assign the map $S$ from Equation~\eqref{eq:S}, and to a destabilization its inverse.

\begin{convention}
\label{conv:data}
Throughout this proof, when discussing Heegaard moves appearing in a multiple Heegaard diagram, we will draw pictures where each curve is oriented, and there is an ordering of the curves in each attaching set. The orientation choices will be made carefully, so as to be compatible with how the coupled orientations for handleslides and stabilizations were defined in Sections~\ref{sec:handleslide} and \ref{sec:stabs}. (Reversing all the orientations at the same time would produce an equally valid picture.) We will then have data to specify Spin structures on each Lagrangian, as in Definition~\ref{def:chooseab}. This will enable us to appeal to the product principle from Example~\ref{ex:product}, just as we did in the proof of Proposition~\ref{prop:hslideiso}.
\end{convention}

We now proceed to check the conditions in Theorem~\ref{thm:shi2}. 

\medskip
{\em(1)} {\bf Functoriality:} This is immediate from the definitions.

\medskip
{\em(2)} {\bf Commutativity along distinguished rectangles:} There are five types of such rectangles to check; see Definition~\ref{def:dr}. Type (1) is commutation between $\alpha$- and $\beta$-handleslides, i.e., a relation of the form
$$ \Psi^{\alphas \to \alphas'}_{\betas'} \circ \Psi^{\alphas}_{\betas\to \betas'} = \Psi^{\alphas'}_{\betas \to \betas'} \circ \Psi^{\alphas \to \alphas'}_{\betas}.$$
This follows from the $A_{\infty}$ relations for polygon maps; compare \cite[Proposition 9.10 (3)]{JTZ}. 
The proofs of commutativity along the other four types of distinguished rectangles are just as in \cite[Section 9.2]{JTZ}. The arguments are based on choosing specific almost complex structures suitable for diffeomorphisms and stabilizations, and do not involve the signs in an essential way. The only thing of note is for type (4), which is the commutation of two stabilizations; there, we use the fact that coupled Spin structures behave well with regard to direct sums (see Remark~\ref{rem:cansum}), and in particular that the result of two direct sums does not depend on their  order. 

%The argument involves reducing to the case where the quadruple diagram $(\Sigma, \alphas, \alphas', \betas, \betas')$ is admissible. This reduction is done by introducing intermediate diagrams so that the quadruple admissibility holds. Then they use commutativity of $\alpha$-equivalences among themselves, and of $\beta$-handleslides among themselves; in our case, this is replaced by commutativity along handlelslide loops, which we will prove momentarily. In the case where the quadruple diagram is admissible, commutativity follows from the $A_{\infty}$ relations for polygon maps. The same $A_{\infty}$ relations hold over $\Z[U]$. 

\medskip
{\em (3)} {\bf Continuity:} The proof of this over $\F_2[U]$ in \cite[Proposition 9.27]{JTZ} is based on relating continuation maps to changes in the almost complex structure, and works just as well over $\Z[U]$.

\medskip
{\em(4)} {\bf Handleswap invariance:} The proof of invariance under a simple handleswap in \cite[Section 9.3]{JTZ} (over $\F_2[U]$) is based on degenerations (neck-stretching) to reduce it to calculations in the genus $2$ picture in Figure~\ref{fig:handleswap}.  The degenerations involve passing to Lipshitz's cylindrical reformulation of Heegaard Floer homology \cite{LipshitzCyl}. These degeneration arguments can be extended to the signed case (over $\Z[U]$); for an explanation, see Section~\ref{sec:cyl} below.

This reduces the problem to several curve counts in the genus $2$ picture. They involve understanding the effect of the two handleslides in Figure~\ref{fig:handleswap}. Let us discuss the $\alpha$-handleslide. (The $\beta$-handleslide is treated similarly.)  The  $\alpha$-handleslide is shown in Figure~\ref{fig:alphaswap}, which is based on Figure 59 in \cite{JTZ}. We have a triple pointed Heegaard diagram $(\Sigma_0, \alphas_0', \alphas_0, \betas_0, p_0)$, where $\Sigma_0$ is a surface of genus two with basepoint $p_0$ and three attaching sets: $$\alphas_0'=\{\alpha_1', \alpha_2'\}, \ \ \alphas_0=\{\alpha_1, \alpha_2\}, \ \ \betas_0 =\{\beta_1, \beta_2\}.$$
Let
$$ \aa = \{a_1, a_2\} = \T_{\alpha_0} \cap \T_{\beta_0}, \ \ \bb=\{b_1, b_2\} = \T_{\alpha'_0} \cap \T_{\beta_0}.$$ 
Let also $\Theta= \{\theta_1, \theta_2\} \in \T_{\alpha'_0} \cap \T_{\alpha_0}$ be the maximal degree generator, as shown in Figure~\ref{fig:alphaswap}.

We equip the curves on $\Sigma_0$ with orientations and orderings as shown in the figure, following Convention~\ref{conv:data}.

The triangle map for the $\alpha$-handleslide is computed (mod 2) in Proposition 9.31 of \cite{JTZ}. We claim that the same calculation holds over $\Z$. In the $\F_2$ calculation, there are three points at which explicit moduli spaces are counted (mod 2). We discuss them in turn, explaining how to refine the counts to get the answer in $\Z$.

\begin{figure}
{
\fontsize{10pt}{11pt}\selectfont
   \def\svgwidth{3in}
   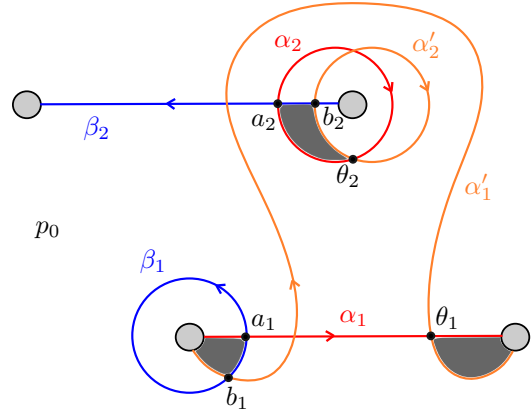
}
\caption{The $\alpha$-handleslide from Figure~\ref{fig:handleswap}. }
\label{fig:alphaswap}
\end{figure}

The first count appears in Lemma 9.52 in \cite{JTZ}. The lemma states that the differential on $\CFhat(\Sigma_0, \alphas_0', \alphas_0, p_0)$  vanishes. This is true in the signed case as well, because the rank of the Floer homology is $4$, according to Lemma~\ref{lem:S1S2}.

Another explicit count is in Lemma 9.53 in \cite{JTZ}, which states that the map
$$ \Psi_{\betas_0}^ {\alphas_0 \to \alphas_0'} : \CFhat(\Sigma_0, \alphas_0, \betas_0, p_0) \to \CFhat(\Sigma_0, \alphas_0', \betas_0, p_0)$$
satisfies $\Psi_{\betas_0}^ {\alphas_0 \to \alphas_0'} (\aa) = \bb.$ There is a unique holomorphic triangle (in the symmetric product) connecting $\Theta, \aa$ and $\bb$, shown in Figure~\ref{fig:alphaswap} as the product of the two darkly shaded triangles. Each of these two triangles has edges oriented as in  Figure~\ref{fig:triangles} (b) (with $\alpha_i'$, $\alpha_i$, $\beta_i$ playing the roles of $\alpha'$, $\beta'$, $\gamma$'), and therefore (according to Example~\ref{ex:abcp})  its sign is $+1$. It follows from Example~\ref{ex:product} that the product of the triangles has sign $+1$ as well, so we indeed have $\Psi_{\betas_0}^ {\alphas_0 \to \alphas_0'} (\aa) = \bb.$

One last count is at the very end of the proof of Lemma 9.58 in \cite{JTZ}. There, the authors compute that 
\begin{equation}
\label{eq:mad}
 \# \M_{(\aa, \aa)}(d) \equiv 1 \pmod{2},
 \end{equation}
where $\M_{(\aa, \aa)}(d)$ is the moduli space of Maslov index $2$ holomorphic curves $u$ on $(\Sigma_0, \alphas_0, \betas_0)$ satisfying $u(d)=p_0$, with $d \in [0,1] \times \R$ being any fixed point. In our setting, we claim that
\begin{equation}
\label{eq:mad2}
\# \M_{(\aa, \aa)}(d)= 1.
\end{equation}

 The proof of \eqref{eq:mad} in \cite{JTZ} is based on an argument from \cite[Appendix A]{LipshitzCyl}. We adapt it here to our purposes. We consider the Heegaard diagrams for $S^1 \times S^2$ shown in Figure~\ref{fig:doublestab}. On the left we have the class $\phi$ of a bigon, with $\Mhat(\phi)$ consisting of a single point; this can be positive or negative, depending on the coupled orientations and Pin structures on the curves there. On the right we have the double stabilization of $\phi$, which we denote by $\phi''$. Stabilization invariance (as discussed in the proof of Proposition~\ref{prop:invariance}) implies that
 $$ \# \Mhat(\phi'') = \# \Mhat(\phi).$$
 \begin{figure}
{
\fontsize{10pt}{11pt}\selectfont
   \def\svgwidth{4.5in}
   %% Creator: Inkscape 1.3.2 (091e20e, 2023-11-25), www.inkscape.org
%% Accompanies image file 'doubleslide.pdf' (pdf, eps, ps)
\begingroup%
  \makeatletter%
  \providecommand\color[2][]{%
    \errmessage{(Inkscape) Color is used for the text in Inkscape, but the package 'color.sty' is not loaded}%
    \renewcommand\color[2][]{}%
  }%
  \providecommand\transparent[1]{%
    \errmessage{(Inkscape) Transparency is used (non-zero) for the text in Inkscape, but the package 'transparent.sty' is not loaded}%
    \renewcommand\transparent[1]{}%
  }%
  \providecommand\rotatebox[2]{#2}%
  \newcommand*\fsize{\dimexpr\f@size pt\relax}%
  \newcommand*\lineheight[1]{\fontsize{\fsize}{#1\fsize}\selectfont}%
  \ifx\svgwidth\undefined%
    \setlength{\unitlength}{493.222291bp}%
    \ifx\svgscale\undefined%
      \relax%
    \else%
      \setlength{\unitlength}{\unitlength * \real{\svgscale}}%
    \fi%
  \else%
    \setlength{\unitlength}{\svgwidth}%
  \fi%
  \global\let\svgwidth\undefined%
  \global\let\svgscale\undefined%
  \makeatother%
  \begin{picture}(1,0.44730102)%
    \lineheight{1}%
    \setlength\tabcolsep{0pt}%
    \put(0.15892671,0.00530232){\makebox(0,0)[lt]{\lineheight{1.25}\smash{\begin{tabular}[t]{l}$(a)$\end{tabular}}}}%
    \put(0.79111838,0.00530232){\makebox(0,0)[lt]{\lineheight{1.25}\smash{\begin{tabular}[t]{l}$(b)$\end{tabular}}}}%
    \put(0,0){\includegraphics[width=\unitlength,page=1]{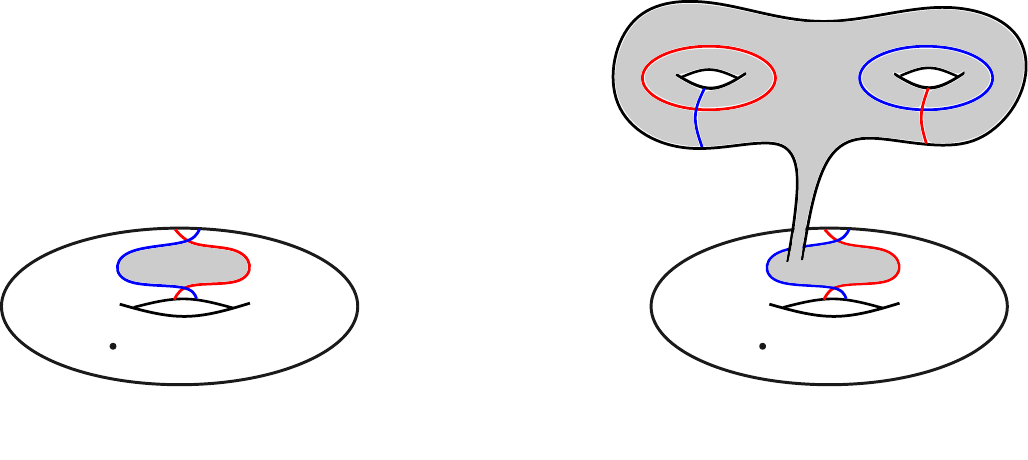}}%
  \end{picture}%
\endgroup%

}
\caption{(a) A bigon on a diagram for $S^1 \times S^2$. (b) Its double stabilization.}
\label{fig:doublestab}
\end{figure}
  On the other hand, by stretching the neck in the cylindrical reformulation, we see that for large neck length we must have
  $$  \Mhat(\phi'') =  \M_{(\aa, \aa)}(d)  \times \Mhat(\phi).$$
This proves \eqref{eq:mad2}. 

The rest of the arguments in the proof of handleswap invariance go through just as in \cite[Section 9.3]{JTZ}.

\medskip
{\em(5)} {\bf Commutativity along handleslide loops:} We need to check each of the six types of loops from Figure~\ref{fig:loops}. The first type is the most complicated, so we discuss the others first.

Consider a loop of type~\eqref{slide2}. We change the four attaching sets appearing in the commuting square by small isotopies, so that each curve intersects the curves obtained from it by isotopies or handleslides transversely at two points. After this change we re-label the attaching sets as
\begin{equation}
\label{eq:bgde}
 \xymatrix{ \betas \ar@{-}[r] \ar@{-}[d] & \gammas \ar@{-}[d]\\
\epsilons \ar@{-}[r] & \deltas }
\end{equation}
so that
$$\betas=\{\beta_1, \beta_2, \beta_3, \beta_4 \} \cup \vbeta,  \ \  \gammas \approx  \{\beta_1', \beta_2, \beta_3, \beta_4 \} \cup \vbeta $$
$$
\epsilons \approx \{\beta_1, \beta_2, \beta_3', \beta_4 \} \cup \vbeta, \ \ \deltas \approx \{\beta_1', \beta_2, \beta_3', \beta_4 \} \cup \vbeta,
$$
where $\approx$ denotes isotopy. Figure~\ref{fig:type2} shows the relevant part of the Heegaard surface with the curves in the systems $\betas$, $\gammas$, and $\deltas$. (We do not show $\epsilons$.)
\begin{figure}
{
\fontsize{10pt}{11pt}\selectfont
   \def\svgwidth{5.7in}
   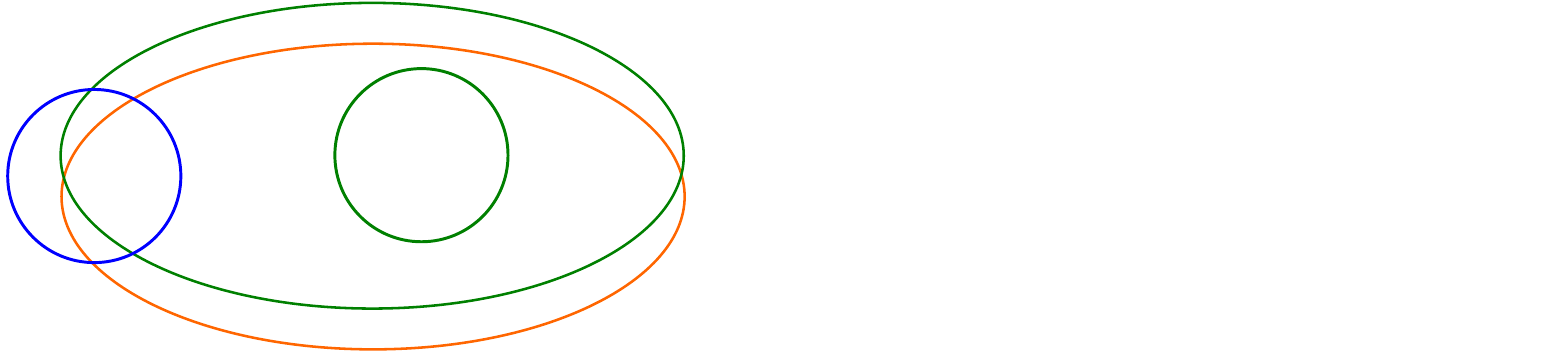
}
\caption{A triple diagram appearing in a handleslide loop of type~\eqref{slide2}. }
\label{fig:type2}
\end{figure}

There are top degree intersection points $\Theta_{\beta, \gamma}$,  $\Theta_{\beta, \epsilon}$, $\Theta_{\gamma, \delta}$ and $\Theta_{\epsilon, \delta}$, which are used to define the respective handleslide maps. The commutativity of the handleslide loop can be written as
$$ \Psi^{\alphas}_{\gammas \to \deltas} \circ \Psi^{\alphas}_{\betas \to \gammas} = \Psi^{\alphas}_{\epsilons \to \deltas} \circ \Psi^{\alphas}_{\betas \to \epsilons}.$$
Using the $A_\infty$ relations for polygon maps, this would follow if we can prove that
\begin{equation}
\label{eq:Fbgde}
F_{\Tb, \Tg, \Td}(\Theta_{\beta, \gamma} \otimes \Theta_{\gamma, \delta}) = F_{\Tb, \Te, \Td}(\Theta_{\beta, \epsilon} \otimes \Theta_{\epsilon, \delta}).
\end{equation}
 
Even though the attaching sets $\betas$ and $\deltas$ do not differ by a single handleslide, observe that we still have a unique top degree intersection point
$$ \Theta_{\beta, \delta} \in \Tb \cap \Td.$$
We claim that Equation~\eqref{eq:Fbgde} holds true because both sides are equal to $\Theta_{\beta, \delta}$. We will check this for the left hand side; the right hand side is similar.

To show that 
\begin{equation}
\label{eq:Fbgd}
F_{\Tb, \Tg, \Td}(\Theta_{\beta, \gamma} \otimes \Theta_{\gamma, \delta}) =\Theta_{\beta, \delta},
\end{equation}
 observe that there is a unique holomorphic triangle of index zero connecting the three intersection points. This is the product of $g$ triangles on the Heegaard surface, four of which are shown darkly shaded in Figure~\ref{fig:type2}. If we orient all the curves in the figure counterclockwise, we obtain the required coupled orientations in each handleslide. It follows from Examples~\ref{ex:abcp} and \ref{ex:product} that the triangle comes with a positive sign; see Figure~\ref{fig:triangles}(c), with $\alpha'$, $\beta'$, $\gamma'$ replaced by $\beta_i$, $\gamma_i$, $\delta_i$. Therefore, the relation ~\eqref{eq:Fbgd} holds.

The handleslide loops of types \eqref{slide3}, \eqref{slide4} and \eqref{slide5} are very similar to type \eqref{slide2}. We isotope and then re-label each of the four attaching sets as in \eqref{eq:bgde}, in the order given in Definition~\ref{def:hloops} for the respective type of loop. In each case we claim that the relation~\eqref{eq:Fbgde} holds, and we prove it by establishing \eqref{eq:Fbgd}. The latter relation follows by exhibiting a unique holomorphic triangle, which appears with sign $+1$ according to the analysis in Example~\ref{ex:abcp}. Figures~\ref{fig:type3}, \ref{fig:type4} and \ref{fig:type5} show the relevant triangles. Each darkly shaded triangle is oriented as in Figure~\ref{fig:triangles}, either (b) or (c).

\begin{figure}
{
\fontsize{10pt}{11pt}\selectfont
   \def\svgwidth{5.3in}
   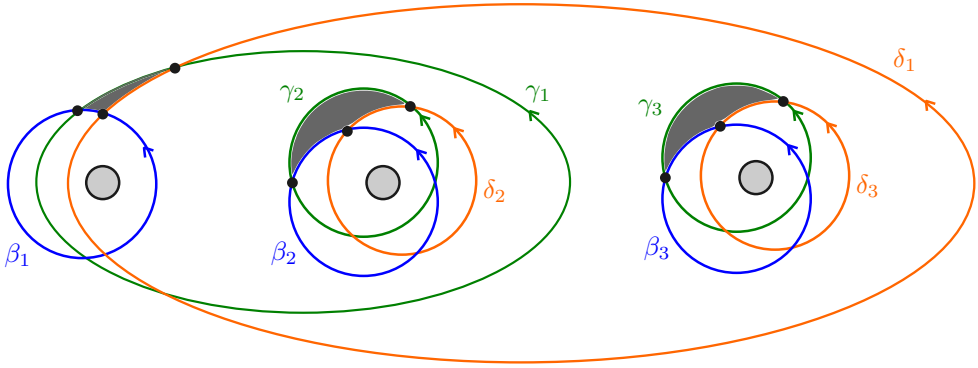
}
\caption{A triple diagram appearing in a handleslide loop of type~\eqref{slide3}. }
\label{fig:type3}
\end{figure}

\begin{figure}
{
\fontsize{10pt}{11pt}\selectfont
   \def\svgwidth{5.5in}
   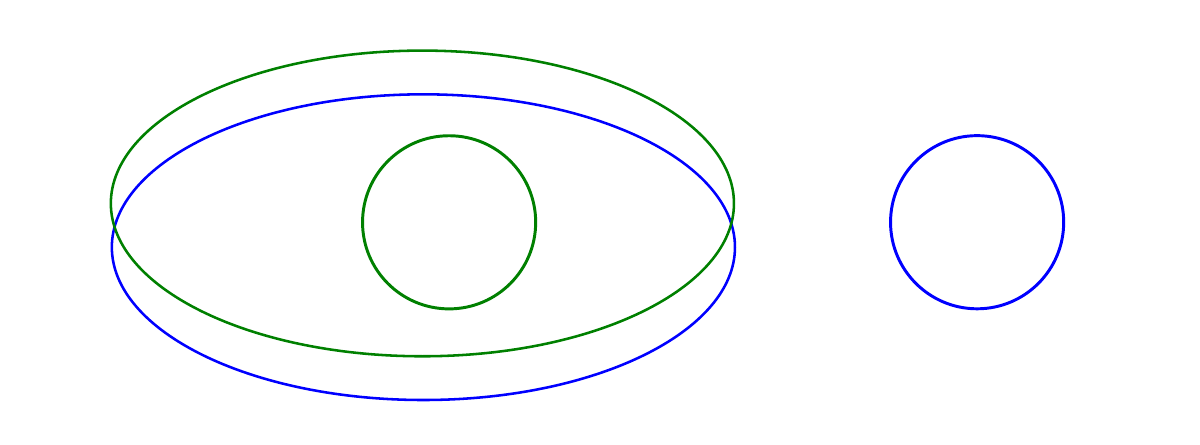
}
\caption{A triple diagram appearing in a handleslide loop of type~\eqref{slide4}.}
\label{fig:type4}
\end{figure}

\begin{figure}
{
\fontsize{10pt}{11pt}\selectfont
   \def\svgwidth{4.6in}
   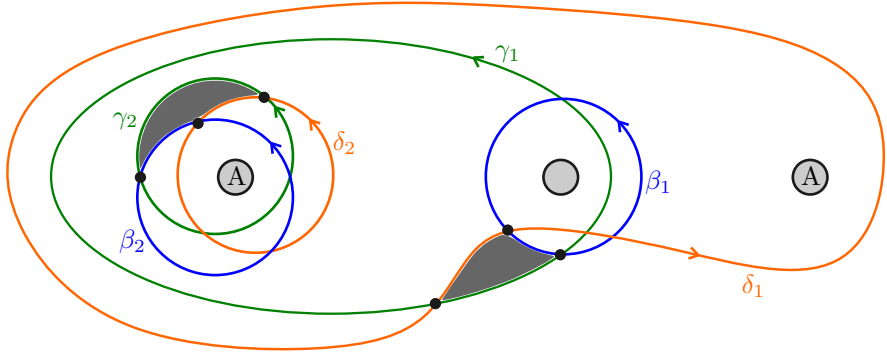
}
\caption{A triple diagram appearing in a handleslide loop of type~\eqref{slide5}. The two disks labeled A co-bound a handle.}
\label{fig:type5}
\end{figure}

The handleslide loop of type ~\eqref{slide6} is also somewhat similar. In this case we have a pentagon. We isotope and re-label the attaching sets from Definition~\ref{def:hloops} as
$$ 
\xymatrixrowsep{3mm}
\xymatrix{ & \betas \ar@{-}[dl] \ar@{-}[dr] &\\
\epsilons \ar@{-}[d] &  &\zetas \ar@{-}[d] \\
\gammas  \ar@{-}[rr]&  & \deltas }$$
We want to show that
\begin{equation}
\label{eq:bgdez}
\Psi^{\alphas}_{\gammas \to \deltas} \circ \Psi^{\alphas}_{ \epsilons \to \gammas} \circ \Psi^{\alphas}_{\betas \to \epsilons} = \Psi^{\alphas}_{ \zetas \to \deltas} \circ \Psi^{\alphas}_{\betas \to \zetas}.
\end{equation}
Even though the attaching set $\beta$ differs from $\gammas$ by two handleslides (rather than one), there is still a unique top degree intersection point $\Theta_{\beta, \gamma} \in \Tb \cap \Tg$, which defines a map $\Psi^{\alphas}_{\betas \to \gammas}$. Similarly, there is an intersection point $\Theta_{\beta, \delta} \in \Tb \cap \Td$ defining a map $\Psi^{\alphas}_{\betas \to \deltas}$. To prove \eqref{eq:bgdez}, we will show that
$$  \Psi^{\alphas}_{ \epsilons \to \gammas} \circ \Psi^{\alphas}_{\betas \to \epsilons} = \Psi^{\alphas}_{\betas \to \gammas}, \ \ \ \Psi^{\alphas}_{ \zetas \to \deltas} \circ \Psi^{\alphas}_{\betas \to \zetas} = \Psi^{\alphas}_{\betas \to \deltas}, \ \ \ \Psi^{\alphas}_{\gammas \to \deltas} \circ \Psi^{\alphas}_{\betas \to \gammas}=\Psi^{\alphas}_{\betas \to \deltas}.$$
which in turn follow from the relations:
\begin{align}
 F_{\Tb, \Te, \Tg}(\Theta_{\beta, \epsilon} \otimes \Theta_{\epsilon, \gamma}) &= \Theta_{\beta, \gamma}, \label{47} \\
 F_{\Tb, \Tz, \Td}(\Theta_{\beta, \zeta} \otimes \Theta_{\zeta, \delta}) &= \Theta_{\beta, \delta}, \label{48} \\
F_{\Tb, \Tg, \Td}(\Theta_{\beta, \gamma} \otimes \Theta_{\gamma, \delta}) &= \Theta_{\beta, \delta}. \label{49}
\end{align}
Relation~\eqref{47} is the same as the one proved in the study of a handleslide of type ~\eqref{slide3} (see Figure~\ref{fig:type3}), because it comes from a curve handlesliding over two other curves. Relations~\eqref{48} and \eqref{49} follow from investigating the unique holomorphic triangles in Figures~\ref{fig:type6z} and \ref{fig:type6}, which both come with positive signs.
\begin{figure}
{
\fontsize{10pt}{11pt}\selectfont
   \def\svgwidth{4.7in}
   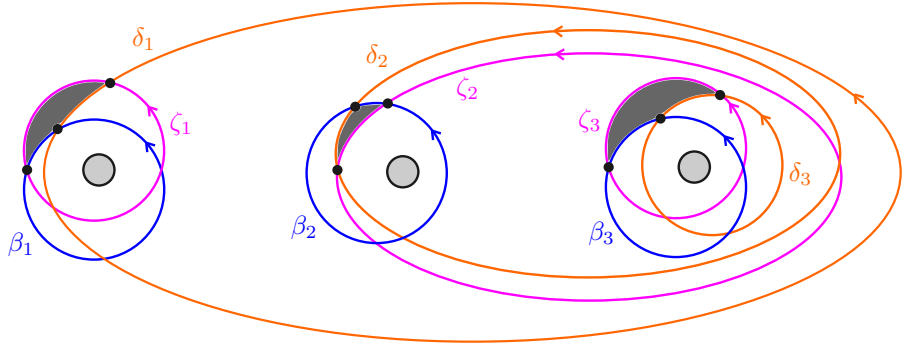
}
\caption{A triple diagram appearing in a handleslide loop of type~\eqref{slide6}.}
\label{fig:type6z}
\end{figure}

\begin{figure}
{
\fontsize{10pt}{11pt}\selectfont
   \def\svgwidth{4.7in}
   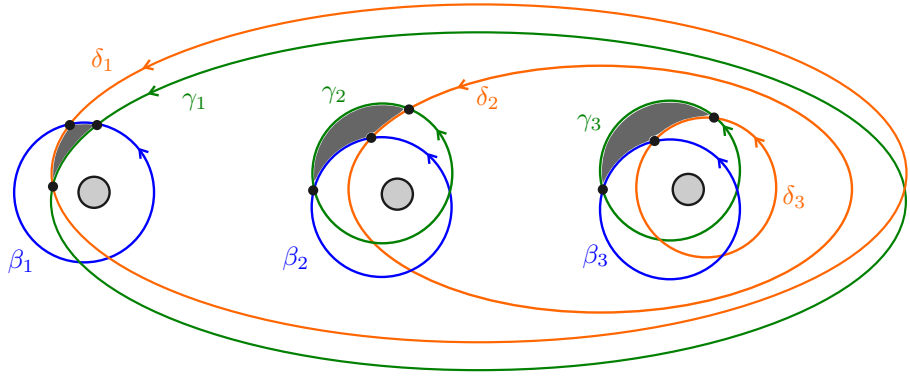
}
\caption{Another triple diagram appearing in a handleslide loop of type~\eqref{slide6}.}
\label{fig:type6}
\end{figure}

Finally, let us discuss handleslide loops of type~\eqref{slide1}, i.e., slide triangles. After an isotopy, we denote the three attaching sets appearing in the triangle
by
$$\xymatrix{ \betas \ar@{-}[rr] \ar@{-}[rd] & & \gammas \ar@{-}[ld]\\
&  \deltas  & }$$
Figure~\ref{fig:type1} shows the resulting triple Heegaard diagram. Note that, because of admissibility issues, we cannot arrange for all three attaching sets to be in the standard position for handleslides from Figure~\ref{fig:slide}. We did arrange this for the $(\betas, \gammas)$ and $(\betas, \deltas)$ pairs, but not for the $(\gammas, \deltas)$ pair. Nevertheless, there is a unique top degree intersection point in $\Tg \cap \Td$, which we denote by $\Theta_{\gamma, \delta}$; this induces a map 
\[ \Psi^{\alphas}_{\gammas \to \deltas}: \HFm(\Ta, \Tg) \to \HFm(\Ta, \Td). \]
\begin{figure}
{
\fontsize{10pt}{11pt}\selectfont
   \def\svgwidth{4.1in}
   %% Creator: Inkscape 1.3.2 (091e20e, 2023-11-25), www.inkscape.org
%% Accompanies image file 'type1.pdf' (pdf, eps, ps)
\begingroup%
  \makeatletter%
  \providecommand\color[2][]{%
    \errmessage{(Inkscape) Color is used for the text in Inkscape, but the package 'color.sty' is not loaded}%
    \renewcommand\color[2][]{}%
  }%
  \providecommand\transparent[1]{%
    \errmessage{(Inkscape) Transparency is used (non-zero) for the text in Inkscape, but the package 'transparent.sty' is not loaded}%
    \renewcommand\transparent[1]{}%
  }%
  \providecommand\rotatebox[2]{#2}%
  \newcommand*\fsize{\dimexpr\f@size pt\relax}%
  \newcommand*\lineheight[1]{\fontsize{\fsize}{#1\fsize}\selectfont}%
  \ifx\svgwidth\undefined%
    \setlength{\unitlength}{423.29154428bp}%
    \ifx\svgscale\undefined%
      \relax%
    \else%
      \setlength{\unitlength}{\unitlength * \real{\svgscale}}%
    \fi%
  \else%
    \setlength{\unitlength}{\svgwidth}%
  \fi%
  \global\let\svgwidth\undefined%
  \global\let\svgscale\undefined%
  \makeatother%
  \begin{picture}(1,0.46486979)%
    \lineheight{1}%
    \setlength\tabcolsep{0pt}%
    \put(0,0){\includegraphics[width=\unitlength,page=1]{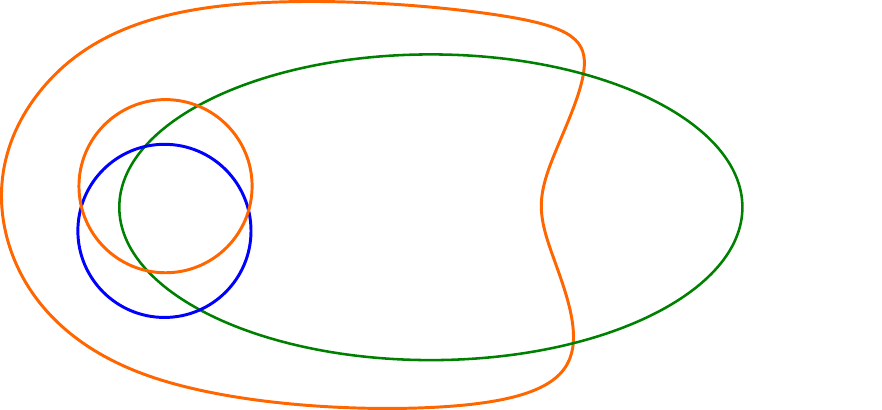}}%
    \put(0.27637656,0.13074034){\color[rgb]{0,0,1}\makebox(0,0)[lt]{\lineheight{1.25}\smash{\begin{tabular}[t]{l}$\beta_1$\end{tabular}}}}%
    \put(0.79431616,0.33506435){\color[rgb]{0,0.50196078,0}\makebox(0,0)[lt]{\lineheight{1.25}\smash{\begin{tabular}[t]{l}$\gamma_1$\end{tabular}}}}%
    \put(0.66443291,0.43073314){\color[rgb]{1,0.4,0}\makebox(0,0)[lt]{\lineheight{1.25}\smash{\begin{tabular}[t]{l}$\delta_2$\end{tabular}}}}%
    \put(0.28910778,0.30732196){\color[rgb]{1,0.4,0}\makebox(0,0)[lt]{\lineheight{1.25}\smash{\begin{tabular}[t]{l}$\delta_1$\end{tabular}}}}%
    \put(0.70673928,0.20711463){\color[rgb]{0,0.50196078,0}\makebox(0,0)[lt]{\lineheight{1.25}\smash{\begin{tabular}[t]{l}$\gamma_2$\end{tabular}}}}%
    \put(0,0){\includegraphics[width=\unitlength,page=2]{type1.pdf}}%
    \put(0.4174003,0.2666335){\color[rgb]{0,0,1}\makebox(0,0)[lt]{\lineheight{1.25}\smash{\begin{tabular}[t]{l}$\beta_2$\end{tabular}}}}%
    \put(0,0){\includegraphics[width=\unitlength,page=3]{type1.pdf}}%
  \end{picture}%
\endgroup%

}
\caption{The triple diagram from a handleslide loop of type~\eqref{slide1}.}
\label{fig:type1}
\end{figure}
The holomorphic triangle darkly shaded in Figure~\ref{fig:type1} comes with positive sign, showing that
$$ F_{\Tb, \Tg, \Td}(\Theta_{\beta, \gamma} \otimes \Theta_{\gamma, \delta}) = \Theta_{\beta, \delta}$$
and therefore
\begin{equation}
\label{eq:ppp}
 \Psi^{\alphas}_{\gammas \to \deltas} \circ \Psi^{\alphas}_{\betas \to \gammas}=\Psi^{\alphas}_{\betas \to \deltas}.
 \end{equation}

However, we are not done! The map $\Psi^{\alphas}_{\gammas \to \deltas}$ is not a handleslide map like those in Section~\ref{sec:handleslide}, where maps were defined from diagrams of the kind shown in Figure~\ref{fig:slide}. We can obtain a diagram of that kind by isotoping the curves $\gamma_2$ and $\delta_1$ as in  Figure~\ref{fig:gdisotopy}, so that they intersect each other and no longer intersect $\delta_2$ and $\gamma_1$.
\begin{figure}
{
\fontsize{10pt}{11pt}\selectfont
   \def\svgwidth{6in}
   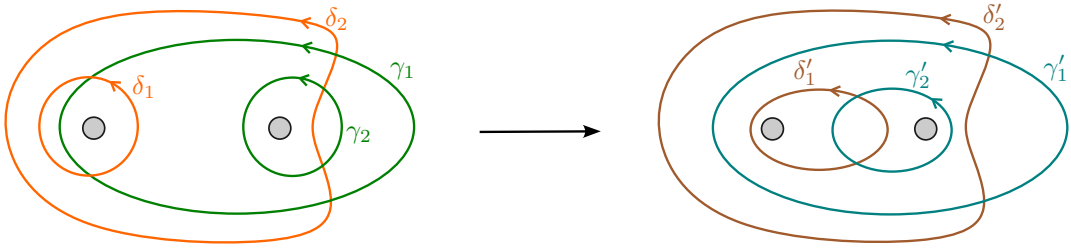
}
\caption{An isotopy to get a standard handlebody diagram.}
\label{fig:gdisotopy}
\end{figure}
We let $\gammas'$ and $\deltas'$ denote the new attaching sets, where:
\begin{itemize}
\item For $i\neq 2$, the curve $\gamma'_i$ is obtained from $\gamma_i$ by a small Hamiltonian isotopy so that $\gamma_i'$ intersects $\gamma_i$ in two points, and does not intersect $\gamma_j$ for any $j \neq i$;
\item For $i\neq 1$, the curve $\delta'_i$ is obtained from $\delta_i$ by a small Hamiltonian isotopy so that $\delta_i'$ intersects $\delta_i$ in two points, and does not intersect $\delta_j$ for any $j \neq i$;
\item The curves $\gamma_2'$ and $\delta_1'$ are as in Figure~\ref{fig:gdisotopy}.
 \end{itemize}
 
Even so, the orderings and orientations of the curves $\gammas'$ in Figure~\ref{fig:gdisotopy} differ from those chosen for a handleslide in Figure~\ref{fig:slide2}. To get to the standard picture, we should switch the ordering of $\gamma_1'$ and $\gamma_2'$, and also the orientation on $\gamma_2'$ (i.e., on the new $\gamma_1'$). The two changes have the combined effect of preserving orientation on the torus $\mathbb{T}_{\gamma'}$, and hence preserving the coupled orientation. They do change the Pin structure, since we are using two different sets of data in Definition~\ref{def:chooseab}.  The two Pin structures come from trivializations of $T\mathbb{T}_{\gamma'}$ that differ by replacing the ordered pair $(\gamma_1', \gamma_2')$ with $(\gamma_2', -\gamma_1')$. Let us relate them by the homotopy given by $90^\circ$ rotation in the $(\gamma_1', \gamma_2')$ plane. This gives a homotopy between the Pin structures, which allows us to identify the respective $\Theta$ generators. Therefore, let us implement these changes for $\gammas'$, and also make the corresponding changes for the set $\gamma$. The result consists in replacing  Figure~\ref{fig:gdisotopy} with Figure~\ref{fig:gdisotopy2}.
\begin{figure}
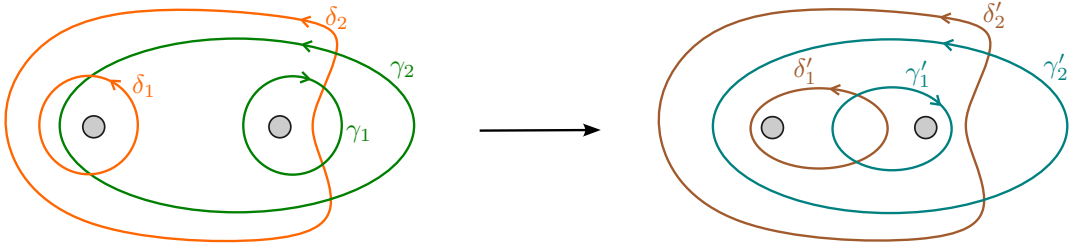

{
\fontsize{10pt}{11pt}\selectfont
   \def\svgwidth{6in}
   
}
\caption{Figure~\ref{fig:gdisotopy} with new orderings and orientations.}
\label{fig:gdisotopy2}
\end{figure}

In the new picture (just as in the old), the attaching set $\gammas'$ represents the same isotopy diagram as $\gammas$, so there is a continuation map $\Gamma^{\alphas}_{\gammas \to \gammas'}.$ Similarly, $\deltas'$ represents the same isotopy diagram as $\deltas$, and we have a continuation map $\Gamma^{\alphas}_{\deltas' \to \deltas}.$ These continuation maps are used in the colimit that defines $\HFm$ for an isotopy diagram. Moreover, the handleslide map between our two isotopy diagrams was defined as $\Psi^{\alphas}_{\gammas' \to \deltas'}.$ Thus, to deduce the commutativity of the slide triangle from \eqref{eq:ppp}, it suffices to prove the following lemma.

\begin{lemma}
\label{lemma:pagd}
We have $\Psi^{\alphas}_{\gammas \to \deltas}= \Gamma^{\alphas}_{\deltas' \to \deltas} \circ \Psi^{\alphas}_{\gammas' \to \deltas'} \circ \Gamma^{\alphas}_{\gammas \to \gammas'}.$
\end{lemma}

\begin{proof}
In the Heegaard diagram $(\Sigma, \gammas, \gammas', z)$, there is a unique top degree intersection point $\Theta_{\gamma, \gamma'}$, which produces a triangle map $\Psi^{\alphas}_{\gammas \to \gammas'}$. This map is the same as the continuation map $\Gamma^{\alphas}_{\gammas \to \gammas'}$, because at the chain level we can interpolate between the two maps by counting monogons as in \cite[Proposition 11.4]{LipshitzCyl}. 

Similarly, there is a top degree intersection point $\Theta_{\delta', \delta}$ that defines a triangle map $\Psi^{\alphas}_{\deltas' \to \deltas}=\Gamma^{\alphas}_{\deltas' \to \deltas}.$ We are left to show that
\begin{equation}
\label{eq:agde}
\Psi^{\alphas}_{\gammas \to \deltas}= \Psi^{\alphas}_{\deltas' \to \deltas} \circ \Psi^{\alphas}_{\gammas' \to \deltas'} \circ \Psi^{\alphas}_{\gammas \to \gammas'}.
\end{equation}

We can further find a unique top degree intersection point $\Theta_{\gamma, \delta'}$, defining a map $\Psi^{\alphas}_{\gammas \to \deltas'}$. We will deduce \eqref{eq:agde} from the relations
$$\Psi^{\alphas}_{\gammas \to \deltas'} =\Psi^{\alphas}_{\gammas' \to \deltas'} \circ \Psi^{\alphas}_{\gammas \to \gammas'}, \ \ \
\Psi^{\alphas}_{\gammas \to \deltas}= \Psi^{\alphas}_{\deltas' \to \deltas} \circ\Psi^{\alphas}_{\gammas \to \deltas'}.$$
In turn, to prove these it suffices to show that
\begin{equation}
\label{eq:ggd}
 F_{\Tg, \Tgp, \Tdp} (\Theta_{\gamma, \gamma'} \otimes \Theta_{\gamma', \delta'} ) = \Theta_{\gamma, \delta'}
 \end{equation}
 and
\begin{equation}
\label{eq:gdd}
F_{\Tg, \Tdp, \Td} (\Theta_{\gamma, \delta'} \otimes \Theta_{\delta', \delta} ) = \Theta_{\gamma, \delta}.
\end{equation}
Equation~\eqref{eq:ggd} follows from investigating the unique holomorphic triangle between the top degree generators in Figure~\ref{fig:ggd}; this comes with a positive sign. 
\begin{figure}
{
\fontsize{10pt}{11pt}\selectfont
   \def\svgwidth{4.1in}
   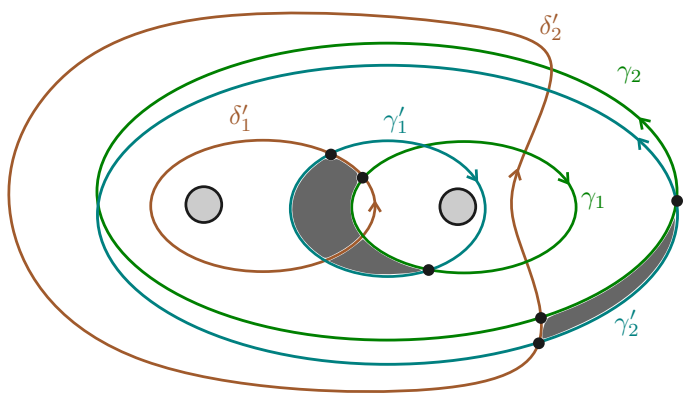
}
\caption{A triple diagram leading to Equation~\eqref{eq:ggd}.}
\label{fig:ggd}
\end{figure}

Equation~\eqref{eq:gdd} follows from a more involved analysis, based on Figure~\ref{fig:gdd}. There, the homotopy class of triangles between theta elements that we need to consider is not made of two bigons, but rather is the darkly shaded hexagon. This is still combined with $g-2$ other standard triangles in the rest of the triple diagram, not shown in the picture. The standard triangles come with a positive sign, so to deduce ~\eqref{eq:gdd}, it remains to prove that the count of holomorphic triangles in the hexagon class (in the second symmetric product) is $+1$. 
\begin{figure}
{
\fontsize{10pt}{11pt}\selectfont
   \def\svgwidth{4.1in}
   %% Creator: Inkscape 1.3.2 (091e20e, 2023-11-25), www.inkscape.org
%% Accompanies image file 'gdd.pdf' (pdf, eps, ps)
\begingroup%
  \makeatletter%
  \providecommand\color[2][]{%
    \errmessage{(Inkscape) Color is used for the text in Inkscape, but the package 'color.sty' is not loaded}%
    \renewcommand\color[2][]{}%
  }%
  \providecommand\transparent[1]{%
    \errmessage{(Inkscape) Transparency is used (non-zero) for the text in Inkscape, but the package 'transparent.sty' is not loaded}%
    \renewcommand\transparent[1]{}%
  }%
  \providecommand\rotatebox[2]{#2}%
  \newcommand*\fsize{\dimexpr\f@size pt\relax}%
  \newcommand*\lineheight[1]{\fontsize{\fsize}{#1\fsize}\selectfont}%
  \ifx\svgwidth\undefined%
    \setlength{\unitlength}{423.45114881bp}%
    \ifx\svgscale\undefined%
      \relax%
    \else%
      \setlength{\unitlength}{\unitlength * \real{\svgscale}}%
    \fi%
  \else%
    \setlength{\unitlength}{\svgwidth}%
  \fi%
  \global\let\svgwidth\undefined%
  \global\let\svgscale\undefined%
  \makeatother%
  \begin{picture}(1,0.49199094)%
    \lineheight{1}%
    \setlength\tabcolsep{0pt}%
    \put(0,0){\includegraphics[width=\unitlength,page=1]{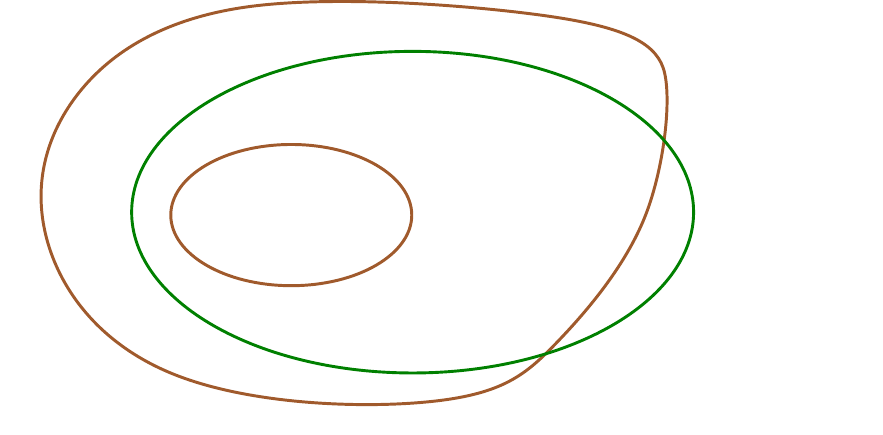}}%
    \put(0.79439358,0.26158412){\color[rgb]{0,0.50196078,0}\makebox(0,0)[lt]{\lineheight{1.25}\smash{\begin{tabular}[t]{l}$\gamma_2$\end{tabular}}}}%
    \put(0.74513004,0.4362155){\color[rgb]{0.62745098,0.35294118,0.17254902}\makebox(0,0)[lt]{\lineheight{1.25}\smash{\begin{tabular}[t]{l}$\delta'_2$\end{tabular}}}}%
    \put(0.35332541,0.13391455){\color[rgb]{0.62745098,0.35294118,0.17254902}\makebox(0,0)[lt]{\lineheight{1.25}\smash{\begin{tabular}[t]{l}$\delta'_1$\end{tabular}}}}%
    \put(0.54996414,0.1308856){\color[rgb]{0,0.50196078,0}\makebox(0,0)[lt]{\lineheight{1.25}\smash{\begin{tabular}[t]{l}$\gamma_1$\end{tabular}}}}%
    \put(0,0){\includegraphics[width=\unitlength,page=2]{gdd.pdf}}%
    \put(0.10829975,0.33718558){\color[rgb]{1,0.4,0}\makebox(0,0)[lt]{\lineheight{1.25}\smash{\begin{tabular}[t]{l}$\delta_1$\end{tabular}}}}%
    \put(0.04344588,0.05112985){\color[rgb]{1,0.4,0}\makebox(0,0)[lt]{\lineheight{1.25}\smash{\begin{tabular}[t]{l}$\delta_2$\end{tabular}}}}%
    \put(0,0){\includegraphics[width=\unitlength,page=3]{gdd.pdf}}%
  \end{picture}%
\endgroup%

}
\caption{A triple diagram with a darkly shaded hexagon leading to Equation~\eqref{eq:gdd}.}
\label{fig:gdd}
\end{figure}

Let us re-draw the hexagon by itself in Figure~\ref{fig:hexagon} (a). To count its holomorphic representatives, we use Lipshitz's cylindrical picture from \cite{LipshitzCyl} and stretch the almost complex structure along the dashed line in Figure~\ref{fig:hexagon}. This kind of deformation was analyzed in \cite[Appendix A]{LipshitzCyl} and used at various places in the literature; see for example \cite[Theorem 5.1]{Links}, \cite[Section 14]{ZemkeHat}, or \cite[Proposition 4.3]{ZemkeQuasi}. In our case, the result of the neck stretching is shown in Figure~\ref{fig:gdd}. (See Section~\ref{sec:cyl} below for more about signs in the cylindrical picture.)

\begin{figure}
{
\fontsize{10pt}{11pt}\selectfont
   \def\svgwidth{4.1in}
   %% Creator: Inkscape 1.3.2 (091e20e, 2023-11-25), www.inkscape.org
%% Accompanies image file 'hexagon.pdf' (pdf, eps, ps)
\begingroup%
  \makeatletter%
  \providecommand\color[2][]{%
    \errmessage{(Inkscape) Color is used for the text in Inkscape, but the package 'color.sty' is not loaded}%
    \renewcommand\color[2][]{}%
  }%
  \providecommand\transparent[1]{%
    \errmessage{(Inkscape) Transparency is used (non-zero) for the text in Inkscape, but the package 'transparent.sty' is not loaded}%
    \renewcommand\transparent[1]{}%
  }%
  \providecommand\rotatebox[2]{#2}%
  \newcommand*\fsize{\dimexpr\f@size pt\relax}%
  \newcommand*\lineheight[1]{\fontsize{\fsize}{#1\fsize}\selectfont}%
  \ifx\svgwidth\undefined%
    \setlength{\unitlength}{475.29169518bp}%
    \ifx\svgscale\undefined%
      \relax%
    \else%
      \setlength{\unitlength}{\unitlength * \real{\svgscale}}%
    \fi%
  \else%
    \setlength{\unitlength}{\svgwidth}%
  \fi%
  \global\let\svgwidth\undefined%
  \global\let\svgscale\undefined%
  \makeatother%
  \begin{picture}(1,0.4752163)%
    \lineheight{1}%
    \setlength\tabcolsep{0pt}%
    \put(0.18705662,0.00550238){\makebox(0,0)[lt]{\lineheight{1.25}\smash{\begin{tabular}[t]{l}$(a)$\end{tabular}}}}%
    \put(0.82766255,0.0055017){\makebox(0,0)[lt]{\lineheight{1.25}\smash{\begin{tabular}[t]{l}$(b)$\end{tabular}}}}%
    \put(0,0){\includegraphics[width=\unitlength,page=1]{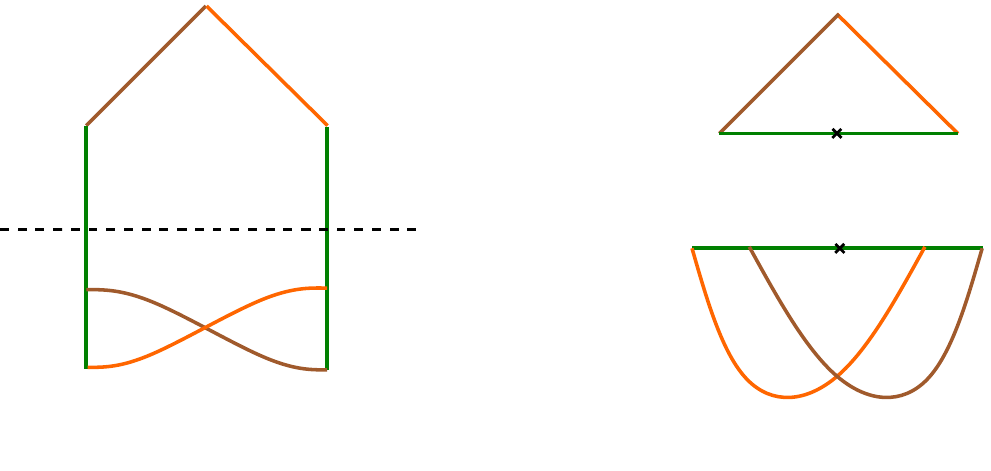}}%
    \put(0.843503,0.30854984){\makebox(0,0)[lt]{\lineheight{1.25}\smash{\begin{tabular}[t]{l}$p$\end{tabular}}}}%
    \put(0.85497519,0.24376501){\makebox(0,0)[lt]{\lineheight{1.25}\smash{\begin{tabular}[t]{l}$q$\end{tabular}}}}%
    \put(0,0){\includegraphics[width=\unitlength,page=2]{hexagon.pdf}}%
    \put(0.19660036,0.30629273){\makebox(0,0)[lt]{\lineheight{1.25}\smash{\begin{tabular}[t]{l}$\psi$\end{tabular}}}}%
    \put(0.82934803,0.38550548){\makebox(0,0)[lt]{\lineheight{1.25}\smash{\begin{tabular}[t]{l}$\psi_1$\end{tabular}}}}%
    \put(0.82934794,0.15550218){\makebox(0,0)[lt]{\lineheight{1.25}\smash{\begin{tabular}[t]{l}$\psi_2$\end{tabular}}}}%
  \end{picture}%
\endgroup%

}
\caption{(a) The hexagon from Figure~\ref{fig:gdd}. (b) Its degeneration.}
\label{fig:hexagon}
\end{figure}

Let $T$ be the standard $2$-simplex (triangle), where we remove neighborhoods of the vertices and attach three infinite half-strips instead.  In the cylindrical picture, we are counting holomorphic maps $u: S \to \Sigma \times T$, where $S$ is a branched cover of $\Delta$ (in our case, a hexagon). Let 
$$\pi_T: \Sigma \times T \to T, \ \ \ \pi_{\Sigma}: \Sigma \times T \to \Sigma$$
be the projections. We impose some boundary conditions on $u$: the edges of $S$ should map under $\pi_T \circ u$ to the edges of the triangle, and under $\pi_\Sigma \circ u$ to the edges of the hexagon shown in Figure~\ref{fig:hexagon}; further, vertices should map to the corresponding vertices. Let $\M(\psi)$ be the moduli space of holomorphic maps of this form, in the hexagon class $\psi$ from Figure~\ref{fig:hexagon}. 

Let us extend the dashed line in Figure~\ref{fig:hexagon} to a separating circle on $\Sigma$. After the degeneration, the surface $\Sigma$ turns into the wedge sum of two surfaces $\Sigma_1$ and $\Sigma_2$, glued to each other at points $p \in \Sigma_1$ and $q \in \Sigma_2$.  The class $\psi$ gets split into $\psi_1$ and $\psi_2$. For $i=1, 2$, we have moduli spaces $\M(\psi_i)$, consisting of holomorphic maps $u_i: T \to \Sigma_i \times T$ in the class $\psi_i$ (with the obvious boundary conditions). Consider the maps
$$ \rho_1: \M(\psi_1) \to \R, \ \ \rho_1=\pi_T ((\pi_{\Sigma} \circ u)^{-1}(p)),$$
$$ \rho_2: \M(\psi_1) \to \R, \ \ \rho_2=\pi_T ((\pi_{\Sigma} \circ u)^{-1}(q)).$$
(In principle, these maps should have image in $T$, but they land on the green edge of $\del T$, which we identify with $\R$ by orienting it using the boundary orientation from $T$.)

For large enough neck length, the space $\M(\psi)$ is identified with the fiber product
\begin{equation}
\label{eq:fiber}
\M(\psi_1) \times_\R \M(\psi_2) = \{(u_1, u_2) \in \M(\psi_1) \times \M(\psi_2) \mid \rho_1(u_1) + \rho_2(u_2) = 0\},
\end{equation}
and this identification preserves orientations. Since a triangle with three marked points on the boundary has a unique automorphism, we see that $\M(\psi_1)$ is a point, so its image under $\rho_1$ is also a fixed point in $ \R$. On the other hand, $\M(\psi_2)$ is one-dimensional, because when we apply the Riemann mapping theorem to construct a holomorphic representative, we can vary the length of a slit along either the orange or the brown curve. 

Thus, the fiber product in ~\eqref{eq:fiber} is simply the preimage $\rho_2^{-1}(t)$ for some $t \in \R$. To count its points (with sign), we can vary $t$ at will. In the limit $t \to +\infty$, the class $\psi_2$ splits into a triangle and a bigon, as in Figure~\ref{fig:psitwo}. The respective moduli spaces both consist of a point with positive sign; see Examples~\ref{ex:abcp} and \ref{ex:circles}. Therefore, the count of points in $\M(\psi) \cong \rho_2^{-1}(t)$ is $+1$.
\end{proof}
\begin{figure}
{
\fontsize{10pt}{11pt}\selectfont
   \def\svgwidth{4.1in}
   %% Creator: Inkscape 1.3.2 (091e20e, 2023-11-25), www.inkscape.org
%% Accompanies image file 'psitwo.pdf' (pdf, eps, ps)
\begingroup%
  \makeatletter%
  \providecommand\color[2][]{%
    \errmessage{(Inkscape) Color is used for the text in Inkscape, but the package 'color.sty' is not loaded}%
    \renewcommand\color[2][]{}%
  }%
  \providecommand\transparent[1]{%
    \errmessage{(Inkscape) Transparency is used (non-zero) for the text in Inkscape, but the package 'transparent.sty' is not loaded}%
    \renewcommand\transparent[1]{}%
  }%
  \providecommand\rotatebox[2]{#2}%
  \newcommand*\fsize{\dimexpr\f@size pt\relax}%
  \newcommand*\lineheight[1]{\fontsize{\fsize}{#1\fsize}\selectfont}%
  \ifx\svgwidth\undefined%
    \setlength{\unitlength}{449.36750757bp}%
    \ifx\svgscale\undefined%
      \relax%
    \else%
      \setlength{\unitlength}{\unitlength * \real{\svgscale}}%
    \fi%
  \else%
    \setlength{\unitlength}{\svgwidth}%
  \fi%
  \global\let\svgwidth\undefined%
  \global\let\svgscale\undefined%
  \makeatother%
  \begin{picture}(1,0.22449574)%
    \lineheight{1}%
    \setlength\tabcolsep{0pt}%
    \put(0,0){\includegraphics[width=\unitlength,page=1]{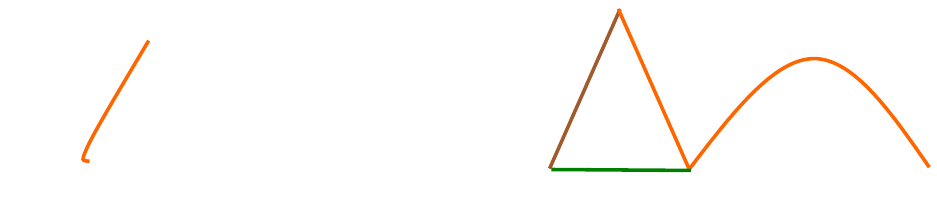}}%
    \put(0.90261768,0.0058198){\makebox(0,0)[lt]{\lineheight{1.25}\smash{\begin{tabular}[t]{l}$q$\end{tabular}}}}%
    \put(0.14963261,0.00816377){\makebox(0,0)[lt]{\lineheight{1.25}\smash{\begin{tabular}[t]{l}$q$\end{tabular}}}}%
    \put(0,0){\includegraphics[width=\unitlength,page=2]{psitwo.pdf}}%
  \end{picture}%
\endgroup%

}
\caption{The class $\psi_2$ from Figure~\ref{fig:hexagon}(b), in the limit when it splits along the $\delta_2$ curve. (Note that, after breaking the orange $\delta_2$ segment in two, the orientation on one half of the segment is reversed. This is as it should be: We want the same coupled orientation between the green and orange curves at the two sides of the breaking point.) }
\label{fig:psitwo}
\end{figure}

\begin{remark}
If we reverse the orientations of all the curves in Figures~\ref{fig:gdd}, \ref{fig:hexagon} and \ref{fig:psitwo}, the coupled Spin structures stay the same (cf. Convention~\ref{conv:data}), and the two triangles we consider still come with positive sign. However, the bigon on the right of Figure~\ref{fig:psitwo} looks like the bigon labeled $B$ (rather than $A$) in Figure~\ref{fig:circles}, so it comes with a negative sign. This sign is cancelled by the fact that the orientation of the green curve is reversed, which changes the orientation of the line $\R$ over which we take the fiber product in \eqref{eq:fiber}. Thus, the overall sign of the hexagon is still $+1$.

An alternate way to think about this in Lipshitz's cylindrical picture is that we orient the moduli space of hexagons $S$ that are double covers of the triangle $T$ branched at a single point. This moduli space is just the triangle $T$ itself, because the hexagon is determined by the position of the branch point $x \in T$. The degeneration in Figure~\ref{fig:hexagon} corresponds to sending $x$ to an edge, and the one in Figure~\ref{fig:psitwo} corresponds to further sending it to a vertex. In our discussion, we have conveniently chosen orientations of the moduli spaces compatible with these degenerations: the triangle $T$ is oriented as a subset of the plane, the green curve is oriented as part of its boundary, and $t \to +\infty$ is the positive end of $\R$.  
\end{remark}

Now that Lemma~\ref{lemma:pagd} is proved, we have finished checking invariance under handleslide loops. We continue with the last step in the proof of Theorem~\ref{thm:main}.

\medskip
{\em(6)} {\bf Commutativity along stabilization slides:} Let us denote the Heegaard diagrams that appear in Figure~\ref{fig:sslide} by
$$ 
\xymatrixrowsep{3mm}
\xymatrix{ & (\Sigma, \alphas, \betas)  \ar[dr]^{\Psi^{\alphas}_{\betas \to \gammas}} &\\
(\Sigma^u, \alphas^u, \betas^u) \ar[ur]^{S}  \ar[dr]_{S'} &  &(\Sigma, \alphas, \gammas)  \ar[dl]^{\Psi^{\alphas}_{\gammas \to \deltas}} \\
& (\Sigma, \alphas, \deltas)  & }
$$
with $\Sigma^u$ be the original (un-stabilized) surface, and $\Sigma$ its stabilization. The two handleslide maps on the right hand side are induced by top degree generators $\Theta_{\beta, \gamma}$ and $\Theta_{\gamma, \delta}$. 

Even though they differ by two handleslides rather than one, the attaching sets $\betas$ and $\deltas$ still produce a unique top degree generator $\Theta_{\beta, \delta}$, and a triangle map $\Psi^{\alphas}_{\betas \to \deltas}$. The holomorphic triangle darkly shaded in Figure~\ref{fig:sslide2} comes with positive sign, which implies that
$$ F_{\Tb, \Tg, \Td}(\Theta_{\beta, \gamma} \otimes \Theta_{\gamma, \delta}) = \Theta_{\beta, \delta}$$
and therefore $ \Psi^{\alphas}_{\gammas \to \deltas} \circ \Psi^{\alphas}_{\betas \to \gammas}=\Psi^{\alphas}_{\betas \to \deltas}.$
\begin{figure}
{
\fontsize{10pt}{11pt}\selectfont
   \def\svgwidth{2.8in}
   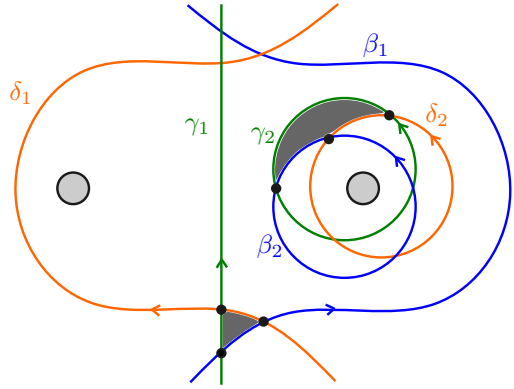
}
\caption{A triple diagram appearing in the stabilization slide.}
\label{fig:sslide2}
\end{figure}

It remains to show that 
$$\Psi^{\alphas}_{\betas \to \deltas} \circ S = S'.$$
This is a consequence of Theorem 10.4 in \cite{HolDiskTwo}, which says that triangle maps commute with stabilizations. In the un-stabilized diagram, the remaining $\beta$ and $\delta$ curves are just small Hamiltonian perturbations of the $\beta^u$ curves, so the respective triangle map is just the nearest point map (which we can think of as the identity, once we identify the corresponding curves by the small isotopy).

\medskip
By applying Theorems~\ref{thm:shi2} and \ref{thm:strongH}, we deduce that $\HFm$ is a natural invariant of based $3$-manifolds. The proof for the other versions is similar.
\end{proof}

\subsection{Signs in the cylindrical picture}
\label{sec:cyl}
In \cite{LipshitzCyl}, Lipshitz reinterprets Heegaard Floer homology as follows. Instead of strips $u: \R \times [0,1] \to \Sym^g(\Sigma)$ with boundaries on $\Ta$ and $\Tb$, he considers holomorphic curves 
$$ u: S \to  \Sigma \times \R \times [0,1]$$
where $S$ is a smooth surface with boundary and infinite ends (of any genus), and the boundary components of $\del S$ are mapped to Lagrangians of the form $\alpha_i \times \{1\} \times \R$ and $\beta_i \times \{1\} \times \R$. (For the exact conditions on these curves, we refer to \cite[Section 1]{LipshitzCyl}.) 

There is a tautological correspondence between the curves counted in Lipshitz's setup and the strips in the symmetric product that are counted in Ozsv\'ath and Szab\'o's original construction. This correspondence is compatible with the determinant index bundles of $D\delbar$ operators, which allows Lipshitz to show that his theory is isomorphic to the original one over $\Z$. (See Section 13 in \cite{LipshitzCyl}, particularly Proposition 13.7.) The cylindrical picture and the tautological correspondence can also be extended to polygon maps. 

In \cite{LipshitzCyl}, the signs for the differential come from coherent orientations, just as in the original  \cite{HolDisk}. In our context, the coupled Spin structure induces signs in the Ozsv\'ath-Szab\'o picture, and hence (via the tautological correspondence) in Lipshitz's cylindrical picture. 

In Heegaard Floer theory, it is common to move to the cylindrical picture when one studies various degenerations. See for example \cite[Section 5]{Links},  \cite[Section 9.3]{JTZ}, \cite[Section 6]{ZemkeHF}, or what we did in the proofs of handleswap invariance and in Lemma~\ref{lemma:pagd} above. In such cases, we need to know that the orientations of the moduli spaces of curves are compatible with these degenerations. Transplanting orientations from the symmetric product picture is not fully satisfactory, because the degeneration arguments are less clear there.

Instead, it suffices to note that one can also construct canonical orientations in Lipshitz's picture, and that these agree with those from the symmetric product. In all the situations we encounter (such as in Lemma~\ref{lemma:pagd}), we fix Spin structures as well as orderings on the alpha curves, as well as on the beta curves. By taking the product of these with the canonical Lie group Spin structure on $\R$, we get Spin structures on Lagrangians such as $\alpha_i \times \{1\} \times \R$ and $\beta_i \times \{1\} \times \R$, which give the boundary conditions in Lipshitz's picture. Thus, we can apply Seidel's work \cite{SeidelBook} and get orientations on the moduli spaces of curves there. In particular, we can construct orientation spaces $o(x)$ for every $x \in \alpha_i \cap \beta_j$, and set
$$ o(\x) = o(x_1) \otimes \dots \otimes o(x_g)$$
for each  $\x = \{x_1, \dots, x_g\} \in \Ta \cap \Tb$. Since the Spin structures on $\Ta$ and $\Tb$ are products of those on the alpha and beta curves, respectively, we find that the orientation spaces $o(\x)$ which generate the Floer complexes in the two pictures can be identified. Furthermore, following the identification of the index bundles in \cite[Section 13]{LipshitzCyl}, we see that the signs of differentials in the two pictures agree. A similar argument applies to the signs of higher polygons. 

Whenever we consider a degeneration of polygons in Lipshitz's picture, the index bundles are compatible with  gluing, and therefore so are the orientations induced by Spin structures. This justifies extending the degeneration arguments to $\Z$, provided the Spin structures before and after the degeneration are compatible. In fact, we always choose Lie group Spin structures on the curves, which are canonical once an orientation is fixed. Therefore, in each situation it suffices to describe orientations on the curves that are compatible with the degeneration. (An example is in Figure~\ref{fig:hexagon}.)

%%% Local Variables:
%%% mode: LaTeX
%%% TeX-master: "signs"
%%% End:

\section{The surgery exact triangle}
\label{sec:exact}

The surgery exact triangle is one of the most useful tools in Heegaard Floer theory. It was originally proved in \cite[Theorem 9.12]{HolDiskTwo}, using $\Z$ coefficients from coherent orientations. A different proof, using the triangle detection lemma from homological algebra, was given in \cite{BrCov}, with $\F_2$ coefficients. In this section we prove it over $\Z$ using our set-up with canonical orientations.

A similar exact triangle exists in Lagrangian Floer homology, due to Seidel \cite{SeidelExact}; see also \cite[Theorem 17.16]{SeidelBook}. This applies to the case where one Lagrangian is obtained from the other by a Dehn twist or, more generally, a fibered Dehn twist (\cite{MakWu}, \cite{WWExact}, \cite{WWOri}). Neither of these is quite the situation in Heegaard Floer homology but, nevertheless, one could imagine an extension of Seidel's exact triangle that applies to Heegaard Floer theory. With regard to orientations, an interesting point to note is that for Seidel's exact triangle, one needs to choose a bounding rather than a Lie group Pin structure on the circle; see \cite[Example 17.15]{SeidelBook}. (In his terminology, nontrivial means bounding.) In our situation, it is also the case that the simplest exact triangle that can be constructed would involve bounding Pin structures. Nevertheless, by using twisted coefficients, we can arrive at a triangle with orientations coming from Lie group Pin structures.

We state the triangle here in the same generality as in monopole theory; compare \cite[Theorem 2.4]{KMOS}.

%Seidel's exact triangle \cite{SeidelExact}: we get a canonical element in $HF(L, \tau_S(L))$, where $\tau_S(L)$ has the brane structure induced from $L$. \cite[Theorem 17.16]{SeidelBook}.

%See also \cite{MakWu}, \cite{WWExact}, \cite{WWOri}. Our situation is not quite a fibered Dehn twist, because if we have a curve $\gamma$ on $\Sigma$ there is no coisotropic $C \to \Sym^{g-1}(\Sigma)$ with fiber $\gamma$. Ideally we would want $C$ to be made of multisets hitting $\gamma$, but when there are two points in $\gamma$ the map to $\Sym^{g-1}(\Sigma)$ is not well-defined.

%In symplectic geometry, we need the bounding framing on 3 Lagrangians to get the exact triangle. The modern proof of this is along the lines of Wehrheim-Woodward and Mak-Wu, with a Lagrangian correspondence in $Y \times \C$ between $\alpha \times 0$ and $\alpha \times 1$, paired with a Lagrangians having ends at the three curves $\beta, \gamma, \delta$. There, it is important to have spin structures on the three circles that are restrictions of a single spin structure on the pair of pants (so that the correspondence is spin). We also get other exact triangles when we twist by a local coefficient system on the whole pair of pants, i.e. when we twist 2 of the 3 Lagrangians. 

\begin{theorem}
\label{thm:exact}
Let $Y$ be a compact oriented $3$-manifold with torus boundary, and let $\beta, \gamma, \delta$ be three oriented simple closed curves on $\del Y$, pairwise intersecting in a single point, with intersection numbers
$$ \beta \cdot \gamma = \gamma \cdot \delta = \delta \cdot \beta = -1.$$
Let $Y_\beta$, $Y_\gamma$, and $Y_\delta$ be the closed three-manifolds obtained from $Y$ by filling it with a solid torus $S^1 \times D^2$ such that $\{1 \} \times \del D^2$ is attached to the respective curve ($\beta$, $\gamma$ or $\delta$). Then, there is an exact sequence
$$\cdots \to  \HFp(Y_{\beta}) \to \HFp(Y_{\gamma}) \to \HFp(Y_{\delta}) \to \cdots$$
\end{theorem}

Before proving the theorem, we recall the setup for the corresponding result in \cite[Section 9]{HolDiskTwo}, and explain why the naive strategy fails. We can choose an admissible pointed Heegaard multi-diagram $(\Sigma, \alphas, \betas, \gammas, \deltas, z)$ such that:
\begin{itemize}
\item The Heegaard diagrams $(\Sigma, \alphas, \betas)$, $(\Sigma, \alphas, \gammas)$ and $(\Sigma, \alphas, \deltas)$ represent $Y_\beta$, $Y_\gamma$, and $Y_\delta$ respectively;
\item For $i=1, \dots, g-1$, the curves $\beta_i$, $\gamma_i$ and $\delta_i$ are small isotopic translates of each other, each pairwise intersecting in a pair of transverse intersection points (with the isotopies being supported away from $z$);
\item We have $\beta_g = \beta$, $\gamma_g =\gamma$, and $\delta_g = \delta$.
\end{itemize}

Thus, we can re-write the desired exact triangle as
\begin{equation}
\label{eq:exactp}
 \cdots \to  \HFp(\Ta, \Tb) \xrightarrow{f_1} \HFp(\Ta, \Tg) \xrightarrow{f_2}  \HFp(\Ta, \Td) \xrightarrow{f_3}  \cdots
 \end{equation}

Note that, for the purposes of proving it, we do not need to worry about naturality issues. We choose Lie group Pin structures $\Pa$, $\Pb$ , $\Pg$, $\Pd$ on $\Ta$, $\Tb$, $\Tg$, $\Td$, and also choose arbitrary trivializations of the lines $\ell (\Pa, \Pb)$, $\ell (\Pa, \Pg)$, $\ell (\Pa, \Pd)$. This allows us to identify the Heegaard Floer groups with the preliminary ones from Section~\ref{sec:definition}, for which we have well-defined polygon maps. (Compare Remark~\ref{rem:notriangle}.) 

As a first guess, we could try to define the maps in \eqref{eq:exactp} by counting triangles:
\[ f_1(\cdot) = F_{\Ta, \Tb, \Tg} (\cdot \otimes \Theta_{\beta, \gamma}), \ \ f_2(\cdot) = F_{\Ta, \Tg, \Td} (\cdot \otimes \Theta_{ \gamma, \delta}), \ \ f_3(\cdot) = F_{\Ta, \Td, \Tb} (\cdot \otimes \Theta_{ \delta, \beta}).\]
Here, $\Theta_{\beta, \gamma}$ can be any of the two top-degree generators of the group $\HFp(\Tb, \Tg) \cong H_*(T^{g-1})\otimes \Z[U^{-1}]$; and similarly for $ \Theta_{ \gamma, \delta}$ and $ \Theta_{ \delta, \beta}$. It does not matter which generator we choose, because they differ by a sign, and changing one of the maps in an exact triangle by $-1$ keeps it exact.

However, defining the maps $f_i$ in this way does not make the double compositions be zero. Indeed, for example, the proof that $f_2 \circ f_1 = 0$ in \cite[Section 9]{HolDiskTwo} reduces to showing that
\begin{equation}
\label{eq:bgdtt}
 F_{\Tb, \Tg, \Td} (\Theta_{\beta, \gamma} \otimes \Theta_{ \gamma, \delta}) = 0.
 \end{equation}

Proposition 9.5 in \cite{HolDiskTwo} shows that the relevant (index zero) holomorphic triangles between $\Tb, \Tg$ and $\Td$ come in pairs, with the triangles in each pair having basepoint multiplicity $n_z = k(k-1)/2$ for $k \geq 1$. This can be established by a degeneration argument to reduce it to the genus $1$ case, where triangles can be counted explicitly on the diagram. The two triangles with $k=1$ (so $n_z=0$) are shown in Figure~\ref{fig:ttt}. The triangles for higher $k$ wrap around the torus more times.

\begin{figure}
{
\fontsize{11pt}{11pt}\selectfont
   \def\svgwidth{2.5in}
   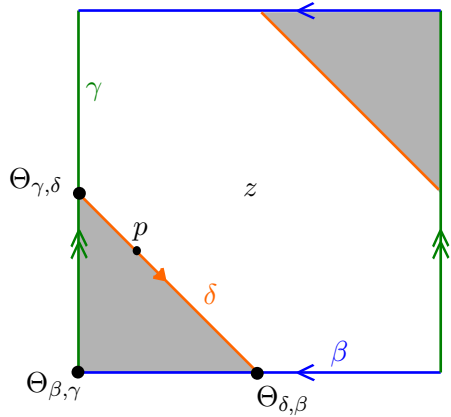
}
\caption{A pair of triangles on the torus. We use the point $p$ to define polygon maps with twisted coefficients.}
\label{fig:ttt}
\end{figure}

Observe that rotation by $180^\circ$ around $z$ takes the two triangles in a pair to each other. Furthermore,  the Lie group Pin structure on each circle is uniquely characterized as being preserved by translations, and this property is preserved by the $180^\circ$ rotation. It follows that the two triangles in a pair are related by a symplectomorphism of the torus that preserves all the underlying structures, and therefore are counted with the same sign. This means that they do not cancel out, so we cannot deduce Equation~\eqref{eq:bgdtt}. (They would cancel out if we had the bounding Pin structure on either one or three of the curves $\beta$, $\gamma$, $\delta$, and the Lie group Pin structure on the rest.)

We shall remedy this problem while keeping Lie group Pin structures on all Lagrangians, by introducing twisted coefficients:
\begin{proof}[Proof of Theorem \ref{thm:exact}] 
 Let $\iota: \del Y \to Y$ be the inclusion. By the ``half lives, half dies'' lemma, the kernel of the map
$$ \iota_* : H_1(\del Y; \Z/2) \to H_1(Y; \Z/2)$$
is one-dimensional. Since $\beta$, $\gamma$ and $\delta$ are simple closed curves on the torus $\del Y$, they represent non-trivial elements in $H_1(\del Y; \Z/2) \cong (\Z/2)^2$. By our hypotheses, their sum is zero (mod $2$) in homology. Hence, the three curves represent the three distinct non-zero elements in $H_1(\del Y; \Z/2)$. Exactly one of them must be in the kernel of $\iota_*$; without loss of generality,  assume this is $\delta$. 

We pick a point $p$ on $\delta$ (away from the other curves), and twist the Heegaard Floer groups and maps  based on boundary intersections with $p$, similarly to \cite[Section 9.3]{HolDiskTwo} but only using signs. Specifically, whenever we count some polygons in a class $\phi$ where one of the edges of the polygon is on $\Td$, we introduce a sign of  
\[ (-1)^{\#( V \cap \del_{\delta} \phi)}, \]
where $\del_{\delta}\phi$ is the part of the boundary of $\phi$ on $\Td$, and $V= \delta_1 \times \cdots \times \delta_{g-1} \times \{p\} \subset \Sym^g(\Td)$. 

At the level of the Heegaard Floer groups, this changs $\HFp(\Ta, \Td)$ into  $\HFptw(\Ta, \Td; A)$ where $A$ is a module over $\Z[H^1(Y; \Z)]$ as in Section~\ref{sec:twistedHF}. Precisely, $A$ is $\Z$ as an abelian group, and the action of an element $c \in H^1(Y; \Z)$ is given by $(-1)^{c([\delta])}$. Since $[\delta]=0 \in  H_1(Y; \Z)$, the action is actually trivial, and we obtain the same group $\HFp(\Ta, \Td)$. The same logic applies to the groups $\HFp(\Tg, \Td)$ and $\HFp(\Td, \Tb)$, so we can still keep the same elements $\Theta_{\gamma, \delta}$, $\Theta_{\delta, \beta}$ as before.

The groups $\HFp(\Ta, \Tb)$ and $\HFp(\Tb, \Tg)$ are also unchanged, as they do not involve the Lagrangian $\Td$. The only difference is with regard to the maps $F_{\Ta, \Tg, \Td}$ and $F_{\Ta, \Td, \Tb}$, which are now twisted, and produce new maps $f_2$, $f_3$ in \eqref{eq:exactp}. As in Section~\ref{sec:twisted}, the twisted maps still satisfy the $A_{\infty}$ polygon relations. Checking that double compositions such as $f_2 \circ f_1$ are zero in \eqref{eq:exactp} reduces to counting triangles between $\Tb$, $\Tg$ and $\Td$, but now the two triangles in a pair with the same $n_2$ come with the opposite signs, because one of the triangles has   $p$ with odd multiplicity on its boundary. 

It follows that with the new definitions, the sequence \eqref{eq:exactp} is a chain complex. To check exactness, either of the two proofs in the literature applies with no significant change. The proof in \cite[Section 9]{HolDiskTwo} uses energy filtrations, while the proof in \cite{BrCov} (stated there for $\HFhat$, but applicable to $\HFp$) uses that a count of holomorphic quadrilaterals is $\pm 1$; whether this count is $+1$ or $-1$ is not relevant for the argument.
\end{proof}

%In our situation, we have a 3-manifold $Y$ with torus boundary, and three fillings $Y_\beta$, $Y_\gamma$, $Y_\delta$. Looking at the fact that the kernel of $H_1(\del Y, Z/2) \to H_1(Y, \Z/2)$ is one dimensional, generated by a class $\mu$. Exactly one of $\beta, \gamma, \delta$ is $\mu$ (say $\delta$), so the other two become $0$ in their respective fillings. (The knots are null-homologous mod 2.) e.g. $[K] \in H_1(Y_\beta)=H_1(Y)/\beta$ is given by anything perpendicular to $\beta$, e.g. $\mu$, which is 0.

%The usual exact triangle from symplectic geometry (using bounding framings) gives an exact triangle between $HF(Y_{\beta})$, $HF(Y_{\gamma})$ and $HF(Y_{\delta}, \text{twisted})$, e.g. $Z \to Z \to Z/2$, which is not the usual one. The OS exact triangle is different, it appears when we twist $\beta$ and $\delta$, or equivalently $\gamma$ and $\delta$. Then we get the Lie group framing on $\delta$, as well as on one of the other two; the other has the bounding framing but it doesn't matter, since we twist along something null-homologous.

%To see that it doesn't matter which one we twist ($\beta$ or $\gamma$) in addition to $\delta$, in terms of the maps, note that both come from the same twist ($\Z$-local system) on $M=\Sym^g(\Sigma)$; a map $\pi(M) = H_1(M) \to \Z/2$ taking $\delta$ to $0$ and $\beta$ and $\gamma$ to $1$. (note that $\delta$ is twisted but so is $\alpha$, so it's like doing no twist on them.) 
%This gives an automorphism of the Fukaya category by twisting on the ambient $M$, and the Seidel-Mak-Wu exact triangle gives a new exact triangle. 

\begin{remark}
In monopole Floer homology, the analogous surgery exact triangle was proved over $\F_2$ in \cite{KMOS}. An extension to coefficients in $\Z[i]$ is developed in \cite{Freeman}, which is based on twisted coefficients. It is argued there that the same class of twistings (more limited than in Heegaard Floer theory) cannot produce an extension to $\Z$. 
\end{remark}

%%% Local Variables:
%%% mode: LaTeX
%%% TeX-master: "signs"
%%% End:

\section{Other invariants}
\label{sec:other}
We now explain how to generalize the canonical orientations in Heegaard Floer homology to other related Floer homologies.

\subsection{Sutured Floer homology}
\label{sec:sutured}
Juh\'asz \cite{Juhasz} developed a Floer homology theory for balanced sutured manifolds. Let us recall a few concepts.

\begin{definition}[Definition 2.2 in \cite{Juhasz}]
A {\em balanced sutured manifold} $(M, \gamma)$ is a compact oriented $3$-manifold $M$ together with a set $\gamma \subset \del M$ of pairwise disjoint annuli, and a decomposition of its complement $ \partial M \setminus \Int(\gamma)$ into two pieces $R_+$ and $R_-$, with the following properties:
\begin{itemize}
\item Each annulus in $\gamma$ has one boundary component in $R_+$ and one in $R_-$;
\item 
 The middle curve $S^1 \times \{0\} \subset S^1 \times [-1,1]$ in each annulus in $\gamma$ is called a {\em suture}, and is oriented as follows. We equip $R_+$ with its orientation induced from $M$, and $R_-$ with the opposite of the orientation induced from $M$. Then, we ask that any component of $\del R_+$ or $\del R_-$, with its induced boundary orientation, represents the same homology class in the respective annulus as the suture;

\item The manifold $M$ has no closed components, and the map $\pi_0(\gamma) \to \pi_0(\del M)$ is onto;

\item The Euler characteristics of $R_+$ and $R_-$ coincide.
\end{itemize}
\end{definition}

\begin{definition}[Definition 2.13 in \cite{JTZ}]
A {\em sutured Heegaard diagram} $(\Sigma, \alphas, \betas)$ for the balanced sutured manifold $(M, \gamma)$ consists of 
an oriented surface $\Sigma \subset M$ whose (oriented) boundary is the union of sutures, and sets of attaching  closed curves 
$$\alphas =\{\alpha_1, \dots, \alpha_n\},  \ \ \betas=\{\beta_1, \dots, \beta_n\},$$
such that:
\begin{itemize}
\item The components of $\alpha$ bound disks on the negative side of $\Sigma$, and compressing these disks yields a surface isotopic to $R_-$ relative to $\gamma$;
\item
The components of $\beta$ bound curves on the positive side of $\Sigma$, and compressing these disks yields a surface isotopic to $R_+$ relative to $\gamma$.
\end{itemize}
\end{definition}

Given a sutured Heegaard diagram, one considers the Lagrangians $\Ta=\alpha_1 \times \cdots \times \alpha_n$ and $\Tb = \beta_1 \times \cdots \times \beta_n \subset \Sym^n(\Sigma)$. Their Lagrangian Floer homology (taken over $\F_2$) is taken as the definition of the sutured Floer homology $\SFH(M,\gamma)$, which Juh\'asz proved to be an invariant of the balanced sutured manifold \cite{Juhasz}. 

This theory can be upgraded to $\Z$ coefficients along the same lines as we did with $\HFcirc$. We only point out the relevant differences.

First, the surface $\Sigma$ is no longer closed, and its genus $g$ does not have to be equal to the number $n$ of alpha (or beta) curves. Perutz's work \cite{Perutz} was originally phrased for closed surfaces and $g=n$. Nevertheless, we can fill in the boundary of $\Sigma$ with disks to make it closed, and only consider holomorphic disks that avoid the new disks. The construction of the symplectic form on $\Sym^n(\Sigma)$ in \cite[Section 7]{Perutz} is for any $n$. The argument about handleslides corresponding to Hamiltonian isotopies can be adapted to any $n$. 

Secondly, the surface $\Sigma$ splits $M$ not into handlebodies, but into two sutured compression bodies $U_{\alpha}$ and $U_{\beta}$. We can still define $\AA$ and $\BB$ as before, and we have the canonical isomorphisms \eqref{eq:Aspace} and \eqref{eq:Bspace}. However, $\AA$ and $\BB$ are not Lagrangian subspaces of $H_1(\Sigma; \R)$. 

In \cite{FJR}, Friedl, Juh\'asz and Rasmussen construct an absolute $\Z/2$ grading on sutured Floer homology, which is is determined by a {\em homology orientation} $\omega$ for the pair $(M, R_-)$: that is, by an orientation of the vector space 
$$H_*(M, R_-; \R) = H_1(M, R_-; \R) \oplus H_2(M, R_-; \R).$$ 
Observe that the alpha and beta curves correspond to $1$- and $2$-cells, respectively, in a relative cell decomposition for $(M, R_-)$. Thus, there is an exact sequence
\begin{equation}
\label{eq:00}
 0 \to H_2(M, R_-; \R) \to \BB \to \AA \to H_1(M, R_-; \R) \to 0.
 \end{equation}

It follows that fixing a homology orientation gives a coupled orientation on $(\AA, \BB)$, and this is how the absolute $\Z/2$-grading is defined in \cite{FJR}. For our purposes, we need more: a coupled Spin structure on the same pair $(\AA, \BB)$. To find one, apply Lemma~\ref{lem:CK} to the same exact sequence \eqref{eq:00}. 
We deduce that it suffices to specify a coupled Spin structure on $(H_1(M, R_-; \R), H_2(M, R_-; \R))$.  By  Lefschetz duality for triples, the space $H_2(M, R_-; \R)$ is dual to $H_1(M, R_+; \R)$, so we can identify them using an inner product.   This suggests the following definition. 

\begin{definition}
\label{def:homcoupled}
A {\em homological coupled Spin structure} on the sutured manifold $(M, \gamma)$ is a coupled Spin structure on the pair of vector spaces $(H_1(M, R_-; \R), H_1(M, R_+; \R))$.
\end{definition}

In general, there is a $\Z/2 \times \RP^\infty$ worth of homological coupled Spin structures, where the $\Z/2$ comes from the two choices of homology orientation. 

We have just seen that a homological coupled Spin structure $S$ determines a coupled Spin structure on the pair $(\AA, \BB)$. Given a $\Spinc$ structure on $(M, \gamma)$ as in \cite[Section 4]{Juhasz}, this allows us to define sutured Floer homology $\SFH(M,\gamma, \s, S)$ over $\Z$, in a manner similar to how we defined Heegaard Floer homology in Section~\ref{sec:3m}: We choose $\Pin$ structures $\Pa$ and $\Pb$ on the Lagrangians, define preliminary Floer complexes, and then tensor them with lines $\ell(\Pa, \Pb)$. 

\begin{proof}[Proof of Theorem~\ref{thm:main_sutured}]
The arguments are entirely similar to those for Heegaard Floer homology in Sections~\ref{sec:invariance} and \ref{sec:naturality}.  With regard to Theorem~\ref{thm:shi2}, the fact that the same loops of handleslides are sufficient in the sutured case follows from the work of Qin \cite{Qin}.
\end{proof}

\subsection{Heegaard Floer homology with multiple basepoints}
\label{sec:multiple}
Let us go back to a closed, connected, oriented three-manifold $Y$. In \cite[Section 4]{Links}, Ozsv\'ath and Szab\'o generalized the definition of Heegaard Floer homology to allow for multiple basepoints instead of just one. We follow the notation from that paper, letting $\ws =\{w_1, \dots, w_\ell\}$ be the collection of basepoints on $Y$. 

There is a notion of an $\ell$-pointed balanced Heegaard diagram for $(Y, \ws)$. This consists of a surface $\Sigma \subset Y$ of genus $g$ containing $\ws$, together with $g+\ell-1$ alpha curves and $g+\ell-1$ beta curves. We ask that:
\begin{itemize}
\item the surface $\Sigma$ splits $Y$ into two handlebodies $U_{\alpha}$ and $U_{\beta}$;
\item the alpha curves bound disks in $U_{\alpha}$, and $\Sigma \setminus \alphas$ consists of $\ell$ planar surfaces, each containing a basepoint;
\item the beta curves bound disks in $U_{\beta}$, and $\Sigma \setminus \betas$ consists of $\ell$ planar surfaces, each containing a basepoint.
\end{itemize}

In a certain sense, this set-up is a particular case of what we had in Section~\ref{sec:sutured}. Indeed, to $(Y, \ws)$ we can associate a sutured manifold $(M, \gamma)$ as follows. We pick small disjoint balls $B_1, \dots, B_\ell \subset Y$, with $B_i$ centered at the basepoint $w_i$. We let $\gamma_i$ be an annulus around the equator for $\del B_i$, and let $R_{+,i}$ and $R_{-.i}$ be the two components of $B_i \setminus \gamma_i$. We let 
$$M = Y \setminus (B_1 \cup \dots \cup B_\ell),$$
with $\gamma$ the union of $\gamma_i$'s, and $R_+$ resp. $R_-$ the union of all $R_{+,i}$ resp. $R_{-.i}$. (See Example 2.3 in \cite{Juhasz}.) Furthermore, from an $\ell$-pointed balanced Heegaard diagram for $Y$ we get a sutured diagram for $(M,\gamma)$ by deleting disks around each $w_i$.

Let $\s$ be a $\Spinc$ structure on $Y$. In \cite[Section 4]{Links}, Ozsv\'ath and Szab\'o define a version of Heegaard Floer homology $\HFm(Y, \ws, \s)$ with coefficients in the ring $\F_2[U_1, \dots, U_{\ell}]$. By setting $U_1 =\dots = U_{\ell}=0$ at the chain level, and then taking homology, they also get a hat version $\HFhat(Y, \ws, \s)$ over $\F_2$. Let us upgrade their constructions to $\Z$ instead of $\F_2$.

In the case of $\HFhat(Y, \ws, \s)$, this is just the sutured Floer homology $\SFH(M, \gamma, \s, S)$ which we constructed in Section~\ref{sec:sutured}, with a dependence on the homological coupled Spin structure $S$. In the case at hand, observe that both $H_1(M, R_-; \R)$ and $H_1(M, R_+; \R)$ are canonically identified with $H_1(Y, \ws; \R)$, by collapsing a hemisphere onto the center $w_i$ of the ball $B_i$. By Lemma~\ref{lem:EE}, there is a canonical coupled Spin structure on the pair 
$$ (H_1(Y, \ws; \R), H_1(Y, \ws; \R)).$$
This gives a canonical homological coupled Spin structure, and we use it to define $\HFhat(Y, \ws, \s)$ without any additional dependence. 

The homological coupled Spin structure gives a coupled Spin structure on the pair $(\AA, \BB)$, which we can also use to define the minus Floer homology $\HFm$, much as we did in the singly based case in Section~\ref{sec:3m}. When counting holomorphic strips in a class $\phi$, we keep track of their intersection $n_{w_i}(\phi)$ with $\{w_i\} \times \Sym^{g+\ell-2}(\Sigma)$ by a factor of $U_i^{n_{w_i}(\phi)}$. The one main difference is that we now have boundary degenerations (disk bubbles), of the kind discussed in Section~\ref{sec:bubbles}. In \cite[Theorem 5.5]{Links}, Ozsv\'ath and Szab\'o prove that for $\ell > 1$ the count of such bubbles with boundary on $\Ta$ is $1$ (mod $2$); and the same is true for those with boundary on $\Tb$. Here is the refinement of that result with $\Z$ coefficients.

\begin{proposition}
\label{prop:alphacount}
Let $\Sigma$ be a surface of genus $g$ equipped with a set $\alphas$ of $g+\ell-1$ attaching circles for a handlebody, where $\ell > 1$. Let $M=\Sym^{g+\ell-1}(\Sigma)$ and $\Ta \subset M$ be the product of the alpha curves, equipped with a Lie group Pin structure. Suppose $\x \in \Ta$ and $\phi\in \pi_2(M, \Ta)$ is a relative homotopy class of disks as in \eqref{eq:u}, with index $\mu(\phi)=2$ and such that the domain of $\phi$ on $\Sigma$ has only non-negative coefficients. Then the signed count of holomorphic disks in the class $\phi$ is $\# \cNhat(\phi)=1$.
\end{proposition}

\begin{proof}
As noted in the proof of \cite[Theorem 5.5]{Links}, the conditions in the theorem imply that $\phi$ must correspond to a domain $D$ on $\Sigma$ which is one of the connected components of $\Sigma \setminus \alphas$. This is a planar surface with boundary, and by de-stabilization and gluing arguments one can reduce the problem to the case where $D$ is a disk with one marked point. In our setting, we can choose orientations and orderings on the alpha curves (as in Definition~\ref{def:chooseab}) to view $\Ta$ as the product of $g+\ell-1$ circles, each of which equipped with the Lie group Spin structure. The arguments in the proof over $\F_2$ go through over $\Z$, and we are left with the moduli space for a disk $D$ whose boundary has the Lie group Spin structure. (In fact, once we reduced to a single circle we only care about Pin.) The moduli space $\cN(\phi)$ is just $\Aut(D^2, 1)$, and the quotient $\cNhat(\phi)$ is a point. The fact that this point is counted with a positive sign can be read off the proof of Lemma 11.17 in \cite{SeidelBook}, using the fact that the Pin structure on $\del D$ is the Lie group one (or ``trivial'', in the terminology of  \cite{SeidelBook}).
\end{proof} 

Of course, Proposition~\ref{prop:alphacount} applies equally well to the beta curves. When considering $\del^2$ for the minus Floer complex, we get contributions from both the alpha and the beta degenerations. By Proposition~\ref{prop:bubbly}, the beta degenerations actually come with negative sign. Overall, we get $\ell$ contributions of $+1$ from the components of $\Sigma \setminus \alphas$, and $\ell$ contributions of $-1$ from the components of $\Sigma \setminus \betas$. Each component contains a single basepoint $w_i$, so it gets multiplied by a factor $U_i$. The two contributions from the domains containing $w_i$ cancel out, implying that $\del^2=0$. We obtain a well-defined homology $\HFm(Y, \ws, \s)$, with coefficients in $\Z[U_1, \dots, U_{\ell}]$. 

\begin{proposition}
\label{prop:multiple}
Let $Y$ be a closed, connected, oriented $3$-manifold equipped with a set of basepoints $\ws \in Y$ and a $\Spinc$ structure $\s$. Then, the Heegaard Floer homologies $\HFhat$ and $\HFm$ are natural invariants of the triple $(Y, \ws, \s)$.
\end{proposition}

\begin{proof}
Similar to that of Theorems~\ref{thm:main} and \ref{thm:main_sutured}. 
\end{proof}

In \cite[Theorem 4.4]{Links}, Ozsv\'ath and Szab\'o also investigate how the Heegaard Floer homology depends on the basepoints, up to (non-canonical) isomorphism. Here is the analogue of their result for $\Z$ coefficients.

\begin{proposition}
\label{prop:noncan}
Let $\ws=\{w_1, \dots, w_\ell\} \subset Y$. Then, the actions of multiplication by $U_i$ on $ \HFm(Y, \ws, \s)$ for $i=1, \dots, \ell$ are the same, giving it the structure of a $\Z[U]$-module (where $U$ is any $U_i$). Moreover, we have an isomorphism of $\Z[U]$-modules
$$ \HFm(Y, \ws, \s) \cong \HFm(Y, \s).$$
We also have an isomorphism of abelian groups
$$ \HFhat(Y, \ws, \s) \cong \HFhat(Y, \s) \otimes H_*(T^{\ell-1}).$$
\end{proposition}

\begin{proof}
The proof of \cite[Theorem 4.4]{Links} proceeds by using Heegaard moves to reduce to the case of a simple index zero/three stabilization. There, we have a Heegaard diagram $(\Sigma, \alphas, \betas, \ws)$ and its connected sum 
$$(\Sigma', \alphas', \betas', \ws') = (\Sigma, \alphas, \betas, \ws) \# (S, \alpha, \beta, v_1, v_2),$$
where $S$ is a sphere with two great circles $\alpha$ and $\beta$ intersecting at two points $x$ and $y$, as in Figure~\ref{fig:03}. The connected sum is taken at a basepoint $w_1 \in \ws$ on $\Sigma$, which gets identified with $v_1$ on $S$. The point $v_2$ becomes a new basepoint $w_{\ell+1}$, so that $\ws'=\ws \cup \{w_{\ell+1}\}$.
\begin{figure}
{
\fontsize{10pt}{11pt}\selectfont
   \def\svgwidth{2in}
   %% Creator: Inkscape 1.3.2 (091e20e, 2023-11-25), www.inkscape.org
%% Accompanies image file '03.pdf' (pdf, eps, ps)
\begingroup%
  \makeatletter%
  \providecommand\color[2][]{%
    \errmessage{(Inkscape) Color is used for the text in Inkscape, but the package 'color.sty' is not loaded}%
    \renewcommand\color[2][]{}%
  }%
  \providecommand\transparent[1]{%
    \errmessage{(Inkscape) Transparency is used (non-zero) for the text in Inkscape, but the package 'transparent.sty' is not loaded}%
    \renewcommand\transparent[1]{}%
  }%
  \providecommand\rotatebox[2]{#2}%
  \newcommand*\fsize{\dimexpr\f@size pt\relax}%
  \newcommand*\lineheight[1]{\fontsize{\fsize}{#1\fsize}\selectfont}%
  \ifx\svgwidth\undefined%
    \setlength{\unitlength}{182.89189136bp}%
    \ifx\svgscale\undefined%
      \relax%
    \else%
      \setlength{\unitlength}{\unitlength * \real{\svgscale}}%
    \fi%
  \else%
    \setlength{\unitlength}{\svgwidth}%
  \fi%
  \global\let\svgwidth\undefined%
  \global\let\svgscale\undefined%
  \makeatother%
  \begin{picture}(1,0.60277882)%
    \lineheight{1}%
    \setlength\tabcolsep{0pt}%
    \put(0,0){\includegraphics[width=\unitlength,page=1]{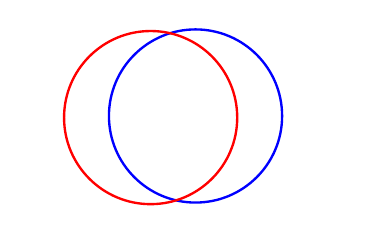}}%
    \put(0.72494647,0.45332615){\color[rgb]{0,0,1}\makebox(0,0)[lt]{\lineheight{1.25}\smash{\begin{tabular}[t]{l}$\beta$\end{tabular}}}}%
    \put(0.16839423,0.45332615){\color[rgb]{1,0,0}\makebox(0,0)[lt]{\lineheight{1.25}\smash{\begin{tabular}[t]{l}$\alpha$\end{tabular}}}}%
    \put(0,0){\includegraphics[width=\unitlength,page=2]{03.pdf}}%
    \put(0.4298868,0.2708797){\makebox(0,0)[lt]{\lineheight{1.25}\smash{\begin{tabular}[t]{l}$v_2$\end{tabular}}}}%
    \put(0.438929,0.54808048){\makebox(0,0)[lt]{\lineheight{1.25}\smash{\begin{tabular}[t]{l}$y$\end{tabular}}}}%
    \put(0.42851535,0.01252923){\makebox(0,0)[lt]{\lineheight{1.25}\smash{\begin{tabular}[t]{l}$x$\end{tabular}}}}%
    \put(0,0){\includegraphics[width=\unitlength,page=3]{03.pdf}}%
    \put(-0.00537961,0.20725628){\makebox(0,0)[lt]{\lineheight{1.25}\smash{\begin{tabular}[t]{l}$v_1$\end{tabular}}}}%
  \end{picture}%
\endgroup%

}
\caption{An index zero/three stabilization.}
\label{fig:03}
\end{figure}
Degeneration arguments show that we have an isomorphism of chain complexes
$$ \CFm(\Sigma', \alphas', \betas', \ws') \cong \CFm(\Sigma, \alphas, \betas, \ws) \otimes \CFm(S, \alpha, \beta, v_1, v_2)$$
where the tensor product is over $\F[U_1]$ and the second factor is the mapping cone of $U_1 - U_{\ell +1}$. In our setting, after we equip each curve with the Lie group Spin structure, the same arguments work over $\Z$. Indeed, $\CFm(S, \alpha, \beta, v_1, v_2)$ is generated by $x$ and $y$, and in its differential the two empty bigons from $y$ to $x$ come with opposite signs (as we have seen in Example~\ref{ex:circles}). The two bigons from $x$ to $y$ also come with opposite signs, and contribute $U_1 - U_{\ell +1}$. Thus, $  \CFm(\Sigma', \alphas', \betas', \ws')$ is isomorphic to the mapping cone
$$ \CFm(\Sigma, \alphas, \betas, \ws)[U_{\ell+1}] \xrightarrow{U_1 - U_{\ell +1}}  \CFm(\Sigma, \alphas, \betas, \ws)[U_{\ell+1}].$$
From here we deduce that the $\F[U_1]$-modules $\HFm$ for the two diagrams are isomorphic, whereas the hat versions $\HFhat$ differ by tensoring with a $H_*(S^1)$ factor. This suffices to establish the desired results. (To see that the $U_i$ actions on $\HFm$ are the same, we relate them to a singly based diagram by Heegaard moves, using any of the $U_i$ as the variable in the coefficient ring.) 
\end{proof}

\subsection{Link Floer homology} 
Link Floer homology is an invariant of links in three-manifolds developed by Ozsv\'ath and Szab\'o in \cite{Links}, with coefficients in $\F_2$. It generalizes the construction for knots in \cite{OS-knots}, \cite{RasmussenThesis}. There are several versions of link Floer homology. We start by reviewing $\HFLhat$ and $\HFLm$, borrowing some terminology and notation from \cite{Links}, \cite{MOS}, and \cite{ZemkeHFK}. 

\begin{definition}
Let $Y$ be a closed oriented $3$-manifold. A {\em multi-based link} in $Y$ is a triple $\Link=(L, \ws, \zs)$ where $L \subset Y$ is an embedded, oriented link, and $\ws$ and $\zs$ are disjoint collections of basepoints on $L$ with the following properties: 
\begin{itemize}
\item every component of $L$ has at least two basepoints, and
\item as we traverse any component, the basepoints alternate between those in $\ws$ and those in $\zs$.
\end{itemize}
\end{definition}

\begin{definition}
Fix $(Y , \Link)$ as above. A {\em multi-based Heegaard diagram} $(\Sigma, \alphas, \betas, \ws, \zs)$ representing $(Y, \Link)$ is a Heegaard diagram for $Y$ with $\ws, \zs \subset \Sigma \setminus (\alphas \cup \betas)$, such that after we attach disks to the alpha and beta curves in their respective handlebody, the handlebodies are split into balls; each ball should have exactly one $w$ and one $z$ basepoint on its boundary, and as we join these two basepoints by an arc inside the ball, the union of these arcs should be the link $L$. The orientation of $L$ should be such that the arcs go from the $w$ to the $z$ basepoints inside the alpha handlebody.
\end{definition}

Let $(Y, \Link)$ be a $3$-manifold with a multi-based link, and $\bs$ be a $\Spinc$ structure on $Y \setminus \text{nbhd}(L)$ relative its boundary, as in \cite[Section 3.2]{Links}. Suppose we have $m$ basepoints of type $w$ and $m$ of type $z$. Given a Heegaard diagram $(\Sigma, \alphas, \betas, \ws, \zs)$ representing $(Y, \Link)$, we define $\HFLm(Y, \Link, \s)$ as the homology of a Floer complex over $\F_2[U_1, \dots, U_m]$, where the Lagrangians are $\Ta$ and $\Tb$ as usual, and we only count holomorphic strips in classes $\phi$ such that
$$ n_{z_1}(\phi) = n_{z_2}(\phi) = \cdots = n_{z_m}(\phi)=0.$$
Further, when counting the strips we include a factor of 
$$ U_1^{n_{w_1}(\phi)} U_2^{n_{w_2}(\phi)} \cdots U_m^{n_{w_m}(\phi)} ,$$
just as we did in Section~\ref{sec:multiple}. 

For $\HFLhat(Y, \Link, \bs)$, we set all $U_i$ to be zero in the complex above and then take homology; that is, we only count strips that do not go over any basepoints. 

Let us upgrade these constructions to $\Z$ instead of $\F_2$. To imitate what we did for $\HFcirc$ in Section~\ref{sec:3m}, we need a coupled Spin structure on $(\AA, \BB)$, where $\AA$ and $\BB$ are the spans of the alpha and beta curves, respectively. Observe that to $(Y, \Link)$ we can associate a sutured manifold $(M, \gamma)$ as follows. We let $M$ be the complement $Y \subset \text{nbhd}(L)$. Take meridians $\mu(w_i), \mu(z_i) \subset \del M$ around the basepoints. Let the sutures $\gamma$ be disjoint annuli around each of these meridians in $\del M$, and let $R_-$ and $R_+$  be the remaining parts of $\del N$ that are in $U_{\alpha}$ and $U_{\beta}$, respectively. (Compare Example 2.4 in \cite{Juhasz}.) 

We can now appeal to the results from Section~\ref{sec:sutured}. In fact, notice that $\HFLhat(Y, \Link, \bs)$ is nothing else than the sutured Floer homology of $(M, \gamma)$. A coupled Spin structure on $(\AA, \BB)$ is determined by a homological coupled Spin structure on $(M, \gamma)$. In our setting, where the sutured manifold comes from a multi-based link, the pairs $(M, R_-)$ and $(M, R_+)$ are homotopy equivalent, by an isotopy supported in a neighborhood of $\del M$ which slides each annulus from $R_-$ into the next one from $R_+$ by following the orientation of the link. Hence, the spaces $H_1(M, R_-; \R)$ and $H_1(M, R_+; \R)$ are  canonically identified, and by Lemma~\ref{lem:EE} we have a preferred homological coupled Spin structure. 

Thus, we can define $\HFLhat(Y, \Link, \bs)$ without making any additional choices. The same goes for $\HFLm(Y, \Link, \bs)$, because the spaces $\AA$ and $\BB$ are the same. Note that for these versions of link Floer homology we do not have any contributions from disk bubbles, because their domains would go through the $z_i$ basepoints and hence they are not counted. The resulting theories are natural invariants of the multi-based link.

\begin{proof}[Proof of Theorem~\ref{thm:main_links}]
Similar to those of Theorem~\ref{thm:main}, Theorem~\ref{thm:main_sutured} and Proposition~\ref{prop:multiple}.
\end{proof}

We also have an analogue of Proposition~\ref{prop:noncan}, with a similar proof.
\begin{proposition}
\label{prop:noncan-links}
Let $(Y, \Link, \bs)$ be a $3$-manifold with a multi-based link and a relative $\Spinc$ structure. Then, 
 the actions of multiplication by $U_i$ on $ \HFLm(Y, \Link, \bs)$ are the same for all variables $U_i$ corresponding to basepoints on the same component of the link.  This gives $ \HFLm(Y, \Link, \bs)$ the structure of a module over a polynomial ring with one variable for each link component. As such, $  \HFLm(Y, \Link, \bs)$ is independent of the number and position of the base points, up to (non-canonical) isomorphism.

In the case of $\HFLm(Y, \Link, \bs)$, if $\Link= (L, \ws, \zs)$ and $\Link' = (L, \ws', \zs')$ are such that $|\ws|=|\zs|=m$ and $|\ws|=|\zs|=m' \leq m$, then we have a (non-canonical) isomorphism
$$ \HFLhat(Y, \Link, \bs) \cong \HFLhat(Y, \Link', \bs) \otimes H_*(T^{m-m'}).$$
\end{proposition}

One can also consider more general versions of link Floer complexes, where we allow strips that go over the $z$ basepoints. For example, in \cite{ZemkeHFK}, Zemke works with a curved chain complex $\CFLcurved(Y, \Link, \s)$ for $\s \in \Spinc(Y)$. This is a module over $\F_2[U_1, \dots, U_m, V_1, \dots, V_m]$, and its differential counts holomorphic strips in a class $\phi$ with a coefficient of
$$ U_1^{n_{w_1}(\phi)}  \cdots U_m^{n_{w_m}(\phi)} \cdot V_1^{n_{z_1}(\phi)}  \cdots V_m^{n_{z_m}(\phi)}.$$
To do this over $\Z[U_1, \dots, U_m, V_1, \dots, V_m]$, we use our preferred homological coupled Spin structure. We now have disk bubbles as in Section~\ref{sec:multiple}, and they no longer cancel. Rather, there are positive contributions to $\del^2$ from disks on $\Ta$, and negative ones from disks on $\Tb$. We get
$$\del^2 = \left( \sum_{i=1}^m (U_i V_{\sigma(i)} - V_i U_{\tau(i)}) \right) \cdot \id,$$
where $z_{\sigma(i)}$ is the basepoint following $w_i$ as we go around the respective component of $L$ (with the given orientation), and $w_{\tau(i)}$ is the basepoint following $\tau(i)$. The complex $\CFLcurved(Y, \Link, \s)$ is curved, so we cannot take its homology (unless we have exactly two basepoints per component, in which case $\del^2=0$). Nevertheless, a variant of Theorem~\ref{thm:main_links} still applies, saying that $\CFLcurved(Y, \Link, \s)$ form a transitive system of curved chain complexes in the sense of \cite[Definition 2.15]{ZemkeHFK}.

\begin{remark}
In \cite{SarkarSigns}, Sarkar constructs $2^{\ell -1}$ versions of link Floer homology over $\Z$, where the different choices correspond to whether, for each of the $\ell$ components of the link, it is the alpha or the beta degenerations with basepoints on that component which get counted with $+1$ (versus $-1$). Our canonical version corresponds to his where all the alpha degenerations are counted with $+1$. This agrees with the signs from grid homology \cite[Definition 4.1]{MOST}. To get the other choices in \cite{SarkarSigns}, one can use twisted coefficients, or equip one or both of the Lagrangians with a Pin structure different from the Lie group one. (Compare Section~\ref{sec:twistedHF}.) 
\end{remark}

\subsection{Involutive Heegaard Floer homology}
\label{sec:HFI}
In \cite{HFI}, Hendricks and the second author define an invariant of $3$-manifolds called involutive Heegaard Floer homology. Let us review that construction. Let $\H=(\Sigma, \alphas, \betas, z)$ be a based Heegaard diagram representing a $3$-manifold $Y$. For simplicity, we focus on $\Spinc$ structures $\s$ that are self-conjugate, i.e., $\s = \bar \s$. (This is the only interesting case.) Given such an $\s$, there is a canonical conjugation isomorphism between Heegaard Floer chain complexes: 
$$\eta : \CFcirc(\H, \s) \xrightarrow{\phantom{u} \cong \phantom{u}} \CFcirc(\bH, \s)$$
where $\bH$ is the Heegaard diagram $(-\Sigma,  \betas, \alphas, z)$ and $\circ \in  \{\widehat{\phantom{u}}, +, -, \infty\}$. Furthermore, since $\H$ and $\bH$ represent the same $3$-manifold, there is a chain homotopy equivalence induced by Heegaard moves
$$\Phi(\bH, \H) : \CFcirc(\bH, \s) \xrightarrow{\phantom{u} \sim \phantom{u}} \CFcirc(\H, \s).$$
We let
$$ \iota= \Phi(\bH, \H) \circ \eta  : \CFcirc(\H, \s) \to \CFcirc(\H,  \s).$$
The involutive Heegaard Floer homology $\HFIcirc(Y, \s)$ is defined as the homology of the mapping cone 
$$\CFcirc(\H, \s) \xrightarrow{\phantom{o} Q (1+\iota) \phantom{o}} Q \ccdot \CFcirc(\H, \s) [-1], $$
where $Q$ is a formal variable and $[-1]$ is a shift in degree. 

We can repeat the same construction over $\Z$, using the Floer complexes and maps from Section~\ref{sec:3m}. 
\begin{proposition}
The isomorphism class of the involutive Heegaard Floer homology $\HFIcirc(Y, \s)$, as a graded $\Z[Q,U] /(Q^2)$-module, is an invariant of the pair $(Y, \s)$.
\end{proposition}

\begin{proof}
This is similar to the proof of the analogous Theorem 1.1 in \cite{HFI}. The key input is that the map $\Phi(\bH, \H)$ is invariant up to chain homotopy equivalence. This is a version of naturality at the chain level, which in our context follows from the proofs in Section~\ref{sec:naturality}. 
\end{proof}

There is also a naturality result for involutive Heegaard Floer homology (with $\F_2$ coefficients), proved in \cite{HFInatural}: It says that $\HFIcirc$ is a natural invariant of $(Y, \s, z, \xi)$, where $\xi$ is a framing of $T_zY$. While we expect the same statement with $\Z$ coefficients, proving it is beyond the scope of the current paper.

%%% Local Variables:
%%% mode: LaTeX
%%% TeX-master: "signs"
%%% End:

%\input{cobordism}
\bibliography{biblio}
\bibliographystyle{custom}

\end{document}